\documentclass[a4paper, 10pt,reqno]{amsart}

\usepackage[usenames,dvipsnames]{xcolor}

\usepackage[portrait,margin=3cm]{geometry}
\usepackage{mathrsfs}

\usepackage{amssymb,amsthm,amsmath,amsfonts,calc, amsbsy,latexsym, tabularx}

\usepackage{color}
\usepackage{graphicx}
\usepackage{bbm}
\usepackage{comment}
\usepackage{a4wide, hyperref,graphicx, bbm, eurosym, listings, float, graphicx, amsfonts, comment,lscape, epstopdf, units, algorithmic, algorithm, textcomp, amsthm, fancyhdr, url, color, soul, comment, enumerate, color}
\usepackage[textsize=small]{todonotes}
\usepackage{enumitem}

\usepackage{mathtools}
\mathtoolsset{showonlyrefs}

\definecolor{MyDarkblue}{rgb}{0,0.08,0.50}
\definecolor{Brickred}{rgb}{0.65,0.08,0}

\hypersetup{
	colorlinks=true,       
	linkcolor=MyDarkblue,          
	citecolor=Brickred,        
	filecolor=red,      
	urlcolor=cyan           
}

\newtheorem*{theorem*}{Theorem}
\newtheorem{theorem}{Theorem}[section]
\newtheorem{lemma}[theorem]{Lemma}

\newtheorem{proposition}[theorem]{Proposition}

\theoremstyle{definition}
\newtheorem{definition}[theorem]{Definition}
\newtheorem{remark}[theorem]{Remark}
\newtheorem{assumption}[theorem]{Assumption}

\renewcommand{\P}{\mathbb{P}}

\newcommand{\Pv}{\mathbb{P}}

\newcommand{\eps}{\varepsilon}

\newcommand{\cF}{\mathcal{F}}
\newcommand{\cG}{\mathcal{G}}

\newcommand{\cZ}{\mathcal{Z}}

\newcommand{\Var}{{\rm Var}}

\newcommand{\e}{{\mathrm e}}

\newcommand{\R}{\mathbb{R}}
\newcommand{\N}{\mathbb{N}}
\newcommand{\Z}{\mathbb{Z}}

\renewcommand{\emptyset}{\varnothing}

\newcommand*{\wt}{\widetilde}

\newcommand*{\be}{\begin{equation}}
\newcommand*{\ee}{\end{equation}}
\newcommand*{\ba}{\begin{aligned}}
	\newcommand*{\ea}{\end{aligned}}
\newcommand*{\barr}{\begin{array}{c}}
	\newcommand*{\earr}{\end{array}}
\def \toinp    {\buildrel {\Pv}\over{\longrightarrow}}
\def \toindis  {\buildrel {d}\over{\longrightarrow}}
\def \toas     {\buildrel {a.s.}\over{\longrightarrow}}

\newcommand*{\ind}{\mathbbm{1}}
\def\namedlabel#1#2{\begingroup
	#2%
	\def\@currentlabel{#2}%
	\phantomsection\label{#1}\endgroup
}

\newcommand{\bes}{\begin{equation*}}
\newcommand{\ees}{\end{equation*}}

\renewcommand{\P}[1]{\mathbb{P}\!\left(#1\right)}
\newcommand{\E}[1]{\mathbb{E}\!\left[#1\right]}

\newcommand{\F}{\mathcal{F}}
\newcommand{\G}{\mathcal{G}}
\newcommand{\Zm}{\mathcal{Z}}
\renewcommand{\N}{\mathbb{N}}

\newcommand{\Fb}{\bar{\F}_n}
\newcommand{\Gamnk}{\Gamma_n^{(k)}}

\newcommand{\dzn}[1]{\Delta \Zm_n(#1)}
\newcommand{\Ef}[2]{\mathbb{E}_\F#1[#2#1]}
\newcommand{\Pf}[1]{\mathbb{P}_\F\!\left(#1\right)}
\newcommand{\I}{\mathbb{I}}

\renewcommand{\d}{\mathrm{d}}
\DeclareMathOperator*{\argmax}{arg\,max}

\newcommand{\zni}{\Zm_n(i)}

\newcommand{\inn}{i\in[n]}
\newcommand{\intk}{[2^{-(k+1)},2^{-k})}

\numberwithin{equation}{section}

\begin{document}
	\title{A phase transition for preferential attachment models with additive fitness}
	
	\date{\today}
	\author[Lodewijks]{Bas Lodewijks}
	\author[Ortgiese]{Marcel Ortgiese}
	\address{Department of Mathematical Sciences,
		University of Bath,
		Claverton Down,
		Bath,
		BA2 7AY,
		United Kingdom.}
	\email{b.lodewijks@bath.ac.uk, ma2mo@bath.ac.uk}
	
	\maketitle 
	
	\begin{abstract}Preferential attachment models form a popular class of growing networks, where
incoming vertices are preferably connected to vertices with high degree. 
We consider a variant of this process, where vertices are equipped with 
a random initial fitness representing initial inhomogeneities among vertices and the fitness influences the attractiveness of a vertex in an additive way.
We consider a heavy-tailed fitness distribution and show that the model 
exhibits a phase transition depending on the tail exponent of the fitness distribution.
In the weak disorder regime, one of the old vertices has maximal degree irrespective of fitness, while for strong disorder the vertex with maximal degree has to satisfy the right balance between fitness and 
age.
Our methods use martingale methods to show concentration of degree evolutions as well as extreme value theory to control the fitness landscape.\end{abstract}
	
\section{Introduction}

A distinctive feature of real-world networks is their inhomogeneity, characterized in particular through the presence of hubs. These are nodes with a number of connections that greatly exceeds the average and thus have a great impact on the overall network topology. The existence of hubs in a network is closely linked to the \emph{scale-free property}, that is, the proportion of nodes in the network with degree (number of connections) $k$ scales as a power law $k^{-\tau}$ for some $\tau>1$.
		 
Preferential attachment models, as popularized by Barab\'asi and Albert~\cite{BarAlbRek99}, form
a class of random graphs that shows this behaviour `naturally',  that is, 
as a result of the dynamics and not because it is imposed otherwise,  see also~\cite{BolRioSpenTus01} for a first mathematical derivation of this fact. In these evolving random graph models new vertices are introduced to the network over time and they connect to earlier introduced vertices with a probability proportional to their degree. This leads to the so-called \emph{rich-get-richer} effect, which means that vertices with a high degree are more likely to increase their degree. It is exactly this effect that yields the power-law degree distributions and the existence of hubs in the graph.
	
The study of the emergence of hubs in random graph models such as the preferential model is often focused on the behaviour of the maximum degree in the graph. M\'ori first showed that for the Barab\'asi-Albert model the maximum degree is of the same order as the degree of the first vertex \cite{Mori05}, which was later generalised by Athreya \emph{et al}.\ to affine preferential attachment models (with random out-degree) and to a larger class of preferential attachment models by Bhamidi in \cite{Athr08} and \cite{Bham07}, respectively. A consequence of the   way in which preferential attachment graphs evolve,
is that the rich-get-richer effect should really  be interpreted as an \emph{old-get-richer} effect: it is the old vertices, who are introduced at the beginning of the evolution of the graph, that are able to attract the most connections \cite{Hof16}.

However, when compared to real-life networks, it is clearly desirable to have a model where younger vertices can compete with the old ones. One way to achieve this is by assigning to each vertex a 
random fitness representing its intrinsic attractiveness and then to let the connection probability of newly incoming vertex be proportional to either the product of the fitness and degree or the sum. These two models were introduced by Barab\'asi and Bianconi in \cite{BiaBar01} and Erg\"un and Rodgers in \cite{ErgRod02}, respectively.
	
Most previous results on preferential attachment models with fitness  deal with the multiplicative case for bounded fitness. One of the reasons is that under certain conditions on the fitness distribution, these models exhibit the phenomenon of condensation, where a positive proportion of incoming vertices connects to vertices with fitness closer and closer to the maximal fitness in the system. 
This phenomenon was first shown in the mathematical literature in~\cite{BorChaDasRoch07}, later extended in~\cite{DerOrt14} for a wide range of models, by looking at the empirical fitness and degree distribution. A full dynamic description of the condensation is a challenging problem, however see~\cite{Der16} for a very detailed analysis in a slightly modified model. \cite{DerMaiMor17} considers a continuous-time embedding of the process into a reinforced branching process, which allows them to control the maximal degree (in the continuous-time setting), which in the non-condensation case can be translated back to the random graph model. Also, under certain assumption on the fitness distribution, they show that condensation is non-extensive in the sense that there is not a single vertex that acquires a positive fraction of the incoming edges. These results are extended by \cite{Mailler2019} to a larger class of (bounded) fitness distributions (as a special case of a more general set-up).

Here, we consider the model with additive fitness, where a vertex is chosen with probability 
proportional to the sum of its degree and its intrinsic fitness. To best of our knowledge the only mathematical result have been~\cite{Bham07} and~\cite{Sen19}, who confirmed the non-rigorous results in~\cite{ErgRod02}. \cite{Bham07} showed that when the fitness is bounded, the degree distribution follows a power law with the same exponent as for the model with an additive constant equal to the expected value of the fitness. 
Moreover, \cite{Bham07} gives the asymptotics for the maximum degree and shows that it agrees again with the asymptotics for the model with additive constant. \cite{Sen19} considers the case of a deterministic additive sequence and shows that there is an equivalence between the preferential attachment (tree) model and a weighted recursive tree. From this, the author deduces $\ell^p$-convergence of the renormalized degree sequence under a growth condition on the additive sequence. Furthermore, he considers geometric properties of the weighted recursive trees.
Somewhat related  is a model of preferential attachment with random (possibly heavy-tailed) initial degree, for which~\cite{DeijEskHofHoog09} show convergence of empirical fitness distributions, but the structure of these networks is very different from the additive fitness case due to large out-degrees.

In our work we consider the case of unbounded fitness and show that when the fitness distribution follows a power law,  a more complex  phase diagram arises. Our first result shows convergence for the empirical degree and fitness distribution. From this we can in particular deduce that if the fitness distribution is sufficiently light-tailed, then we are in a \emph{weak disorder regime}, where the same result as in~\cite{Bham07} still holds for both tail exponent of the degree distribution and 
the asymptotics of the maximum. However, if the tail exponent of the fitness distribution is sufficiently small (but so that the fitness still has a finite first moment), then we are in a \emph{strong disorder regime}, where the tail exponent of the degree distribution is the same as for the fitness distribution. Moreover, the maximal degree grows of the same order as the largest fitness in the system. However, the
vertex that maximizes the degree has to satisfy a delicate balance between arriving early and having a large fitness. In the limit this competition is expressed as an optimization of a functional of a Poisson point process. 

Finally, we can also consider the \emph{extreme disorder regime} when the fitness does not have a finite first moment. In that case, we show that a uniformly selected vertex does not connect to any incoming vertices with high probability. Moreover, the maximal degree now scales as order $n$ and the maximising vertex again satisfies the right balance between arriving early and large fitness. We note that our results for the degree distribution improve on those by Erg\"un and Rodgers  \cite{ErgRod02}, where these different regimes are overlooked and only the weak disorder regime is covered.

Our proof for the empirical degree/fitness distributions uses a stochastic approximation argument, which was also used in \cite{DerOrt14} for the multiplicative case. The analysis of the maximal degree is split into two steps: First we show concentration of the degrees when compared to the expected degree (conditional on the fitness values) adapting the martingale arguments of M\'ori~\cite{Mori05} (see also~\cite{Hof16} for an exposition with more general attachment rules). For the weak disorder case, similar arguments as in~\cite{Hof16} are sufficient to control the maximal degree. However, in the strong
and extreme disorder case, we have to control the conditional expectation of the degrees, which are a function of the fitness only. We then show that these functionals simplify and converge to a functional of a Poisson point process, so that with the concentration we can deduce convergence of the maximal degree. Finally, our analysis is robust and covers essentially three variants of preferential attachment models: a model with possibly random out-degree as in~\cite{DerMor09} (and at most one edge between vertices) and two variations where the out-degree of each new vertex is fixed and then the connection probabilities are either updated after each edge is drawn or are kept fixed.
 		
{\bf Notation.} Throughout we will use the following notation. We let $\N = \{1,2,3,\ldots\}$ be the natural numbers, we write $\N_0 = \{0,1, 2,\ldots\}$ if we want to include $0$ and let $[n]:=\{1,\ldots,n\}$. Moreover, for any sequence $a_n$ and $b_n$ of positive real numbers, we say $a_n = \Theta(b_n)$ if there exists a constant $C> 0$ such that $a_n \leq C b_n$ and $b_n \leq C a_n$. 
Moreover, we say $a_n \sim b_n$ if $\lim_{n \rightarrow \infty} \frac{a_n}{b_n} = 1$. Also, we use the conditional probability measure $\mathbb{P}_\F(\cdot):=\P{\cdot\,|\,(\F_i)_{i\in\N}}$ and expectation $\mathbb{E}_\F[\cdot]:=\E{\cdot\,|\,(\F_i)_{i\in\N}}$.
	
\section{Definitions and main results}
	
The preferential attachment model is an evolving random graph model, where vertices are added to the graph consecutively and then connected to older vertices. We denote by $\G_n$ the resulting directed graph at the stage when the vertex set is $[n]$. Moreover, 
we take  edges to be directed from the vertex with high index to the one with lower index.
Throughout, we will use the following notation,
\be
\Zm_n(i):=\mbox{in-degree of vertex } i \mbox{ in }\G_n. 
\ee
We now introduce three different preferential attachment with fitness models (PAF), the first one which allows for a random out-degree in the spirit of Dereich and M\"orters~\cite{DerMor09}, the second one where the out-degree of a new vertex is fixed and we connect
edges while keeping the degrees fixed and the last one 
 with a fixed out-degree, but where we update degrees in between connections (where the later is the 
fitness modification of the a model  closer to~\cite{BolRioSpenTus01}).

\begin{definition}[Preferential attachment with fitness]\label{def:paf}
	Let $(\F_i)_{i\geq 1}$ be a sequence of i.i.d.\ copies of a random variable $\F$ taking values in $(0,\infty)$
	with distribution $\mu$. For any $n \in \N$, define
	\[ S_n := \sum_{i=1}^n \cF_i . \]
Let $n_0, m_0 \in \N$.
We say that a sequence of  random graphs $(\mathcal{G}_n)_{n \geq n_0}$
is a \emph{preferential attachment model with (additive) fitness} if 
$\mathcal{G}_n$ is a directed and weighted graph on the vertex set $[n]$ 
with edges directed from larger to smaller indices.
Moreover, we assume that $\G_{n_0}$ has $m_0$ edges and we assign
fitness  values $\F_1,\F_2,\ldots,\F_{n_0}$ to the vertices $1,2, \ldots, n_0$ respectively. 

To obtain $\cG_{n+1}$ from $\cG_n$ for some $n \geq n_0$, add vertex $n+1$ to the vertex set and 
attach fitness $\F_{n+1}$ to $n+1$.
Furthermore, we assume that the updating rules satisfies one of the following three assumptions for some fixed $m \in \N$: 

\begin{enumerate}[labelindent = 1cm, leftmargin = 2.2cm]
\item[{\bf (PAFRO)}]
 \emph{Preferential attachment with fitness and random out-degree}. Here $m =1$ and conditionally on $\mathcal{G}_n$,   vertex $n+1$ is connected to each vertex in $[n]$
by at most one edge and the probability to connect to a given $i\in[n]$ is
\begin{equation} \label{eq:PAFRO}
			\frac{\zni+\F_i}{m_0+(n-n_0)+S_n}.
			\end{equation}
			Furthermore, conditionally on $\G_n$ the degree increments $(\Delta\zni:=\Zm_{n+1}(i)-\zni, i \in [n])$ are pairwise non-positively correlated. \\
\item[{\bf (PAFFD)}] \emph{Preferential attachment with fitness and fixed degree}. To vertex $n+1$ we assign $m$ half-edges. Conditionally on $\G_n$, connect each half-edge independently to some $i\in[n]$ with probability
			\be 
			\frac{\zni+\F_i}{m_0+m(n-n_0)+S_n}.
			\ee  
\item[{\bf (PAFUD)}] \emph{Preferential attachment with fitness and updating degree}. To vertex $n+1$ we assign $m$ half-edges.  Let $\Zm_{n,j}(i)$ denote the in-degree of vertex $i$ when $n+1$ has attached $j$ of its half-edges, $j=1,\ldots,m$. For $j=1,\ldots,m$, conditionally on the graph of size $n$ including the first $j-1$ half-edges from $n+1$, connect the $j^{\mathrm{th}}$ half-edge to $i\in[n]$ with probability
			\be 
			\frac{\Zm_{n,j-1}(i)+\F_i}{m_0+m(n-n_0)+(j-1)+S_n}.
			\ee 
\end{enumerate}

\end{definition}

\begin{remark}
The quantity in~\eqref{eq:PAFRO} is always less than $1$, since $\sum_{i=1}^{n_0}\cZ_{n_0}(i) = m_0$ and at each step $\cZ_n(i)$ increases by at most  one. Note also that for the PAFRO  assumption, the exact distribution of $(\Delta \cZ_n(i), i \in [n])$ is not specified. For example, for $m = 1$, the PAFFD and the PAFUD model 
are identical and both satisfy PAFRO. Another possibility is to consider a model with a random out-degree, where $(\Delta \cZ_n(i), i \in [n])$ is a vector of independent Bernoulli variables with success probability as given in~\eqref{eq:PAFRO}.
\end{remark}

We have defined our random graph model for an arbitrary fitness distribution. 
However, for the analysis the most interesting case occurs when we are dealing with heavy-tailed distributions. 
In this case the fitness can have a significant effect on the behaviour of the system as a whole, whereas the `fitness effect' is smoothed out when its tail behaviour is too light. In the latter case, one sees no differences in the mean-field behaviour when changing from a deterministic, fixed fitness to random i.i.d.\ fitness values. Therefore, in the following, we will frequently consider the following assumption:

\begin{assumption}\label{Ass:powerlaw}
	The fitness distribution is a power law with exponent $\alpha>1$, i.e.
	\be 
	\P{\F\geq x}= \mu(x,\infty) = \ell(x)x^{-(\alpha-1)},\quad \mbox{for } x > 0, 
	\ee 
	where $\ell$ is a slowly-varying function at infinity, i.e.\  for all $ c>0 \ \lim_{x\rightarrow\infty}\ell(cx)/\ell(x)=1$.
\end{assumption} 

We continue by stating our first main result. We define the following measures,
\be\ba\label{eq:gammas}
\Gamma_n&:=\frac{1}{n}\sum_{i=1}^n \Zm_n(i)\delta_{\F_i},\qquad
\Gamnk:=\frac{1}{n}\sum_{i=1}^n \mathbbm{1}_{\{\Zm_n(i)=k\}}\delta_{\F_i},\qquad p_n(k):=\Gamnk([0,\infty)),
\ea\ee 
which correspond to the the empirical fitness distribution of a vertex sampled with weight given by its in-degree, then the joint empirical fitness-in-degree distribution and finally the empirical degree distribution. 

\begin{theorem}[Degree distributions in PAF models]\label{Thrm:degree}
	Consider the three PAF models as in Definition~\ref{def:paf} and suppose the fitness satisfies $\E{\F} < \infty$.  Let $\theta_m:=1+\E{\F}/m$. Then, almost surely, for any $k \in \N_0$, as $n \rightarrow \infty$,
\be\label{eq:gammaconv}
\Gamma_n \rightarrow \Gamma, \quad 	\Gamnk \rightarrow \Gamma^{(k)}, \quad \mbox{ and }\quad
p_n(k) \rightarrow p(k), 
\ee
where the first two statements hold with respect to the weak$^*$ topology and the limits 
are given as 
\begin{equation}
 \label{eq:gamma}
	\Gamma(\d x)  =\frac{x}{\theta_m-1}\mu(\d x), \qquad 
\Gamma^{(k)}(\d x)=\frac{\theta_m}{x+\theta_m}\prod_{\ell=1}^k\frac{(\ell-1)+x}{\ell+x+\theta_m}\mu(\d x),
\end{equation}
and
\begin{equation} \label{eq:pk}
	p(k) =\int_0^\infty \frac{\theta_m}{x+\theta_m}\prod_{\ell=1}^k\frac{(\ell-1)+x}{\ell+x+\theta_m}\mu(\d x).
\end{equation}
\end{theorem}
	
\begin{remark}
Throughout this article we work with Definition \ref{def:paf}. However, Theorem~\ref{Thrm:degree} also holds 
under the following slightly weaker conditions. 
Set
\be
\Fb:=\frac{1}{n}\sum_{i=1}^{n}(\Zm_n(i)+\F_i),
\ee 
and define the degree increment at step $n+1$ of vertex $i$ by $\Delta \Zm_n(i):=\Zm_{n+1}(i)-\Zm_n(i)$. We  assume the graph $\G_{n_0}$ is given deterministically such that $m_0 := \sum_{i \in [n_0]} \Zm_{n_0}(i) \geq 1$. Furthermore, we assume  for $n \geq n_0$,
\text{}
\begin{enumerate}[label=(A\arabic*)]
	\item \label{Ass:A1} $\E{\Delta \Zm_n(i)\;|\;\G_n}= (\Zm_n(i)+\F_i)/(n\Fb)\ind_{\{i\leq n\}}.$
	\item \label{Ass:A4} $\exists \  C_{\mathrm{var}}>0:\ \Var(\Delta \Zm_n(i)\;|\;\G_n)\leq C_{\mathrm{var}}\E{\Delta \Zm_n(i)\;|\;\G_n}$.
	\item \label{Ass:A3} $\sup_{i=1,\ldots,n} n\big|\P{\Delta \Zm_n(i)=1\; |\; \G_n}-\E{\Delta \Zm_n(i)\;|\;\G_n}\big| \toas 0.$
	\item \label{Ass:A5} Conditionally on $\G_n$, $\{\Delta \Zm_n(i)\}_{i\in[n]}$ is negatively quadrant dependent in the sense that for any $i\neq j$ and $k,l\in\Z^+$,
	\be \label{eq:AssA5}
	\P{\Delta \Zm_n(i)\leq k,\Delta \Zm_n(j)\leq l\;|\;\G_n}\leq \P{\dzn{i}\leq k\;|\;\G_n}\P{\dzn{j}\leq l\;|\;\G_n}.
	\ee
\end{enumerate}
As can be seen from the proof, Theorem \ref{Thrm:degree} holds for any evolving random graph model that satisfies these  assumptions. See also Lemma~\ref{lemma:nqd} below, where we show that the PAFFD and the PAFUD model satisfy the negative quadrant dependency as in~\ref{Ass:A5}.
\end{remark}

By comparing with the case where the fitness is constant, we can interpret Theorem~\ref{Thrm:degree} such that the degree of a typical vertex can be found via  a two-step process, where
first the fitness is chosen according to $\mu$ and then the degree evolves as in the case with an additive constant
equal to the fitness.

However, while at first our result looks similar to the constant fitness case, 
by looking at the tail exponent 
of the degree distribution we can see that this is only the case when the fitness is not too heavy-tailed.
Indeed, suppose that the fitness distribution follows a power law, then we can distinguish three different regimes.
As the next theorem shows, if the fitness distribution has finite  moments of order $\theta_m = 1 + \E{\cF}/m$, then the degree distribution has power law exponent $1+\theta_m$, which is the same as in the model with constant fitness equal to $\mathbb{E}[\F]$. Using the terminology used in the field of random media, we refer to this situation as the \emph{weak disorder regime}.
However, if the fitness distribution is more heavy-tailed, but still with finite first moment, then the degree distribution follows the same power law as the fitness distribution, a situation which we will refer to as the
\emph{strong order regime}. 
Finally, we can also consider the \emph{extreme disorder} case when the fitness distribution does not have a finite first moment. In this case we show that with high probability, a uniformly chosen vertex has not received any incoming edges (since most connections are made to vertices with very high fitness). 

\begin{theorem}\label{Thrm:pkasymp} Suppose $p(k), k \in \N_0$, is as in~\eqref{eq:pk} and $\theta_m = 1 + \E{\F}/m$.
\begin{itemize}
	\item[(i)] \emph{Weak disorder.}
	If $\mathbb{E}[\F^{\theta_m}] < \infty$, then for $k \rightarrow \infty$,
	\[ p(k) \sim C k^{- (1+\theta_m)}, \quad \mbox{where } C :=\theta_m\int_0^\infty \frac{\Gamma(x+\theta_m)}{\Gamma(x)}\mu(\d x),  \]
	and where $\Gamma$ is the Gamma function.
	\item[(ii)] \emph{Strong disorder.}
	Suppose $\cF$ has a power law distribution as in Assumption~\ref{Ass:powerlaw}. Then, if
	$\alpha = 1 + \theta_m$ and $\mathbb{E}[\F^{\theta_m}]=\infty$, we have as $k \rightarrow \infty$
	\[ p(k)=\Theta(\ell^\star(k)k^{-(1+\theta_m)}), \]
	where $\ell^\star(k):=\int_{1}^k \ell(x)/x \ \d x$
	and if $\alpha \in (2, 1+\theta_m)$, then as $k \rightarrow \infty$,
	\[ 	p(k)= \Theta(\ell(k) k^{-\alpha}). \]
	\item[(iii)]\emph{Extreme disorder.} Suppose $\cF$ has a power law distribution as in Assumption~\ref{Ass:powerlaw}
	with $\alpha \in (1,2)$ and consider the three PAF model as in Definition~\ref{def:paf}.
	Let $U_n$ be a uniformly chosen vertex in $\G_n$, let $\eps>0$ and let $E_n:=\{\Zm_n(U_n)=\Zm_{n_0}(U_n)\}$, be the event that $U_n$ has not increased its degree with respect to the initialisation $\G_{n_0}$. Then, for for $n$ sufficiently large,
	\be
	\P{E_n}\geq 1-C n^{-((2-\alpha)\wedge (\alpha-1))/\alpha+\eps},
	\ee 
	for some constant $C>0$. 
\end{itemize}
\end{theorem}

Our next main result provides a more detailed analysis of the dynamic behaviour of the system by describing the asymptotics of the maximal degree. 
As might be expected from the different phases observed for the tail of the degree distribution, there are also three distinct phases for the maximal degree. 
Again under the assumption that the fitness has a power law, 
we observe that in the \emph{weak disorder regime}, where the fitness has relatively light tails that the vertex with maximal degree is one of the old vertices, similar to the system with constant fitness. This first result (parts (i) and (iii) in the theorem below) in the special case of the PAFUD/PAFFD model with $m =1$ is also contained in~\cite{Sen19}.

However, if the fitness is more heavy-tailed (but still with finite first moment), i.e.\  in the \emph{strong order regime},  the maximal degree 
grows at the same rate as the maximal fitness in the system (i.e.\ approximately like $n^{1/(\alpha-1)}$). In this case,  the maximal degree satisfies a delicate balance between arriving early enough and having large fitness. 
Finally, in the \emph{extreme disorder regime}, where the fitness does not have a first moment,  the maximal degree grows of order $n$, again satisfying a non-trivial optimisation between large fitness value and arriving early. The main difference compared to the strong disorder regime is that now the sum of the fitness values in the normalization, e.g.\ in~\eqref{eq:PAFRO}, is random to first order
and depends on the extreme values  of the fitness landscape. 
As is common in extreme value theorem, the limiting variables in the strong and extreme disorder regime are described in terms of a functional of a Poisson point process capturing the extremes of the fitness 
(in competition with the advantage of arriving early).

\begin{theorem}[(Maximum) degree behaviour in PAFs]\label{Thrm:maxdegree}
	Consider the three PAF models as in Definition \ref{def:paf}. First, the following results hold for fixed degrees: 
	\begin{itemize}
		\item[(i)] Suppose $\mathbb{E}[\F^{1+\eps}] < \infty$ for some $\eps> 0$, then  for all fixed $i\in\N$,
	\be\label{eq:maxconv1}
	\Zm_n(i)n^{-1/\theta_m}\toas \xi_i,
	\ee 
	where $\xi_i$ is an almost surely finite random variable with no atom at $0$ and $\theta_m:=1+\E{\F}/m$. 
	\item[(ii)] When the fitness distribution satisfies Assumption \ref{Ass:powerlaw} with $\alpha\in (1,2)$, for all fixed $i\in\N$,
	\be \label{eq:maxconv2}
	\Zm_n(i)\toas \Zm_\infty(i),
	\ee 
	for some almost surely finite random variable $\Zm_\infty(i)$.
	\end{itemize}
In the following let $I_n:=\argmax_{i\in[n]}\Zm_n(i)$ (resolving any ties by taking the smaller index). 
	\begin{itemize} 
		\item[(iii)] \emph{Weak disorder:} If $\mathbb{E}[\F^{\theta_m + \eps}] < \infty$ for some $\eps > 0$, then we have
	\be\label{eq:indexmaxconv}
	I_n\toas I,\qquad \max_{i\in[n]}\zni n^{-1/\theta_m}\toas \sup_{i\geq 1}\xi_i,
	\ee 
	for some almost surely finite random variable $I$.
	\end{itemize}
Additionally, assume that the fitness distribution is a power law with parameter $\alpha$ as in Assumption \ref{Ass:powerlaw} and define $u_n:=\inf\{t\in\R: \P{\F\geq t} \geq 1/n\}$. Let $\Pi$ be a Poisson point process on $(0,1)\times (0,\infty)$ with intensity measure $\nu(\d t,\d x):=\d t \times (\alpha-1)x^{-\alpha}\d x$. Then, the following results hold:
\begin{itemize}
	\item[(iv)] \emph{Strong disorder:} When $\alpha\in(2,1+\theta_m)$, 
	\be\label{eq:ppplimit}
	(I_n/n,\max_{i\in[n]}\zni/u_n) \toindis (I,\sup_{(t,f)\in\Pi} f(t^{-1/\theta_m}-1)),
	\ee
	where $I$ has law
	\be \label{eq:lawI}
	\P{I\leq t}=\frac{\int_0^t(x^{-1/\theta_m}-1)^{\alpha-1}\d x}{\int_0^1(x^{-1/\theta_m}-1)^{\alpha-1}\d x},\qquad t \in[0,1],
	\ee 
	and $\max_{(t,f)\in\Pi}f(t^{-1/\theta_m}-1)$ has a Fr\'echet distribution with shape parameter $\int_0^1(x^{-1/\theta_m}-1)^{\alpha-1}\d x$.
	\item[(v)] \emph{Extreme disorder:} When $\alpha\in(1,2)$,
	\be\ba \label{eq:infmeanppplimit}
	(I_n/n,\max_{i\in[n]}\zni/n)\toindis\bigg(I, m\sup_{(t,f)\in\Pi}  f \int_t^1\bigg(\int_E g\ind_{\{u\leq s\}}\d \Pi(u,g)\bigg)^{-1}\d s\bigg),
	\ea \ee 
	for some random variable $I$ with values in $(0,1)$.
	\end{itemize}
\end{theorem}

\section{Overview of the proofs}\label{sec:overview}

In this section, we give a short overview of the proofs of the main theorems
and the structure of the remaining paper.

In Section~\ref{sec:degree} we prove Theorems~\ref{Thrm:degree} and~\ref{Thrm:pkasymp}. 
In order to prove Theorem \ref{Thrm:degree}, we  use the theory of stochastic approximation in a similar setup as in~\cite{DerOrt14}, where it was used for models with multiplicative fitness. 

The main idea is to consider, for $0 \leq f < f' < \infty$, the quantities  \be
\Gamma_n((f,f'])=\frac{1}{n}\sum_{i=1}^n \zni \ind_{\{\F_i\in(f,f']\}},\quad \text{and} \quad \Gamma^{(k)}_n((f,f'])=\frac{1}{n}\sum_{i=1}^n \ind_{\{\zni=k,\F_i\in(f,f']\}},\quad k\geq 0,
\ee 
where $0\leq f<f'<\infty$. 
Then, by considering the conditional increment  and using the preferential attachment dynamics, 
we show that 
\be
\Gamma_{n+1}((f,f'])-\Gamma_n((f,f'])\leq \frac{1}{n+1}(A_n-B_n \Gamma_n((f,f']))+(R_{n+1}-R_n),
\ee 
and also a similar lower bound with slightly different sequences $A_n,B_n$.
This should be interpreted as a time-discretisation of a differential inequality.
Then, a basic stochastic approximation argument (see also Lemma~\ref{lemma:stochapprox} below) shows that if $A_n, B_n$ and $R_n$ converge almost surely, then we obtain an upper bound on the $\limsup$ of $\Gamma_n((f,f'])$
(and similarly a lower bound). By an approximation argument this yields 
convergence of $\Gamma_n$. We obtain similar bounds for  $\Gamma_n^{(k)}((f,f'])$ (involving $\Gamma_n^{(k-1)}((f,f'])$) 
so that with an induction argument we also can deduce convergence of $\Gamma_n^{(k)}$.

In the last part of Section~\ref{sec:degree} we prove Theorem \ref{Thrm:pkasymp} 
using  standard arguments.

The remainder of the paper deals with the asymptotics of the degree of a fixed vertex, as well as the maximal degree, as stated in Theorem~\ref{Thrm:maxdegree}.  In the following we only discuss the proof for the PAFUD model, but the proofs for the PAFRO model and PAFFD model follow with minor modifications. 

A central tool in the analysis of the degree evolutions is the following martingale introduced by~\cite{Mori05} in the context of classical preferential attachment (see also~\cite{Hof16}). 
For $k \geq - \min\{ \F_i,1\}$, define a  sequence
\be \label{eq:mart}
M^k_n(i):=c^k_n{\zni+\F_i+(k-1)\choose k},\ 
\ee 
where $c^k_n$ is a carefully chosen normalisation sequence and
\[ { a \choose b} = \frac{\Gamma(a+1)}{\Gamma(b+1) \Gamma(a-b+1)}, \quad \mbox{for } a,b>-1 \mbox{ such that } a-b>-1, \]
is the generalized binomial coefficient defined in terms of the Gamma function $\Gamma$.
Next, we write 
\[ \mathbb{P}_\F \quad \mbox{ and } \quad \mathbb{E}_\F  \]
for the (regular) conditional probability measure (and its expectation respectively) when conditioning on the fitness values 
$\F_1, \F_2, \ldots$.
Then, as for the standard preferential model, one can show that 
$(M^k_n(i), n \geq i)$ is a martingale under the conditional measure $\mathbb{P}_\F$.

Note also that with $k =1$,
\[ \cZ_n(i) = (c_n^1)^{-1} M_n^1(i) - \F_i , \] 
and $M_n^1(i)$ converges being a non-negative martingale. So for fixed $i$, the asymptotics are determined by 
$c_n^1$. Indeed, we will see that
\begin{equation}\label{eq:asymp_cnk} c_n^k \approx \prod_{j=1}^{n-1} \Big(1 - \frac{k}{mj + S_j }\Big)^m
\approx \exp\bigg\{- \sum_{j=1}^{n-1} \frac{k}{j + S_j / m }\bigg \}  ,  \end{equation}
where $S_j = \sum_{\ell = 1}^j \cF_\ell$. In Lemma~\ref{lemma:ckn}, we will prove that if $\E{\F} < \infty$, then  by the law of large numbers
the sequence $c^k_n$ rescaled by $n^{k/\theta_m}$ converges almost surely. 
Moreover, if  $\alpha\in(1,2)$ (for a power law fitness distribution), then $c^k_n$ converges almost surely without rescaling. This proves the first two statements~\eqref{eq:maxconv1} and \eqref{eq:maxconv2} of Theorem~\ref{Thrm:maxdegree}.

To prove the statements about the maximal degree, we first consider 
the conditional expectation of $\cZ_n(i)$ which using the martingale $M_n^1{(i)}$
can be written as
\begin{equation}\label{eq:cond_moment} \Ef{}{\cZ_n(i)} = \F_i \Big( \frac{c_i^1}{c_n^1} - 1\Big) , \end{equation}
at least for $i > n_0$, otherwise a small correction is necessary. From this point, the proofs in the the three different regimes deviate from each other. 

First, if we assume that $\E{\F} < \infty$, then by~\eqref{eq:cond_moment} and 
the  the asymptotics of $c_n^1$ from above
we can deduce that
\begin{equation}\label{eq:first_moment} \Ef{}{\cZ_n(i)}  \approx \F_i \Big( \Big( \frac{n}{i}\Big)^{1/\theta_m} - 1\Big) . \end{equation}

Now, suppose that $\E{\F^{\theta_m + \eps}} < \infty$ for some $\eps > 0$. Then, in Lemma \ref{lemma:supmkn}, we show that 
\be
\lim_{i\to\infty}\sup_{n\geq n_0\vee i}M^1_n(i)=0.
\ee 
Intuitively, this follows from~\eqref{eq:first_moment}, since under the assumption
the maximum of the fitness values satisfies $\max_{i \in [n]} \F_i 
= o (n^{1/\theta_m})$ (with high probability), so that the term 
$( \frac{n}{i} )^{1/\theta_m}$ dominates for $i$ small. 
Thus, together with a concentration argument we obtain the weak disorder result~\eqref{eq:indexmaxconv}. 

Next, we consider the \emph{strong disorder regime}, where the fitness distribution is a power law with parameter $\alpha$  with $\alpha \in (2, 1 + \theta_m)$. Extreme value theory tell us that in this case $\max_{i \in [n]} \F_i \approx n^{-1/(\alpha -1)}$ so that~\eqref{eq:first_moment} suggests that in this regime vertices with high fitness have a chance to compete with the old vertices. To capture the asymptotics  of  the peaks of the fitness landscape more precisely, 
we consider the point process
\be \label{eq:Pin}
\Pi_n:=\sum_{i=1}^n \delta_{(i/n,\F_i/u_n)}, 
\ee 
where $u_n := \inf\{ t \geq 0 \, : \, \mathbb{P}(\F \geq t) \geq 1/n\}$. Then, classical extreme value theory (see e.g.\ the exposition in~\cite{Res13}) tells us that 
\[ \Pi_n \Rightarrow \Pi , \]
where $\Pi$ is a Poisson point process on $(0,1) \times (0,\infty)$ with 
intensity measure $\nu(\d t,\d x):=\d t \times (\alpha-1)x^{-\alpha}\d x$
(see also Section~\ref{sec:infmean} below for more details). 
From this convergence, we can then deduce using~\eqref{eq:first_moment} that
\be \label{eq:meansteps}
\max_{\inn}\Ef{}{\zni/u_n}\toindis \sup_{(t,f)\in\Pi}f(t^{-1/\theta_m}-1),
 \ee 
see the first part of Proposition~\ref{lemma:condmeanconv} for details.
A non-trivival part of the proof is showing that the approximation in~\eqref{eq:first_moment}  works sufficiently well for the relevant range of $i$.
The proof of Theorem~\ref{Thrm:maxdegree} is then completed by showing concentration of $\cZ_n(i)$ around its conditional mean, so that
\[ 
\max_{\inn}\zni/u_n-\max_{\inn}\Ef{}{\zni/u_n}\toinp 0.  \]
The concentration argument relies on the martingale $M_n^k(i)$ for carefully chosen $k$ (which correspond approximately to $k^{\mathrm{th}}$ moments of $\cZ_n(i)$), see the first part of Proposition~\ref{lemma:concentration}.

Finally, we consider the \emph{extreme disorder regime}, where $\alpha \in (1,2)$ so that the fitness does not have finite first moments. In particular, the law of large numbers no longer applies to the sum $S_n = \sum_{i=1}^n \F_i$
appearing in the normalizing constant in the attachment probabilities.
In this case, we obtain from~\eqref{eq:asymp_cnk} that for $i$ of order $n$
\[ \frac{c_i^1}{c_n^1} -1  \approx \exp\bigg\{m \sum_{j=i}^{n-1} \frac{1}{S_j}\bigg\} -1 \approx  m \sum_{j=i}^{n-1} \frac{1}{S_j}. \]
Then, it follows from~\eqref{eq:cond_moment} with the same $\Pi_n$ as in~\eqref{eq:Pin} that
\be \label{eq:functional}
\begin{aligned} 
\frac{\Ef{}{\cZ_n(i) }}{n} & \approx m
\frac{\F_i}{u_n} \Big( \frac{1}{n} \sum_{j=i}^n \frac{u_n}{S_j}\Big) \\
& =m \frac{\F_i}{u_n}\int_{i/n}^1 \Big(\int_E f\ind_{\{t\leq s\}}\d \Pi_n(f,t)\Big)^{-1}\d s \\
& =: m\frac{\F_i}{u_n}T^{i/n}(\Pi_n),
\end{aligned} 
\ee 
where $E:=(0,1)\times(0,\infty)$. 
From this we can eventually deduce that 
\[ \max_{\inn}\Ef{}{\zni/n}\toindis m\sup_{(t,f)\in\Pi}  f \int_t^1\bigg(\int_E g\ind_{\{u\leq s\}}\d \Pi(u,g)\bigg)^{-1}\d s . \]
Unfortunately, the corresponding functionals are not directly continuous in $\Pi_n$, so that the arguments involve careful cut-off arguments (see Section~\ref{sec:infmean}). 

Then, the final step is to show concentration 
\[ \max_{\inn} \zni/n-\max_{\inn}\Ef{}{\zni/n}\toinp 0 ,  \] 
which again uses the martingale $M_n^1(i)$, but in this case is slightly easier than for $\alpha > 2$.

Overall, the proof of Theorem~\ref{Thrm:maxdegree} is structured in the following way. In Section~\ref{sec:infmean}, we will first show convergence of the functional $T^{i/n}(\Pi_n)$ introduced in~\eqref{eq:functional}. 
Here, we take the opportunity to recap some of the basics of convergence of point process convergence and we will also carry out the technical cut-off arguments. Then, in Section~\ref{sec:ppp} we introduce the martingales $M_n^k(i)$ more formally and prove some of their properties. In particular, we then use those to show concentration in all three cases and also we show the point process convergence in the strong disorder case, where we can then refer back to the technical details dealt with in Section~\ref{sec:infmean} for the extreme disorder case. 
Finally, in Section \ref{sec:mainproof} we prove Theorem \ref{Thrm:maxdegree} by gathering together all the necessary results from the previous two sections.
 
\section{Degree and fitness distributions}\label{sec:degree}

This section is devoted to first proving Theorem~\ref{Thrm:degree} using the ideas of stochastic approximation and then at the end we prove Theorem~\ref{Thrm:pkasymp}. However, before we prove Theorem~\ref{Thrm:degree}, we introduce several lemmas that are required for the proof. The first lemma comes from \cite[Lemma 3.1]{DerOrt14}, which is the main ingredient in the proof of Theorem \ref{Thrm:degree}:

\begin{lemma}\label{lemma:stochapprox}
	Let $(X_n)_{n\geq 0}$ be a non-negative stochastic process. We suppose that the following estimate holds:
	\be 
	X_{n+1}-X_n\leq \frac{1}{n+1}(A_n-B_nX_n) +R_{n+1}-R_n,
	\ee 
	where
	\begin{enumerate}[label=(\roman*)]
		\item $(A_n)_{n\geq 0}$ and $(B_n)_{n\geq 0}$ are almost surely convergent stochastic processes with deterministic limits $A,B>0$.
		\item $(R_n)_{n\geq 0}$ is an almost surely convergent stochastic process.
	\end{enumerate}
	Then, almost surely,
	\be 
	\limsup_{n\rightarrow\infty}X_n\leq \frac{A}{B}.
	\ee 
	Similarly, if instead, under the same conditions $(i)$ and $(ii)$,
	\be 
	X_{n+1}-X_n\geq \frac{1}{n+1}(A_n-B_nX_n) +R_{n+1}-R_n,
	\ee 
	then almost surely,
	\be 
	\liminf_{n\rightarrow \infty}X_n \geq \frac{A}{B}.
	\ee 
\end{lemma}

In the next lemma, we discuss two specific examples of the stochastic process $R_n$ as introduced in Lemma \ref{lemma:stochapprox}, which are used in the proof of Theorem \ref{Thrm:degree}:

\begin{lemma}\label{lemma:rnconv}
	Recall $\Gamma_n$ and $\Gamma_n^{(k)}$ from \eqref{eq:gammas} and let $0\leq f<f'<\infty,k\in\N_0$ and assume the fitness distribution has a finite mean. We then have the two following results:
	\begin{enumerate}
		\item[$(i)$] Set $X_n:=\Gamma_n^{(k)}((f,f'])$, $\Delta R_n:=X_{n+1}-\E{X_{n+1}\,|\,\G_n}$ and $R_n:=\sum_{j=n_0}^n \Delta R_j$. Then $R_n$ converges almost surely.
		\item[$(ii)$] Set $X_n:=\Gamma_n((f,f'])$, $\Delta R_n:=X_{n+1}-\E{X_{n+1}\,|\,\G_n}$ and $R_n:=\sum_{j=n_0}^n \Delta R_j$. Then $R_n$ converges almost surely.
		\end{enumerate}
\end{lemma} 

Before proving Lemma \ref{lemma:rnconv}, we recall the concept of negative quadrant dependence (NQD) as introduced in \eqref{eq:AssA5}. We note that the PAFRO model has been defined with an additional assumption of non-positively correlated degree increments. Note that, since the degree increments in this model are Bernoulli random variables, NQD is equivalent to non-positive correlation. For the PAFFD and PAFUD models, NQD follows directly from the definition of the model, as we show in the following lemma:

\begin{lemma}\label{lemma:nqd}
	Recall the degree increments $\Delta \zni:=\Zm_{n+1}(i)-\zni$. For the PAFUD and PAFFD model, the $(\Delta \zni)_{\inn}$ are negative quadrant dependent, in the sense of \eqref{eq:AssA5}.
\end{lemma}

\begin{proof}
	The NQD of the PAFFD model directly follows from \cite{JoaPro83}, as $(\Delta \zni)_{\inn}$ forms a multinomial distribution, for which NQD is known. For the PAFUD model, $(\Delta \zni)_{\inn}$ is a convolution of unlike multinomial distributions (the probabilities of the multinomial distribution change at each step/sampling), for which NQD is proved in \cite{JoaPro83} as well. However, since the changes in the probabilities are dependent on the previous samplings (where previous edges are attached), we require a slightly more careful argument. Let us write $\Delta\zni:=X_1+\ldots+X_m,\Delta \Zm_n(j):=Z_1+\ldots+Z_m$, where the $X_k,Z_k$ are Bernoulli random variables which take value $1$ if the $k^{\mathrm{th}}$ edge of vertex $n+1$ connects to $i,j$, respectively, $k\in[m]$. Since $X_1,Z_1$ are part of a multinomial vector with one trial, \eqref{eq:AssA5} holds for these random variables. Then, we investigate $X_1+X_2,Z_1+Z_2$, where we prove \eqref{eq:AssA5} for $X_1+X_2,Z_1+Z_2$, but with $\geq$ rather than $\leq$ in the event, which is an equivalent definition of NQD. We write, for $k,\ell\geq 0$,
	\be 
	\P{X_1+X_2\geq k,Z_1+Z_2\geq \ell\;|\;\G_n}=\E{\P{X_2\geq k-X_1,Z_2\geq \ell-Z_1\;|\;\G_n,X_1,Z_1}\;|\;\G_n}.
	\ee 
	Since conditional on $\G_n$ and $(X_1,Z_1)$, the random variables $(X_2,Z_2)$ are part of a multinomial vector with a single trial, by the same argument we used for $X_1,Z_1$, we obtain the upper bound
	\be \label{eq:nqdstep}
	\E{\P{X_2\geq k-X_1\;|\;\G_n,X_1,Z_1}\P{Z_2\geq \ell-Z_1\;|\;\G_n,X_1,Z_1}\;|\;\G_n}.
	\ee 
	It follows from the definition of the PAFUD model that $X_2$, conditional on $X_1$, is independent of $Z_1$ and $Z_2$, conditional on $Z_1$, is independent of $X_1$. Then, as the probabilities in \eqref{eq:nqdstep} are increasing functions of $X_1,Z_1$, respectively, it follows from the definition of negative association in \cite{JoaPro83}, which is equivalent to NQD, that
	\be \ba
	\mathbb{E}&[\P{X_2\geq k-X_1\;|\;\G_n,X_1,Z_1}\P{Z_2\geq \ell-Z_1\;|\;\G_n,X_1,Z_1}\;|\;\G_n]\\
	&\leq \E{\P{X_2\geq k-X_1\;|\;\G_n,X_1}\;|\;\G_n}\E{\P{Z_2\geq \ell-Z_1\;|\;\G_n,Z_1}\;|\;\G_n}\\
	&=\P{X_1+X_2\geq k\;|\;\G_n}\P{Z_1+Z_2\geq \ell\;|\;\G_n}.
	\ea\ee
	We can continue the same argument to obtain the same inequality for the $m$ terms in $\Delta\zni=X_1+\ldots+X_m,\Delta \Zm_n(j)=Z_1+\ldots+Z_m$. We then recall that this result is equivalent to \eqref{eq:AssA5}, as required.
\end{proof}

We now prove Lemma \ref{lemma:rnconv}.

\begin{proof}[Proof of Lemma \ref{lemma:rnconv}]
	We first note that, in both cases, $R_n$ is a zero-mean martingale with respect to $\G_n$. The convergence of $R_n$ can be proved by showing its martingale increments $\Delta R_n=R_{n+1}-R_n$ have summable conditional second moments, or have summable second moments. We first deal with case $(i)$. We write $\Delta R_n$ as the difference of two martingales. For $k\geq 1$,
	\be
	\Delta R_n=\frac{1}{n+1}\sum_{i\in\I_n}\left(\ind_{\{\Zm_{n+1}(i)=k\}}-\P{\Zm_{n+1}(i)=k\;|\;\G_n}\right)=\Delta M^{(1)}_n-\Delta M_n^{(2)},
	\ee 
	where $\Delta M_n^{(i)}:=M_{n+1}^{(i)}-M_n^{(i)},\ i\in\{1,2\}$, and
	\be \ba \label{eq:mns}
	M_{n+1}^{(1)}&:=M_n^{(1)}+\frac{1}{n+1}\bigg(\sum_{i\in\I_n}\ind_{\{\Zm_n(i)<k,\Zm_{n+1}(i)\geq k+1\}}-\mathbb{E}\bigg[\sum_{i\in\I_n}\ind_{\{\Zm_n(i)<k,\Zm_{n+1}(i)\geq k\}}\,\bigg|\,\G_n\bigg]\bigg),\\
	M_{n+1}^{(2)}&:=M_n^{(2)}+\frac{1}{n+1}\bigg(\sum_{i\in\I_n}\ind_{\{\Zm_n(i)\leq k,\Zm_{n+1}(i)> k\}}-\mathbb{E}\bigg[\sum_{i\in\I_n}\ind_{\{\Zm_n(i)\leq k,\Zm_{n+1}(i)> k\}}\,\bigg|\,\G_n\bigg]\bigg).
	\ea \ee 
	Here, we use that 
	\be \ba
	\ind_{\{\Zm_{n+1}(i)=k\}}&=\ind_{\{\Zm_{n+1}(i)=k,\Zm_n(i)\leq k\}}=\ind_{\{\Zm_{n+1}(i)\geq k,\Zm_n(i)\leq k\}}-\ind_{\{\Zm_{n+1}(i)>k,\Zm_n(i)\leq k\}}\\
	&=\ind_{\{\Zm_n(i)=k\}}+\ind_{\{\Zm_{n+1}(i)\geq k,\Zm_n(i)< k\}}-\ind_{\{\Zm_{n+1}(i)>k,\Zm_n(i)\leq k\}}.
	\ea \ee 
	We note that, as the indicators in $M_n^{(1)},M_n^{(2)}$ only differ by one index $k$, it is sufficient to prove the summability of the conditional second moment of $\Delta M_n^{(2)}$ for all fixed $k\geq 1$. So, we write
	\be \ba\label{eq:mn2bound}
	\mathbb{E}\Big[&(\Delta M_n^{(2)})^2\,\Big|\,\G_n\Big]\\
	&=\frac{1}{(n+1)^2}\mathbb{E}\bigg[\Big(\sum_{i\in\I_n}\Big(\ind_{\{\Zm_n(i)\leq k,\Zm_{n+1}(i)> k\}}-\P{\Zm_n(i)\leq k,\Zm_{n+1}(i)> k\;|\;\G_n}\Big)\Big)^2\,\bigg|\,\G_n\bigg].
	\ea\ee
	 Using the non-positive correlation of the degree increments for the PAFRO model and Lemma \ref{lemma:nqd} for the PAFFD and PAFUD models, we can bound this from above by,
	\be\ba\label{eq:Rnconvbound}
	\frac{1}{(n+1)^2}&\sum_{i\in\I_n}\mathbb{E}\Big[\Big(\ind_{\{\Zm_n(i)\leq k,\Zm_{n+1}(i)> k\}}-\P{\Zm_n(i)\leq k,\Zm_{n+1}(i)> k\;|\;\G_n}\Big)^2\,\Big|\,\G_n\Big]\\
	&\leq \frac{1}{(n+1)^2}\sum_{i\in\I_n}\ind_{\{\Zm_n(i)\leq k\}}\P{\Delta \Zm_n(i)\geq 1\;|\;\G_n}\\
	&\leq  \frac{1}{(n+1)^2}\sum_{i=1}^n \E{\Delta \Zm_n(i)\;|\;\G_n}= \frac{m}{(n+1)^2},
	\ea\ee 
	where we use Markov's inequality in the final step and use that the increments of all in-degrees is exactly $m$ by the definition of the PAFFD and PAFUD models. Hence, the final statement is summable almost surely, which proves the almost sure convergence of $ R_n$. For the PAFRO model, we use the same steps as in \eqref{eq:mn2bound} and \eqref{eq:Rnconvbound}, but take the expected value on the left- and right-hand-side. Then, using the definition of the PAFRO model, we arrive at 
	\be \label{eq:2ndmomentbound}
	\mathbb{E}\big[(\Delta M_n^{(2)})^2\big]\leq \frac{1}{(n+1)^2}\sum_{i=1}^n \E{\Delta \zni}\leq \frac{1}{(n+1)^2}\sum_{i=1}^n\frac{\E{\zni+\F_i}}{m_0+(n-n_0)}.
	\ee 
	By using the tower rule and conditioning on $\G_{n-1}$, we find
	\be 
	\E{\zni+\F_i}=\E{\E{\zni+\F_i\,|\,\G_{n-1}}}\leq\E{\Zm_{n-1}(i)+\F_i}\Big(1+\frac{1}{m_0+(n-1-n_0)}\Big).
	\ee 
	Continuing this recursion yields
	\be \ba
	\E{\zni+\F_i}&\leq \mathbb{E}[\Zm_{i\vee n_0}(i)+\F_i]\prod_{j=i\vee n_0}^{n-1}\Big(1+\frac{1}{m_0+(j-n_0)}\Big)\leq \frac{(m_0+\E{\F})(m_0+(n-n_0))}{m_0+(i\vee n_0 - n_0)}.
	\ea\ee 
	Using this upper bound in \eqref{eq:2ndmomentbound}, we obtain
	\be \label{eq:pafro2ndmoment}
	\mathbb{E}\big[(\Delta M_n^{(2)})^2\big]\leq\leq \frac{1}{(n+1)^2}\Big(C_1+\sum_{i=n_0+1}^n \frac{m_0+\E{\F}}{m_0+(i-n_0)}\Big)\leq \frac{C_1+C_2\log n}{(n+1)^2},
	\ee 
	for some constants $C_1,C_2>0$, which is indeed summable.
	
	For $k=0$, we can write $\Delta R_n$ as
	\be 
	\Delta R_n:=\Delta M_{n}^{(1)}+\Delta M_{n}^{(2)} +(\ind_{\{\F_{n+1}\in(f,f']\}}-\mu((f,f']))/(n+1),
	\ee 
	where $\Delta M_n^{(1)}=0$ and $\Delta M_n^{(2)}$ is as in \eqref{eq:mns} with $k=0$. We already proved the summability of the second conditional moment of $M_n^{(2)}$ which follows for $k=0$ as well, and the last term has a second conditional moment bounded by $\mu((f,f'])/(n+1)^2$, which is summable too. This proves the almost sure convergence of $R_n$.\\ 
	
	For $(ii)$, we have
	\be 
	\Delta R_n=\frac{1}{n+1}\sum_{i\in\I_n}(\Zm_{n+1}(i)-\E{\Zm_{n+1}(i)\;|\;\G_n})=\frac{1}{n+1}\sum_{i\in\I_n}(\Delta\Zm_{n}(i)-\E{\Delta\Zm_{n}(i)\;|\;\G_n}),
	\ee 
	as $\Zm_{n+1}(i)=\zni+\Delta \zni$. We now bound the conditional second moments of $\Delta R_n$ by
	\be \ba \label{eq:varbound}
	\E{\Delta R_n^2\,|\,\G_n}&=\frac{1}{(n+1)^2}\mathbb{E}\Big[\Big(\sum_{i\in\I_n}(\Delta\Zm_{n}(i)-\E{\Delta\Zm_{n}(i)\; |\;\G_n})\Big)^2\;\Big|\;\G_n\Big]\\
	&\leq \frac{1}{(n+1)^2}\sum_{i\in\I_n}\Var(\Delta \Zm_n(i)\;|\;\G_n).
	\ea\ee
	The second line follows from Lemma \ref{lemma:nqd} for the PAFFD and PAFUD models and from the conditional non-positive correlation of the $\zni$ for the PAFRO model. Then, for the PAFUD and PAFFD models, we use that $\Delta \zni$ is a sum of $m$ indicator random variables and hence that its variance can be bounded by a $m$ times its mean. Also noting that the sum of all the increments of the in-degrees equals $m$, we obtain the upper bound $(m/(n+1))^2$, which is summable almost surely. For the PAFRO model, we again take the expected value on both sides of \eqref{eq:varbound} to get rid of the conditional statement. Then, as the variance of $\Delta \zni$ is bounded by its mean for the PAFRO model, and the same approach as used in \eqref{eq:2ndmomentbound} through \eqref{eq:pafro2ndmoment} works here as well to arrive at a summable upper bound.	
\end{proof}

With these lemmas at hand, we can prove Theorem \ref{Thrm:degree}:

\begin{proof}[Proof of Theorem \ref{Thrm:degree}]
	We provide a proof for the PAFFD and PAFUD models, the proof for the PAFRO model follows by setting $m=1$; the additional required adjustments are all included in the proof of Lemma \ref{lemma:rnconv}.\\ 
	
	First, we show that $\Gamma_n$ converges in the weak$^*$ topology to $\Gamma$, defined in \eqref{eq:gamma}. To this end, we let $0\leq f<f'<\infty$, and set
	\be \label{eq:in}
	\I_n:=\{i\in[n]\;|\; \F_i\in(f,f']\},\quad X_n:=\frac{1}{n}\sum_{i\in\I_n}\Zm_n(i)=\Gamma_n((f,f']).
	\ee 
	We develop a recursion for $X_{n+1}-X_n$. By writing $\Zm_{n+1}(i)=\zni+\Delta \zni$ and $\Fb:=(m_0+m(n-n_0)+S_n)/n$, we find
	\be 
	\E{X_{n+1}\;|\;\G_n}=\frac{1}{n+1}\Big(\sum_{i\in \I_n}\E{\Zm_{n+1}(i)\;|\;\G_n}\Big)=X_n+\frac{1}{n+1}\Big(\sum_{i\in\I_n}\frac{\Zm_n(i)+\F_i}{n\Fb/m}-X_n\Big),
	\ee 
	where we note that this holds for both the PAFFD as well as the PAFUD model. Then,
	\be 
	X_{n+1}-X_n=\frac{1}{n+1}\Big(\sum_{i\in\I_n}\frac{\Zm_n(i)+\F_i}{n\Fb/m}-X_n\Big)+\Delta R_n,
	\ee 
	with $\Delta R_n:=X_{n+1}-\E{X_{n+1}\;|\;\G_n}$. It is now possible to write the following two bounds:
	\be \ba
	X_{n+1}-X_n&\geq \frac{1}{n+1}\Big(-\Big(1-\frac{m}{\Fb}\Big)X_n+\frac{|\I_n|}{n}\frac{mf}{\Fb}\Big)+\Delta R_n,\\
	X_{n+1}-X_n&\leq \frac{1}{n+1}\Big(-\Big(1-\frac{m}{\Fb}\Big)X_n+\frac{|\I_n|}{n}\frac{mf'}{\Fb}\Big)+\Delta R_n.
	\ea \ee 
	We note that, by the strong law of large numbers, $|\I_n|/n$ converges almost surely to $\mu((f,f'])$ and $\Fb$ converges almost surely to $m\theta_m$, where we recall that $\theta_m=1+\E{\F}/m$. From Lemma \ref{lemma:rnconv} it follows that $R_n:=\sum_{k=n_0}^{n}\Delta R_n$ converges almost surely, so it follows from Lemma \ref{lemma:stochapprox} that almost surely
	\be \ba\label{eq:Xnbounds}
	\liminf_{n\rightarrow \infty}X_n&\geq \frac{f}{\theta_m-1}\mu((f,f']),\qquad
	\limsup_{n\rightarrow \infty}X_n&\leq \frac{f'}{\theta_m-1}\mu((f,f']).
	\ea\ee 
	We now take a countable subset $ \mathbb{F}\subset [0,\infty)$ that is dense, such that for each $f\in\mathbb{F}$, $\mu(\{f\})=0$. As $\mathbb{F}$ is countable, there exists an almost sure event $\Omega_0$ on which both statements in \eqref{eq:Xnbounds} hold for any pair $f,f'\in\mathbb{F}$ such that $f<f'$. Take an arbitrary open set $U$, and approximate $U$ from below by a sequence of sets $(U_m)_{m\in\N}$, where each $U_m$ is a finite union of small disjoint intervals $(f,f']$, with $f,f'\in\mathbb{F}$. Then, for any $m\in\N$, applying a Riemann approximation to \eqref{eq:Xnbounds},
	\be \label{eq:Gammabounds}
	\liminf_{n\rightarrow \infty} \Gamma_n(U)\geq \liminf_{n\rightarrow \infty} \Gamma_n(U_m)\geq \Gamma(U_m)\text{ on $\Omega_0$.}
	\ee 
	Hence, by the monotone convergence theorem, it follows that $\liminf_{n\rightarrow \infty}\Gamma_n(U)\geq \Gamma(U)$. Likewise, for any closed set $C$, a similar argument shows that $\limsup_{n\rightarrow \infty}\Gamma_n(C)\leq \Gamma(C)$. It hence follows from the Portmanteau lemma \cite[Theorem 13.16]{KleAch13} that $\Gamma_n$ converges to $\Gamma$ a.s.\ in the weak$^*$ topology.\\ 
	
	The approach to prove the other two parts in \eqref{eq:gammaconv} is to apply induction on $k$ to the convergence of the measures $\Gamnk$ (and thus $p_n(k)$). We prove the statements in \eqref{eq:gammaconv} hold for $k=0$, the initialisation of the induction, below, and show the induction step first. Let us assume that the last two statements in \eqref{eq:gammaconv} hold for all $0\leq i<k$, for some $k\geq 1$. We now advance the induction hypothesis.

	Let us take $0\leq f<f'<\infty$, and define $X_n:=\Gamnk((f,f'])$. Then, we can write the following recurrence relation, using $\I_n$ as in \eqref{eq:in}:
	\be\ba \label{eq:stochaprox}
	\E{X_{n+1}\, \big|\, \G_n} ={} &\frac{1}{n+1}\sum_{i=1}^{n+1}\P{\Zm_{n+1}(i)=k, \F_i\in(f,f']\ \big|\ \G_n}\\
	= {}& \frac{1}{n+1}\sum_{i\in\I_n}\sum_{\ell=0}^k \ind_{\{\Zm_n(i)=\ell\}}\P{\Delta \Zm_n(i)=k-\ell\ \big|\ \G_n}\\
	= {}& \frac{1}{n+1}\Big(\sum_{i\in\I_n}\sum_{\ell=0}^{k-1} \ind_{\{\Zm_n(i)=\ell\}}\P{\Delta \Zm_n(i)=k-\ell\ \big|\ \G_n}\\
	&+\sum_{i\in\I_n}\ind_{\{\Zm_n(i)=k\}}\left(1-\P{\Delta \Zm_n(i)\geq 1\ \big|\ \G_n}\right)\Big),
	\ea\ee
	where in the second step we note that $\Zm_{n+1}(n+1)=0<k$ by definition and where we isolated the $\Zm_n(i)=k$ case in the last step. We do this, as this will prove to be the only part that does not converge to zero almost surely. We can then write
	\be \ba\label{eq:Xnrecursion}
	\E{X_{n+1}\, \big|\, \G_n}= {}& X_n+\frac{1}{n+1}\bigg(\sum_{i\in\I_n}\sum_{\ell=0}^{k-1} \ind_{\{\Zm_n(i)=\ell\}}\P{\Delta \Zm_n(i)=k-\ell\ \big|\ \G_n}\\
	& -\sum_{i\in\I_n} \ind_{\{\Zm_n(i)=k\}}\P{\Delta \Zm_n(i)\geq 1\ \big|\ \G_n}-X_n\bigg)\\
	={}& X_n+\frac{1}{n+1}\bigg(\sum_{i\in\I_n}\sum_{\ell=0}^{k-1} \ind_{\{\Zm_n(i)=\ell\}}\P{\Delta \Zm_n(i)=k-\ell\ \big|\ \G_n}\\
	& -\sum_{i\in\I_n} \ind_{\{\Zm_n(i)=k\}}\left(\P{\Delta \Zm_n(i)\geq 1\ \big|\ \G_n}-\frac{k+\F_i}{n\Fb/m}\right) \\
	&+\sum_{i\in\I_n}\ind_{\{\Zm_n(i)=k\}}\left(\frac{f'-\F_i}{n\Fb/m}\right)-\left(1+\frac{k+f'}{\Fb/m}\right)X_n\bigg).
	\ea\ee 
	We can therefore write, using that $f'-\F_i\geq 0$ holds almost surely for all $i\in\I_n$, 
	\be\label{eq:lowerboundstochapprox}
	X_{n+1}-X_n\geq \frac{1}{n+1}\left(A_n-B_n X_n\right)+R_{n+1}-R_n,
	\ee
	where
	\be \ba\label{eq:anbnrn}
	A_n:={}&\sum_{i\in\I_n}\sum_{\ell=0}^{k-1} \ind_{\{\Zm_n(i)=\ell\}}\P{\Delta \Zm_n(i)=k-\ell\, \big|\, \G_n}\\
	&-\sum_{i\in\I_n} \ind_{\{\Zm_n(i)=k\}}\left(\P{\Delta \Zm_n(i)\geq 1\, \big|\, \G_n}-\frac{k+\F_i}{n\Fb/m}\right),\\
	B_n:={}&1+\frac{k+f'}{\Fb/m},\\
	\Delta R_n:={}&R_{n+1}-R_n=X_{n+1}-\E{X_{n+1}\,|\,\G_n}.
	\ea\ee 
	We now prove the convergence of all three terms. First, we prove the convergence of $A_n$ to 
	\be\label{eq:A}
	A:=\frac{1}{\theta_m}\int_{(f,f']}(k-1+x)\ \Gamma^{(k-1)}(\d x).
	\ee 
	We note that, by the induction hypothesis, almost surely,
	\be \label{eq:gamconv}
	\lim_{n\rightarrow\infty}\Big|\int_{(f,f']}(k-1+x)\ \Gamma^{(k-1)}(\d x)-\int_{(f,f']}(k-1+x)\ \Gamma^{(k-1)}_n(\d x)\Big|=0.
	\ee 
	We now deal with the two terms in $A_n$ separately. We start with the second term. By the definition of the PAFFD and PAFUD models in Definition \ref{def:paf}, it follows that for both models,
 	\be 
 	\P{\Delta \zni \geq 1\,|\,\G_n}\leq 1-\Big(1-\frac{\zni+\F_i}{n\Fb}\Big)^m=\sum_{\ell=1}^m {m\choose \ell}(-1)^{\ell+1} \Big(\frac{\zni+\F_i}{n\Fb}\Big)^\ell.
 	\ee 
 	Using this in the second term of $A_n$ in \eqref{eq:anbnrn}, we obtain
 	\be \label{eq:probbigger1}
 	\sum_{i\in\I_n}\ind_{\{\zni=k\}}\sum_{\ell=2}^m {m\choose \ell}(-1)^{\ell+1} \Big(\frac{k+\F_i}{n\Fb}\Big)^\ell\leq C_m\sum_{\ell=2}^m n^{1-\ell} \Big(\frac{k+f'}{\Fb}\Big)^\ell,  
 	\ee 
 	where $C_m>0$ is a constant. We note that this expression tends to zero almost surely as $n$ tends to infinity, and that a similar lower bound that tends to zero almost surely can be constructed as well. For the first term, we write,
	\be \ba\label{eq:Anconv}
	\lim_{n\rightarrow \infty}\Big|\sum_{i\in\I_n}&\sum_{\ell=0}^{k-1} \ind_{\{\Zm_n(i)=\ell\}}\P{\Delta \Zm_n(i)=k-\ell\, \big|\, \G_n}-\frac{1}{\theta_m}\int_{(f,f']}(k-1+x)\ \Gamma^{(k-1)}(\d x)\Big|\\
	\leq{}& \lim_{n\rightarrow \infty}\bigg[\ \Big|\frac{1}{\theta_m}-\frac{1}{\Fb/m}\Big|\int_{(f,f']}(k-1+x)\ \Gamma^{(k-1)}(\d x) \\
	&+\frac{1}{\Fb/m}\Big|\int_{(f,f']}(k-1+x)\ \Gamma^{(k-1)}(\d x)-\int_{(f,f']}(k-1+x)\ \Gamma^{(k-1)}_n(\d x)\Big|\\
	&+ \Big|\sum_{i\in\I_n}\ind_{\{\Zm_n(i)=k-1\}}\P{\Delta \Zm_n(i)=1\, \big|\, \G_n}-\frac{1}{\Fb/m}\int_{(f,f']}(k-1+x)\ \Gamma^{(k-1)}_n(\d x)\Big|\\
	&+\sum_{i\in\I_n}\sum_{\ell=0}^{k-2}\ind_{\{\Zm_n(i)=\ell\}}\P{\Delta \Zm_n(i)\geq 2\, \big|\, \G_n} \; \bigg].
	\ea\ee 
	The first line converges to zero almost surely by the strong law of large numbers. By the induction hypothesis as used in \eqref{eq:gamconv}, the second line converges to zero almost surely and by a similar argument as in \eqref{eq:probbigger1} the last line converges to zero almost surely. For the third line, we use the definition of $\Gamma^{(k-1)}_n$, as defined in \eqref{eq:gammas}, to find
	\be \ba
	\sum_{i\in\I_n}&\ind_{\{\Zm_n(i)=k-1\}}\P{\Delta \Zm_n(i)=1\, \big|\, \G_n}-\frac{1}{\Fb/m}\int_{(f,f']}(k-1+x)\ \Gamma^{(k-1)}_n(\d x)\\
	&=\sum_{i\in\I_n}\ind_{\{\Zm_n(i)=k-1\}}\Big(\P{\Delta \Zm_n(i)=1\, \big|\, \G_n}-\frac{k-1+\F_i}{n\Fb/m}\Big),
	\ea\ee 
	and so, again using similar steps as in \eqref{eq:probbigger1}, the third line in \eqref{eq:Anconv} converges to zero almost surely, which finishes the proof of the almost sure convergence of $A_n$ to $A$, as in \eqref{eq:A}. Now, for $B_n$ we immediately conclude that
	\be 
	\lim_{n\rightarrow \infty}B_n=1+\frac{k+f'}{\theta_m}=:B,
	\ee 
	almost surely. Finally, the almost sure convergence of $R_n$ again follows from Lemma \ref{lemma:rnconv}. We thus obtain from Lemma \ref{lemma:stochapprox},
	\be\label{eq:Xnlowerbound}
	\liminf_{n\rightarrow \infty}X_n\geq \frac{A}{B}=\frac{1}{k+f'+\theta_m}\int_{(f,f']}k-1+x\ \Gamma^{(k-1)}(\d x).
	\ee
	Likewise, the upper bound
	\be\label{eq:Xnupperbound}
	\limsup_{n\rightarrow\infty}X_n\leq \frac{1}{k+f+\theta_m}\int_{(f,f']}k-1+x\ \Gamma^{(k-1)}(\d x)
	\ee
	can be established from \eqref{eq:Xnrecursion}, too, when we replace the $f'$ by $f$ in \eqref{eq:Xnrecursion} and note that $f-\F_i\leq 0$ holds almost surely for all $i\in\I_n$. 

	We now again take a countable subset $ \mathbb{F}\subset [0,\infty)$ that is dense, such that for each $f\in\mathbb{F}$, $\mu(\{f\})=0$. As $\mathbb{F}$ is countable, there exists an almost sure event $\Omega_0$ on which both \eqref{eq:Xnlowerbound} and \eqref{eq:Xnupperbound} hold for any pair $f,f'\in\mathbb{F}$ such that $f<f'$. A similar argument as in \eqref{eq:Xnbounds} and \eqref{eq:Gammabounds} can be made, using Riemann approximations and the Portmanteau lemma, which yields for any open set $ U\subseteq[0,\infty)$ and any closed set $ C\subseteq[0,\infty)$,
	\be \ba\label{eq:Uliminf}
	\liminf_{n\rightarrow \infty}\Gamnk(U)&\geq \int_U\frac{k-1+x}{k+x+\theta_m} \Gamma^{(k-1)}(\d x),\\
	\limsup_{n\rightarrow \infty}\Gamnk(C)&\leq \int_C\frac{k-1+x}{k+x+\theta_m} \Gamma^{(k-1)}(\d x),
	\ea\ee 
	and thus $\Gamnk$ converges in the weak$^*$ topology to $\Gamma^{(k)}$, given by 
	\be 
	\Gamma^{(k)}(\d x)=\frac{(k-1)+x}{k+x+\theta_m} \Gamma^{(k-1)}(\d x)=\ldots=\prod_{\ell=1}^k \frac{(\ell-1)+x}{\ell+x+\theta_m} \Gamma^{(0)}(\d x).
	\ee 
	What remains is to perform the initialisation of the induction, regarding $\Gamma^{(0)}_n$. Analogous to the steps in \eqref{eq:stochaprox}, we now set $X_n:=\Gamma^{(0)}_n((f,f'])$, with $0\leq f<f'<\infty$, to obtain
	\be \ba 
	\E{X_{n+1}\;|\;\G_n}&=\frac{1}{n+1}\bigg(\sum_{i\in\I_n} \P{\Zm_{n+1}(i)=0\;|\;\G_n}+\P{\F_{n+1}\in(f,f']}\bigg)\\
	&=\frac{1}{n+1}\bigg(\sum_{i\in\I_n}\ind_{\{\Zm_n(i)=0\}}\P{\Delta \Zm_n(i)=0\;|\;\G_n}+\mu((f,f'])\bigg)\\
	&=X_n+\frac{1}{n+1}\bigg(-\sum_{i\in\I_n} \ind_{\{\Zm_n(i)=0\}}\P{\Delta \Zm_n(i)\geq 1\;|\;\G_n}-X_n+\mu((f,f'])\bigg).
	\ea \ee 
	Similar to \eqref{eq:Xnrecursion}, \eqref{eq:lowerboundstochapprox} and \eqref{eq:anbnrn}, we find
	\be \label{eq:Xnupperbound2}
	X_{n+1}-X_n\geq \frac{1}{n+1}(A_n-B_nX_n)+\Delta R_n,
	\ee 
	where $A_n\rightarrow \mu((f,f'])$, $B_n\rightarrow (f'+\theta_m)/\theta_m$ a.s. as $n\to\infty$, and $\Delta R_n=R_{n+1}-R_n:=X_{n+1}-\E{X_{n+1}\,|\,\G_n}$. As before, the almost sure convergence of $R_n$ follows from Lemma \ref{lemma:rnconv}. Analogously to \eqref{eq:Xnupperbound2}, 
	\be  
	X_{n+1}-X_n\leq \frac{1}{n+1}(A_n-B_n'X_n)+\Delta R_n
	\ee 
	holds, with $B_n'\rightarrow (1+f+\theta_m)/\theta_m$ almost surely. Hence, using Lemma \ref{lemma:stochapprox},
	\be \ba
	\liminf_{n\rightarrow \infty} X_n\geq  \frac{\theta_m}{f'+\theta_m}\mu((f,f']),\qquad
	\limsup_{n\rightarrow \infty} X_n \leq \frac{\theta_m}{f+\theta_m}\mu((f,f']),
	\ea\ee 
	and thus, with a similar reasoning as in \eqref{eq:Uliminf}, almost surely $\Gamma^{(0)}_n$ converges weakly in the weak$^*$ topology to
	\be
	\Gamma^{(0)}(dx):=\frac{\theta_m}{x+\theta_m}\mu(\d x),
	\ee
	which yields
	\be
	\Gamma^{(k)}(dx)=\frac{\theta_m}{x+\theta_m}\prod_{\ell=1}^k \frac{(\ell-1)+x}{\ell+x+\theta_m}\mu(\d x).
	\ee
	Then,
	\be
	p(k):=\lim_{n\rightarrow\infty}p_n(k)=\int_0^\infty\frac{\theta_m}{x+\theta_m}\prod_{\ell=1}^k \frac{(\ell-1)+x}{\ell+x+\theta_m}\mu(\d x),
	\ee
	which proves \eqref{eq:gammaconv} and concludes the proof.
	\end{proof}

We now prove Theorem \ref{Thrm:pkasymp}:

\begin{proof}[Proof of Theorem \ref{Thrm:pkasymp}]
	We start by proving $(i)$. The integrand of the integral in \eqref{eq:pk} can be written as
	\be 
	\frac{\theta_m}{x+\theta_m}\prod_{\ell=1}^k \frac{(\ell-1)+x}{\ell+x+\theta_m}=\theta_m \frac{\Gamma(x+\theta_m)}{\Gamma(k+x+1+\theta_m)}\frac{\Gamma(k+x)}{\Gamma(x)}.
	\ee 
	From \cite[Theorem 1]{Jam13} it follows that $k^{1+\theta_m}\Gamma(k+x)/\Gamma(k+x+1+\theta_m)\leq 1$ for all $x,k \geq 0$. By also using that $\Gamma(t+a)/\Gamma(t)=t^a(1+\mathcal{O}(1/t))$ as $t\rightarrow\infty$ and $a$ fixed, we find that the dominated convergence theorem yields
	\be 
	\lim_{k\to\infty}p(k)k^{1+\theta_m}=\int_0^\infty \theta_m\frac{\Gamma(x+\theta_m)}{\Gamma(x)}\mu(\d x),
	\ee
	which is finite since $\mathbb{E}[\F^{\theta_m}]<\infty$. We now prove $(ii)$, so the fitness distribution satisfies Assumption \ref{Ass:powerlaw}. First, let $\alpha\in(2,1+\theta_m)$. We write the integral in \eqref{eq:pk} as two separate integrals by splitting the domain into $(0,k)$ and $(k,\infty)$. We first concentrate on an upper bound. We note that, by symmetry, it also follows that $x^{1+\theta_m}\Gamma(k+x)/\Gamma(k+x+1+\theta_m)\leq 1$. Hence, we obtain the upper bound
	\be \label{eq:powerlawsplit1}
	k^{-(1+\theta_m)}\int_{0}^k\theta_m \frac{\Gamma(x+\theta_m)}{\Gamma(x)x^{\theta}}x^{\theta}\mu(\d x)+\int_{k}^\infty \theta_m \frac{\Gamma(x+\theta_m)}{\Gamma(x)x^{\theta_m}}x^{-1}\mu(\d x).
	\ee 
	We note that there exists a constant $c>1$ such that $\Gamma(x+\theta_m)/(\Gamma(x)x^{\theta_m})\in [1,c]$ when $x\geq 1$. Hence, using Assumption \ref{Ass:powerlaw}, we can bound \eqref{eq:powerlawsplit1} from above by
	\be \ba\label{eq:pkasymp}
	\theta_mk^{-(1+\theta_m)}&\int_0^{1}\frac{\Gamma(x+\theta_m)}{\Gamma(x)}\mu(\d x)+c\theta_mk^{-(1+\theta_m)}\int_{1}^k x^{\theta_m}\mu(\d x)+c\theta_mk^{-1}\int_k^\infty \mu(\d x)\\
	&=o(k^{-\alpha})+c\theta_mk^{-(1+\theta_m)}\E{\F^{\theta_m}\ind_{\{1\leq \F^{\theta_m}\leq k\}}}+c\theta_m\ell(k)k^{-\alpha}\\
	&=o(k^{-\alpha})+c\theta_m^2k^{-(1+\theta_m)}\int_{1}^k x^{\theta_m-1}\ell(x)x^{-(\alpha-1)}\d x+c\theta_m\ell(k)k^{-\alpha},
	\ea\ee
	where the first term follows from the fact that $\alpha<1+\theta_m$ and that the integral from $0$ to $1$ is finite.
	Hence, by \cite[Proposition 1.5.8]{BinGolTeu87}, as $k$ tends to infinity, this is asymptotically
	\be
	\frac{c\theta_m(2\theta_m-(\alpha-1))}{\theta_m-(\alpha-1)}\ell(k)k^{-\alpha}.
	\ee
	For a lower bound, we bound the second integral in \eqref{eq:powerlawsplit1} from below by zero, and bound the first integral, using similar steps as before, from below by
	\be \label{eq:pklower}
	o(k^{-\alpha})+\theta_m^2k^{-(1+\theta_m)}\int_{1}^k x^{\theta_m-1}\ell(x)x^{-(\alpha-1)}\d x,
	\ee 
	which is asymptotically, as $k$ tends to infinity, $(\theta_m^2/(\theta_m-(\alpha-1))\ell(k)k^{-\alpha}$. Finally, for $\alpha=1+\theta_m$, we note that the first term of \eqref{eq:pkasymp} is no longer $o(k^{-\alpha})$, but of the same order as the other terms. Furthermore, since the argument of the integral in the last line of \eqref{eq:pkasymp} (as well as in \eqref{eq:pklower}) now equals $\ell(x)/x$, the integral equals $\ell^\star(k)$ and it follows from \cite[Proposition 1.5.9a]{BinGolTeu87} that either $\ell^\star$ converges, in which case this falls under the first case $(i)$ as the $\theta_m^{\mathrm{th}}$ moment exists, or that $\ell^\star$ is slowly varying itself. Thus, in the latter case, we obtain an upper and lower bound with asymptotics, respectively,
	\be\ba
	\Big(\theta_m\int_0^{1}\frac{\Gamma(x+\theta_m)}{\Gamma(x)}\mu(\d x)+c\theta_m \ell(k)+c\theta_m^2\ell^\star(k)\Big)k^{-(1+\theta_m)}&=:\overline L(k)k^{-(1+\theta_m)},\\
	\Big(\theta_m\int_0^{1}\frac{\Gamma(x+\theta_m)}{\Gamma(x)}\mu(\d x)+\theta_m^2\ell^\star(k)\Big)k^{-(1+\theta_m)}&=:\underline L(k)k^{-(1+\theta_m)}.
	\ea\ee  
	We also have from \cite[Proposition 1.5.9a]{BinGolTeu87} that, in the case that $\ell^\star$ diverges as $k$ tends to infinity, $\ell^\star(k)/\ell(k)\to\infty$ as $k\to\infty$ as well, so that $\overline L(k)\sim \underline L(k)\sim \ell^\star(k)$ as $k\to \infty$, which finishes the proof of $(ii)$.
	
	Finally, we tend to $(iii)$. We provide a proof for the PAFFD and PAFUD models with $m\geq 1$ first, and then show how the results follows for the PAFRO model as well.
	
	Recall that $U_n$ is a uniformly chosen vertex from $[n]$. We first condition on the size of the fitness of $U_n$. Let $0<\beta <((2-\alpha)/(\alpha-1)\wedge 1)$. Note that when $U_n>n_0$, $E_n$ denotes the event that $\Zm_n(U_n)=0$. Then, 
	\be\label{eq:enbound}
	\P{E_n}\geq \P{E_n\cap\{ \F_{U_n}\leq n^\beta\}}=\P{\F_{U_n}\leq n^\beta}-\P{E_n^c\cap\{ \F_{U_n}\leq n^\beta\}}.
	\ee 
	Clearly, for $\eps>0$ fixed and $n$ large,
	\be \label{eq:FUn}
	\P{\F_{U_n}\leq n^\beta}=\P{\F\leq n^\beta}=1-\ell(n^\beta)n^{-(\alpha-1)\beta}\geq 1-n^{-(\alpha-1)\beta+\eps},
	\ee
	where we use Potter's theorem \cite[Theorem 1.5.6]{BinGolTeu87}, which states that for any fixed $\eps>0$ and any function $\ell$, slowly-varying at infinity,
	\be\label{eq:potter}
	\lim_{x\rightarrow \infty}\ell(x)x^\eps=\infty,\qquad 	\lim_{x\rightarrow \infty}\ell(x)x^{-\eps}=0.
	\ee
	For the second probability on the right-hand-side of \eqref{eq:enbound}, we write
	\be\ba 
	\P{E_n^c\cap\{ \F_{U_n}\leq n^\beta\}}&=\mathbb{P}\Big(\bigcup_{j=U_n\vee n_0}^{n-1}\{\Delta \Zm_j(U_n)\geq 1\}\cap\{ \F_{U_n}\leq n^\beta\}\Big) \\
	& = \sum_{k=1}^n \frac{1}{n} \mathbb{P}\Big(\bigcup_{j=k\vee n_0}^{n-1}\{\Delta \Zm_j(k)\geq 1\}\cap\{ \F_{k}\leq n^\beta\}\Big) \\
	& \leq \sum_{k=1}^n \sum_{j=k\vee n_0}^{n-1} \frac{1}{n}\P{\{\Delta \Zm_j(k)\geq 1\}\cap\{ \F_{k}\leq n^\beta\}}.
	\ea\ee
	Now, using Markov's inequality, applying the tower rule and switching the summations yields the upper bound, writing $\Fb=(m_0+m(n-n_0)+S_n)/n$,
	\be \ba\label{eq:expMibound}
	\frac{1}{n} \sum_{j=n_0}^{n-1}& \sum_{k=1}^{j}\E{ (\Zm_{j}(k)+n^\beta)/(j\bar \F_j)\ind_{\{\F_k\leq n^\beta\}}}\\
	={}& \frac{1}{n}\sum_{j=n_0}^{n-1}\sum_{k=1}^j\Big(\E{\Zm_j(k)/(j\bar \F_j)\ind_{\{\F_k\leq n^\beta\}} }+n^{\beta} \E{(j\bar \F_j)^{-1}\ind_{\{\F_k\leq n^\beta\}}}\Big)\\
	\leq{}& \frac{1}{n}\sum_{j=n_0}^{n-1}\sum_{k=1}^j \Big( \E{\Zm_j(k)/(m_0+M_j)}+n^{\beta}\E{(m_0+M_j)^{-1}}\Big),
	\ea \ee
	where $M_j:=\max_{k\leq j}\F_k$, we bound $j\bar\F_j$ from below by $m_0+M_j$ and we bound the indicator variables from above by $1$. We now bound the first moment from above. Note that, for the PAFFD and PAFUD models, 
	\be \label{eq:sumznexp1}
	\sum_{k=1}^j\E{ \Zm_j(k)}=m_0+m(j-n_0),
	\ee
	since every vertex $i>n_0$ has out-degree $m$. Hence, combining \eqref{eq:expMibound} and \eqref{eq:sumznexp1}, we obtain the upper bound, by using the tower rule and conditioning on the fitness,
	\be\ba \label{eq:encfinalbound}
	\frac{1}{n}\sum_{j=n_0}^{n-1}(m+ m_0+n^{\beta})j\E{1/(m_0+M_j)}\leq Cn^{\beta-1}\sum_{j=n_0}^{n-1}j\E{1/(m_0+M_j)},
	\ea \ee
	when $n$ is sufficiently large, for some constant $C>0$. We now bound $\E{1/(m_0+M_j)}$ from above.
	\be \ba\label{eq:Mibound}
	\E{1/(m_0+M_j)}={}&\E{1/(m_0+M_j)\ind_{\{ M_j\leq j^{1/(\alpha-1)-\eps}\}}}+\E{1/(m_0+M_j)\ind_{\{M_j\geq j^{1/(\alpha-1)-\eps}\}}}\\
	\leq{}& \P{M_j\leq j^{1/(\alpha-1)-\eps}}+j^{-1/(\alpha-1)+\eps}
	\ea \ee 
	where we bound $M_j$ from below by zero and $j^{1/(\alpha-1)-\eps}$ in the first and second expectation, respectively. Then, using $1-x\leq \mathrm{e}^{-x}$, for $j$ large,
	\be \ba  \label{eq:Miprobbound}
	\P{M_j\leq j^{1/(\alpha-1)-\eps}}\leq \exp\{-\ell(j^{1/(\alpha-1)-\eps})j^{(\alpha-1)\eps}\}\leq \exp\{-j^{(\alpha-1)\eps/2}\},
	\ea \ee 
	where we use Potter's theorem, as in \eqref{eq:potter}, in the last step. By combining \eqref{eq:Mibound} and \eqref{eq:Miprobbound}, it follows that for $j$ sufficiently large (say $j> j_0$ for some $j_0\in\N$),
	\be 
	\E{1/(m_0+M_j)}\leq 2j^{-1/(\alpha-1)+\eps},
	\ee 
	and $\E{1/(m_0+M_j)}\leq 1$ for $j\leq j_0$. Using this in \eqref{eq:encfinalbound} yields
	\be \ba\label{eq:boundenc}
	\P{E_n^c\cap \{  \F_{U_n}\leq n^\beta\}}&\leq Cj_0n^{\beta-1}+4Cn^{\beta-1}\sum_{j=j_0+1}^{n-1} j^{1-1/(\alpha-1)+\eps}\leq \wt C n^{\beta+((1-1/(\alpha-1))\vee -1)+\eps}\\
	&=\wt C n^{\beta-((2-\alpha)/(\alpha-1)\wedge 1)+\eps},
	\ea \ee 
	which, by the definition of $\beta$ and the fact that $\eps$ is arbitrarily small, tends to zero as $n$ tends to infinity. Finally, we combine \eqref{eq:boundenc} and \eqref{eq:FUn} in \eqref{eq:enbound} to find
	\be \label{eq:Enfinal}
	\P{E_n}\geq 1- n^{-(\alpha-1)\beta+\eps}-\wt C n^{\beta-((2-\alpha)/(\alpha-1)\wedge 1)+\eps}.
	\ee 
	We now finish the proof of Theorem \ref{Thrm:degree} by choosing the optimal value of $\beta\in(0,((2-\alpha)/(1-\alpha)\wedge 1))$, namely $\beta=(2-\alpha)/(\alpha(\alpha-1))\wedge (1/\alpha)$, and setting $C=1+\wt C$.\\
	
	For the PAFRO model, set $m$ to equal $1$. Then, there is one adjustment required. Namely, the equality in \eqref{eq:sumznexp1} does not hold. Rather, using \eqref{eq:pafro2ndmoment} yields the upper bound 
	\be
	\sum_{k=1}^j \Ef{}{\Zm_j(k)}\leq Cj(\log j -1)\leq Cj^{1+\eps},
	\ee 
	for some large constant $C>0$. This adds at most an extra $\eps$ in the exponent of the final expression in \eqref{eq:Enfinal} and since $\eps$ is arbitrarily small, the result still holds, which concludes the proof.
\end{proof}

\section{Convergence of point process functionals}\label{sec:infmean}

As mentioned in the proof overview in Section~\ref{sec:overview}, 
in this section, we complete an important step in the proof of Theorem~\ref{Thrm:maxdegree} and show convergence of a functional of a point process as defined in \eqref{eq:functional} in the extreme disorder case ($\alpha \in (1,2)$). At the same time, we take the chance to discuss some of the required theory of point process convergence, which will also be useful in the next section when we consider the strong disorder case. A good reference for this theory is the book~\cite{Res13}.

Recall $u_n$ from Theorem \ref{Thrm:maxdegree} and let $M_p(E)$ be the space of point measures (point processes) on $E:=(0,1)\times (0,\infty)$. Let us define the point process 
\be\label{eq:pin}
\Pi_n:=\sum_{i=1}^n \delta_{(i/n,\F_i/u_n)},
\ee 
with $\delta$ a Dirac measure. It follows from \cite[Corollary 4.19]{Res13} that, when the fitness distribution satisfies Assumption \ref{Ass:powerlaw} for any $\alpha>1$, $\Pi_n$ has a weak limit $\Pi$, which is a Poisson point process (PPP) on $E$ with intensity measure $\nu(\d t, \d x):=\d t\times (\alpha-1)x^{-\alpha}\d x$.  \cite[Proposition 4.20]{Res13} shows that an almost surely continuous functional $T_1$ applied to $\Pi_n$ converges in distribution to $T_1$ applied to $\Pi$ by the continuous mapping theorem. In this section, we prove a similar result, though a slightly different approach is required.

Let $\eps,\delta>0,E_\delta:=(0,1)\times(\delta,\infty)$. For a point measure $\Pi\in M_p(E)$, define
\be \label{eq:Top}
T^\eps(\Pi):=\int_\eps^1 \Big(\int_{E}f\ind_{\{t\leq s\}}\d \Pi(t,f)\Big)^{-1}\d s, \qquad T^\eps_\delta(\Pi):=\int_\eps^1 \Big(\int_{E_\delta}f\ind_{\{t\leq s\}}\d \Pi(t,f)\Big)^{-1}\d s,
\ee 
whenever these are well-defined. That is, when $\Pi((0,s)\times (0,\infty))>0$ for all $s\in(\eps,1)$ and when $\Pi((0,s)\times(\delta,\infty))>0$ for all $s\in(\eps,1)$, respectively. The main goal in this section is to prove the following proposition:

\begin{proposition}\label{prop:infmeanmaxconv}
	Let $(\F_i)_{i\in\N}$ be i.i.d.\ copies of a random variable $\F$, which follows a power-law distribution as in Assumption \ref{Ass:powerlaw} with $\alpha\in(1,2)$. Consider the point measure $\Pi_n$ in \eqref{eq:pin}, its weak limit $\Pi$ and the functional $T^\eps$ in \eqref{eq:Top}. Then,
	\be 
	\max_{\inn}\frac{\F_i}{u_n}T^{i/n}(\Pi_n) \toindis \sup_{(t,f)\in\Pi}f T^t(\Pi).
	\ee 
\end{proposition}

In order to prove Proposition \ref{prop:infmeanmaxconv}, one would normally prove the continuity of the functional $T^\eps$ and combine the weak convergence of $\Pi_n$ with the continuous mapping theorem to yield the required result, as Resnick does in his proof of Proposition 4.20. This does, however, not work in this case. Due to the specific form of the functional, proving its continuity is not directly possible. Therefore, we investigate $T^\eps_\delta$ as defined in \eqref{eq:Top} and show that this functional is indeed continuous and is `sufficiently close' to $T^\eps$. This is worked out in the following two lemmas:

\begin{lemma}\label{lemma:cont}
	Consider, for $\eps\in(0,1),\delta>0$ fixed, the operator $T^\eps_\delta$ as in \eqref{eq:Top}. Then, the mapping $\Pi\mapsto\sum_{(t,f)\in\Pi:t>\eps,f>\delta}\delta_{(fT_\delta^t(\Pi))}$ is continuous in the vague topology for measures $\Pi\in M_p(E)$ satisfying the following conditions:
	\be\ba\label{eq:picond}
	\Pi(\{s\}\times(0,\infty))&=\Pi((s,t)\times\{0\})=\Pi((s,t)\times \{\infty\})= 0,\qquad &&\forall s<t\in[0,1],\\
	\Pi((0,\eps)\times(\delta,\infty))&>0,\qquad \Pi((s,t)\times(x,\infty))< \infty,\qquad &&\forall s<t\in[0,1],x> 0.
	\ea \ee 
\end{lemma}

\begin{remark}\label{rem:cont}
	We note that for a PPP $\Pi$ with intensity measure $\nu$ as introduced above, all the conditions in \eqref{eq:picond} are satisfied almost surely, except for $\Pi((0,\eps)\times(\delta,\infty))>0$, which happens with positive probability only.
\end{remark}

\begin{proof}[Proof of Lemma \ref{lemma:cont}]
	We first prove that, for fixed $\eps\in(0,1),\delta>0$, the mapping \\
	$\Pi\mapsto \sum_{(t,f)\in \Pi:t>\eps,f>\delta}\delta_{(T^\eps_\delta(\Pi))}$ is continuous in the vague topology for measures $\Pi\in M_p(E)$. We obtain this by taking $\Pi_n,\Pi\in M_p(E)$ such that $\Pi_n\overset{v}{\rightarrow}\Pi$, and showing that the image of the mapping pf $\Pi_n$ introduced above also converges vaguely to the mapping of $\Pi$. Since the image is a point measure with only finitely many points, due to the last condition in \eqref{eq:picond}, we can label the points $(t,f)$ in $\Pi$ such that $t>\eps,f>\delta$, by $(t_i,f_i), 1\leq i\leq p$ for some $p\in\N$, where we order the points such that $t_i$ is increasing in $i$. We can do the same for the points of $\Pi_n$, where we add a superscript $n$. Vague convergence is then equivalent to the convergence of $(t^n_i,f^n_i)\in\Pi_n$ to $(t_i,f_i)\in\Pi$ for all $1\leq i\leq p$, since there are only finitely many points.  

	By \cite[Proposition 3.13]{Res13}, we can fix $\eta>0$ and take $n$ large enough such that the balls $B_i:=B((t_i,f_i),\eta)$, centred around $(t_i,f_i)$ with radii $\eta$, contain the points $(t_i^n,f_i^n)$ and $B_i\cap B_j=\emptyset$ for $i\neq j$. Thus, let us set $q:=\Pi((0,\eps)\times(\delta,\infty))>0$ and take $n$ large enough such that $\Pi_n((0,\eps)\times(\delta,\infty))=q$ as well. That is, points $(t_i,f_i),(t_i^n,f_i^n), i=1,\ldots,q,$ satisfy $t_i^n<\eps$ and points $(t_i,f_i),(t_i^n,f_i^n), i=q+1,\ldots,p$, satisfy $t_i^n>\eps$ (due to the first condition in \eqref{eq:picond} there are no points $(t,f)$ such that $t=\eps$ a.s.). We can now express $T^\eps_\delta(\Pi)$ in terms of a sum. Namely,
	\be\label{eq:Tinsum}
	T^\eps_\delta(\Pi)=\int_\eps^1 \Big(\int_{E_\delta} f\ind_{\{t\leq s\}} \d\Pi(t,f)\Big)^{-1} \d s = \sum_{i=q+1}^{p+1}\Big[(t_i - t_{i-1}\vee \eps) \Big(\sum_{j=1}^{i-1} f_i \Big)^{-1}\Big],
	\ee 
	where we set $t_{p+1}:=1$. A similar expression follows for $\Pi_n$, with $t^n_{p+1}:=1$. Since the sum contains a finite number of terms, the convergence of $T^\eps_\delta(\Pi_n)\to T^\eps_\delta(\Pi)$ immediately follows from the convergence of the individual points. As $\Pi_n\overset{v}{\longrightarrow}\Pi$, $f_i^n\rightarrow f_i$ as $n$ tends to infinity for all $i=1,\ldots,p$ as well. What remains to prove, is that $(T^{t^n_i}_\delta(\Pi_n),1\leq i \leq p)\to (T^{t_i}_\delta(\Pi),1\leq i \leq p)$ as $n\to\infty$. Using the triangle inequality, we obtain
	\be
	|T^{t^n_i}_\delta(\Pi_n)-T^{t_i}_\delta(\Pi)|\leq |T^{t^n_i}_\delta(\Pi_n)-T^{t_i}_\delta(\Pi_n)|+|T^{t_i}_\delta(\Pi_n)-T^{t_i}_\delta(\Pi)|.
	\ee 
	Let us first consider $2\leq i\leq p$. The second term on the right-hand-side tends to zero by the above, as for $i\geq 2$, $\Pi_n((0,t_i)\times (\delta,\infty))>0$ and thus the conditions in \eqref{eq:picond} are satisfied with $\eps=t_i$. The first term can be rewritten using the definition of $T^\eps_\delta$ in \eqref{eq:Top} as 
	\be \ba
	|T^{t^n_i}_\delta(\Pi_n)-T^{t_i}_\delta(\Pi_n)|&=\int_{t^n_i\wedge t_i}^{t^n_i\vee t_i}\Big(\int_{E_\delta}f\ind_{\{t\leq s\}}\d \Pi_n(t,f)\Big)^{-1}\d s\\
	&\leq |t^n_i-t_i|\Big(\int_{E_\delta}f\ind_{\{t\leq t^n_i\wedge t_i\}}\d \Pi_n(t,f)\Big)^{-1},
	\ea\ee 
	where we bound the integrand of the outer integral from above by replacing the integration variable $s$ by $t^n_i\wedge t_i$ in the integral's argument. In the integral that remains, we can bound $f$ from below by $\delta$ and therefore, for $n$ sufficiently large, we can bound the integral from below by $\delta$, as there is always at least one particle $(t,f)$ such that $t\leq t^n_i\vee t_i$ since $i\geq 2$ and the balls $B_i$ introduced above are disjoint. We thus obtain the upper bound $|t^n_i-t_i|/\delta$, which tends to zero with $n$. For $i=1$, we adapt our approach to find
	\be\ba
	|T^{t^n_1}_\delta(\Pi_n)-T^{t_1}_\delta(\Pi)|\leq\min\{&|T^{t^n_1}_\delta(\Pi_n)-T^{t_1}_\delta(\Pi_n)|+|T^{t_1}_\delta(\Pi_n)-T^{t_1}_\delta(\Pi)|,\\ &|T^{t^n_1}_\delta(\Pi_n)-T^{t_1^n}_\delta(\Pi)|+|T^{t_1^n}_\delta(\Pi)-T^{t_1}_\delta(\Pi)|\}.
	\ea\ee 
	When $t_1<t^n_1$, the first term is infinite and we use the second term, while the second term is infinite when $t_1>t^n_1$ and we then use the first term. When the first term in finite ($t_1>t^n_1$), its first term is bounded from above by $(t_1-t^n_1)\delta^{-1}<\eta/\delta$ and its second term can be bounded by a constant times $\eta$, as follows when using \eqref{eq:Tinsum}. Similarly, when the second term of the minimum is finite ($t_1\leq t^n_1$), its second term is bounded from above by $(t^n_1-t_1)\delta^{-1}<\eta/\delta$ and its first term can be bounded by a constant times $\eta$. As $\eta$ is arbitrary, the required result holds.
\end{proof}

We are also interested in how `close' $T^\eps(\Pi)$ and $T^\eps_\delta(\Pi)$ (resp.\ $T^\eps(\Pi_n)$ and $T^\eps_\delta(\Pi_n)$) are when $\delta$ is small (resp.\ $\delta$ is small and $n$ is large). We formalise this in the following lemma:

\begin{lemma}\label{lemma:Tdeltaconv}
	Consider the operator $T^\eps_\delta$ as in \eqref{eq:Top} and the point process $\Pi_n$ as in \eqref{eq:pin}, let $\Pi$ be its weak limit and let Assumption \ref{Ass:powerlaw} hold with $\alpha\in(1,2)$. For $\eps\in(0,1),\eta>0$ fixed, 
	\be \ba\label{eq:Tdeltabehave}
	&T^\eps_\delta(\Pi)\toinp T^\eps(\Pi)\text{ as }\delta\downarrow 0,\\
	\lim_{\delta\downarrow0}\lim_{n\rightarrow\infty}&\P{|T^\eps_\delta(\Pi_n)-T^{\eps}(\Pi_n)|\geq \eta}=0.
	\ea\ee 
\end{lemma}

\begin{proof}
	We start by proving the first statement. We fix $\eta>0$ and define $E_\delta^\xi:=(0,\eps)\times(\delta^{(2-\alpha)/2}(1+\delta^{-\xi}),\infty)$, where $\xi\in(0,(2-\alpha)/2)$. Then,
	\be\label{eq:Tconditional}
	\P{|T_\delta^\eps(\Pi)-T^\eps(\Pi)|\geq \eta}\leq\mathbb{P}(|T^\eps_\delta(\Pi)-T^\eps(\Pi)|\geq \eta\,|\,\Pi(E_\delta^\xi)\neq 0)+\mathbb{P}(\Pi(E_\delta^\xi)=0).
	\ee
	We condition on $\{\Pi(E_\delta^\xi)\neq0\}$ to ensure that $T^\eps_\delta(\Pi)$ is finite and show that on $\{\Pi(E_\delta^\xi)\neq0\}$ the difference in $T^\eps_\delta(\Pi)$ and $T^\eps(\Pi)$ will tend to zero in probability as $\delta\downarrow0$. We first compute the second probability on the right-hand-side. 
	\be \label{eq:emptyppp}
	\P{\Pi(E_\delta^\xi)=0}=\exp\bigg\{-\int_{E_\delta^\xi}(\alpha-1) y^{-\alpha}\d y\d t\bigg\}=\exp\bigg\{-\eps\delta^{-(\alpha-1)(2-\alpha)/2}(1+\delta^{-\xi})^{-(\alpha-1)}\bigg\}.
	\ee 
	Note that, by the choice of $\xi$, this probability tends to zero with $\delta$. Now, we bound the conditional probability in \eqref{eq:Tconditional}.
	\be\ba \label{eq:tepsdeltatoinp}
	\mathbb{P}&(|T^\eps_\delta(\Pi)-T^\eps(\Pi)|\geq \eta\,|\, \Pi(E_\delta^\xi)\neq0)\\
	&=\P{\bigg| \int_\eps^1\bigg(\int_{E_\delta} f\ind_{\{t\leq s\}}\d\Pi(t,f)\bigg )^{-1}-\bigg(\int_{E} f\ind_{\{t\leq s\}}\d \Pi(t,f)\bigg)^{-1}\d s \bigg|\geq \eta\,\bigg|\, \Pi(E_\delta^\xi)\neq0}\\
	&\leq\P{\int_\eps^1\bigg(\int_{E\backslash E_\delta}f\ind_{\{t\leq s\}}\d\Pi(t,f)\bigg)\Big/ \bigg(\int_{E_\delta} f\ind_{\{t\leq s\}}\d \Pi(t,f) \bigg)^2\d s\geq \eta\,\bigg|\, \Pi(E_\delta^\xi)\neq0}\\
	&\leq \P{\int_{E\backslash E_\delta}f \d\Pi(t,f)\geq \frac{\eta}{1-\eps}\bigg(\int_{E_\delta} f\ind_{\{t\leq \eps\}}\d \Pi(t,f)\bigg)^2\,\bigg|\, \Pi(E_\delta^\xi)\neq0},
	\ea\ee 
	where, in the last line, we replaced the integration variable $s$ with $1$ in the integral in the numerator and with $\eps$ in the integral the denominator. We now bound the integral over $E_\delta$ on the right-hand-side from below using $\Pi(E^\xi_\delta)\geq 1$ and use Markov's inequality to find the upper bound
	\be \ba\label{eq:Tdeltaprobbound}
	\mathbb{P}&\bigg(\int_{E\backslash E_\delta}f \d\Pi(t,f)\geq \frac{\eta}{1-\eps}\delta^{2-\alpha}(1+\delta^{-\xi})^2\,\bigg|\, \Pi(E_\delta^\xi)\neq0\bigg)\\
	&=\P{\int_{E\backslash E_\delta}f \d\Pi(t,f)\geq \frac{\eta}{1-\eps}\delta^{2-\alpha}(1+\delta^{-\xi})^2}\\
	&\leq \E{\int_{E\backslash E_\delta}f d\Pi(t,f)}\frac{1-\eps}{\eta}\delta^{-(2-\alpha)}(1+\delta^{-\xi})^{-2}\\
	&=\int_{E\backslash E_\delta}(\alpha-1)x^{1-\alpha} \d t\, \d x\frac{1-\eps}{\eta}\delta^{-(2-\alpha)}(1+\delta^{-\xi})^{-2}= \frac{(1-\eps)(\alpha-1)}{\eta(2-\alpha)}(1+\delta^{-\xi})^{-2},
	\ea \ee 
	which tends to zero as $\delta \downarrow 0$. Note that we can omit the conditional statement in the second line, as the integral is independent of $\Pi(E^\xi_\delta)$. Combining \eqref{eq:emptyppp} and the upper bound of \eqref{eq:Tdeltaprobbound} in \eqref{eq:Tconditional}, implies that $T^\eps_\delta(\Pi)\toinp T^\eps(\Pi)$ as $\delta\downarrow 0$. We now prove the second statement in \eqref{eq:Tdeltabehave}, which uses a similar approach. Namely, using analogous steps as in \eqref{eq:Tconditional}, \eqref{eq:tepsdeltatoinp} and \eqref{eq:Tdeltaprobbound}, we obtain
	\be \ba\label{eq:Tdeltapinconv}
	\mathbb{P}&(|T^\eps(\Pi_n)-T^\eps_\delta(\Pi_n)|\geq \eta )\\
	&\leq \P{\int_{E\backslash E_\delta}f \d\Pi_n(t,f)\geq \frac{\eta}{1-\eps}\delta^{2-\alpha}(1+\delta^{-\xi})^{-2}}+\P{\Pi_n(E^\xi_\delta)=0}.
	\ea\ee 
	The second probability on the right-hand-side converges to $\mathbb{P}(\Pi(E^\xi_\delta)=0)$ as $n$ tends to infinity, and then to zero as $\delta$ tends to zero by \eqref{eq:emptyppp}. Using Markov's inequality, we obtain an upper bound for the first probability on the right-hand-side of the form
	\be\ba
	\sum_{i=1}^n& \ \E{\F_i/u_n \ind_{\{\F_i/u_n\leq \delta\}}}\frac{1-\eps}{\eta}\delta^{-(2-\alpha)}(1+\delta^{-\xi})^2\\
	&=\frac{1-\eps}{\eta}\delta^{-(2-\alpha)}(1+\delta^{-\xi})^2\frac{n}{u_n}\int_{x_\ell}^{\delta u_n} \ell(x)x^{-(\alpha-1)}\d x,
	\ea \ee 
	where $x_\ell:=\inf\{x\in \R:F_\F(x)>0\}$. Using \cite[Proposition 1.5.8]{BinGolTeu87}, yields 
	\be 
	\int_{x_\ell}^{\delta u_n} \ell(x)x^{-(\alpha-1)}\d x \sim \frac{1}{2-\alpha}(\delta u_n)^{2-\alpha}\ell(\delta u_n), \text{ as $n\rightarrow \infty$.}
	\ee 
	Thus, as $n\rightarrow\infty$, since $\ell$ is slowly-varying,
	\be
	\frac{1-\eps}{\eta}\delta^{-(2-\alpha)}(1+\delta^{-\xi})^2\frac{n}{u_n}\int_{x_\ell}^{\delta u_n} \ell(x)x^{-(\alpha-1)}\d x \sim \frac{(1-\eps)}{\eta(2-\alpha)}(1+\delta^{-\xi})^{-2}n\ell(u_n)u_n^{-(\alpha-1)}.
	\ee 
	Using \cite[Corollary 4.19 and Proposition 3.21]{Res13}, we conclude that $n\ell(u_n)u_n^{-(\alpha-1)}$ converges to $1$ and so the right-hand-side tends to zero with $\delta$. Thus,
	\be \label{eq:Tdeltapinlim}
	\lim_{\delta\downarrow0}\lim_{n\rightarrow\infty}\P{|T^\eps(\Pi_n)-T^\eps_\delta(\Pi_n)|\geq \eta}=0,
	\ee
	which finishes the proof.
\end{proof}

We now prove Proposition \ref{prop:infmeanmaxconv}.

\begin{proof}[Proof of Proposition \ref{prop:infmeanmaxconv}]
	
	For a closed set $C\subseteq \R_+$ and $\eta>0$, let  $C_\eta:=\{x\in\R:\inf_{y\in C}|x-y|\leq \eta\}$ be the $\eta$-enlargement of $C$ and let us define the events
	\be \ba\label{eq:eepsdelta}
	E_{n,\eps,\delta}(\eta)&:=\Big\{\Big|\max_{\inn}\frac{\F_i}{u_n}T^{i/n}(\Pi_n)-\max_{\eps n\leq i\leq n: \F_i\geq \delta u_n}\frac{\F_i}{u_n}T_\delta^{i/n}(\Pi_n)\Big|<\eta\Big\},\\ F_{n,\eps,\delta}&:=\{\Pi_n((0,\eps)\times(\delta,\infty))\geq 1\}.
	\ea\ee 
	We can then write
	\be \ba\label{eq:mappingtoindis}
	\mathbb{P}\Big(\max_{\inn}\frac{\F_i}{u_n}T^{i/n}(\Pi_n)\in C\Big)\leq{}& \P{\Big\{\max_{\inn}\frac{\F_i}{u_n}T^{i/n}(\Pi_n)\in C\Big\}\cap E_{n,\eps,\delta}(\eta) \cap F_{n,\eps,\delta}}\\
	&+\P{E_{n,\eps,\delta}(\eta)^c}+\P{F_{n,\eps,\delta}^c}.
	\ea\ee 
	Then, on $E_{n,\eps,\delta}(\eta)$ and using $C_\eta$, we can bound the first probability on the right-hand-side from above by 
	\be
	\P{\Big\{\max_{\eps n\leq i\leq n:\F_i\geq \delta u_n}\frac{\F_i}{u_n}T_\delta^{i/n}(\Pi_n)\in C_\eta\Big\}\cap F_{n,\eps,\delta}}.
	\ee 
	We note that every term in the maximum is bounded from above by $1$. Then, since for $n$ large $\Pi_n((\eps,1)\times(\delta,\infty))=\Pi((\eps,1)\times(\delta,\infty))<\infty$ and on $F_{n,\eps,\delta}$, it follows from the continuous mapping theorem, Lemma \ref{lemma:cont} and Remark \ref{rem:cont} that
	\be\ba \label{eq:maxtcontmap}
	\lim_{n\to\infty}&\P{\Big\{\max_{\eps n\leq i\leq n:\F_i\geq \delta u_n}\frac{\F_i}{u_n}T_\delta^{i/n}(\Pi_n)\in C_\eta\Big\}\cap F_{n,\eps,\delta}}\\
	&=\mathbb{P}\bigg(\Big\{\sup_{(t,f)\in\Pi: t\geq \eps,f\geq \delta}f T_\delta^{t}(\Pi)\in C_\eta\Big\}\cap F_{\eps,\delta}\bigg),
	\ea\ee 
	where $F_{\eps,\delta}:=\{\Pi((0,\eps)\times(\delta,\infty))\geq 1\}$. We now claim that it is possible to remove the $\delta$ in $T^\eps_\delta(\Pi)$ and the $\delta$ and $\eps$ constraints in the supremum in \eqref{eq:maxtcontmap}, as well as that the two terms in the last line of \eqref{eq:mappingtoindis} tend to zero when letting $n$ tend to infinity, and then $\delta$ and $\eps$ to zero. These two tasks require a very similar approach, as they are essentially the same, one with $\Pi_n$ and the other with its weak limit $\Pi$. We start with the latter claim. We want to show that
	\be \label{eq:loseconstraint}
	\Big|\sup_{(t,f)\in\Pi: t\geq \eps,f\geq \delta}f T_\delta^{t}(\Pi)-\sup_{(t,f)\in\Pi}f T^{t}(\Pi)\Big|\toinp 0\text{ as first }\delta\downarrow0\text{ and then }\eps\downarrow 0.
	\ee 
	To this end, we write 
	\be \ba
	\Big|\sup_{(t,f)\in\Pi: t\geq \eps,f\geq \delta}f T_\delta^{t}(\Pi)-\sup_{(t,f)\in\Pi}f T^{t}(\Pi)\Big|\leq{}& \Big|\sup_{(t,f)\in\Pi: t\geq \eps,f\geq \delta}f T_\delta^{t}(\Pi)-\sup_{(t,f)\in\Pi: t\geq \eps}f T^{t}(\Pi)\Big|\\
	&+\Big|\sup_{(t,f)\in\Pi: t\geq \eps}f T^{t}(\Pi)-\sup_{(t,f)\in\Pi}f T^{t}(\Pi)\Big|\\
	=:{}&D_1+D_2.
	\ea \ee 
	We first prove $D_1$ tends to zero in probability as $\delta\downarrow 0$. Namely, using the triangle inequality and the definitions of $T^\eps_\delta$ and $T^\eps$ in \eqref{eq:Top},
	\be \ba\label{eq:d1bound}
	D_1\leq{}& \Big|\sup_{(t,f)\in\Pi: t\geq \eps,f\geq \delta}f T_\delta^{t}(\Pi)-\sup_{(t,f)\in\Pi: t\geq \eps,f\geq \delta}f T^{t}(\Pi)\Big|\\
	&+\Big|\sup_{(t,f)\in\Pi: t\geq \eps,f\geq \delta}f T^{t}(\Pi)-\sup_{(t,f)\in\Pi: t\geq \eps}f T^{t}(\Pi)\Big|\\
	\leq{}& \sup_{(t,f)\in\Pi: t\geq \eps,f\geq \delta} f (T^t_\delta(\Pi)-T^t(\Pi))+\sup_{(t,f)\in\Pi:t\geq \eps,f<\delta} fT^t(\Pi)\\
	\leq{}& \Big(\sup_{(t,f)\in\Pi}f\Big) \sup_{(t,f)\in\Pi: t\geq \eps}(T^t_\delta(\Pi)-T^t(\Pi))+\delta T^\eps(\Pi)\\
	\leq{}&\Big(\sup_{(t,f)\in\Pi}f\Big) (T^\eps_\delta(\Pi)-T^\eps(\Pi))+\delta T^\eps(\Pi),
	\ea\ee 
	where the final inequality follows from the definitions of $T^\eps$ and $T^\eps_\delta$. Since $\alpha>1$, $\sup_{(t,f)\in\Pi}f<\infty$ almost surely. Furthermore, for any $\eps>0$ fixed, $T^\eps(\Pi)<\infty$ almost surely as well. Finally, by Lemma \ref{lemma:Tdeltaconv}, $(T^\eps_\delta(\Pi)-T^\eps(\Pi))\toinp 0$ as $\delta\downarrow 0$. Thus, we obtain that $D_1\toinp 0$ as $\delta\downarrow0$. We now show that $D_2\toas 0$ as $\eps\downarrow0$. We discretise the interval $(0,1)$ into smaller sub-intervals $[2^{-(k+1)},2^{-k}),k\geq0$. Then, 
	\be \label{eq:d2lim}
	\lim_{\eps\downarrow  0} D_2\leq \lim_{\eps\downarrow0} \sup_{(t,f)\in\Pi: t<\eps}fT^t(\Pi)=\lim_{K\rightarrow\infty}\sup_{k\geq K}\sup_{(t,f)\in\Pi:t\in[2^{-(k+1)},2^{-k})} fT^t(\Pi).
	\ee 
	We now bound the inner supremum, by controlling the size of the maximum fitness value in these sub-intervals. That is, we define, for $\xi>0,k\in\Z^+$,
	\be \ba\label{eq:lkhk}
	\ell_k&:=2^{-(k+1)/(\alpha-1)}\log((k+2)^{1+\xi})^{-1/(\alpha-1)},\\ h_k&:=2^{-(k+1)/(\alpha-1)}\log((1-(k+2)^{-(1+\xi)})^{-1})^{-1/(\alpha-1)}.
	\ea \ee 
	Now,
	\be \ba\label{eq:fbounds}
	\P{\Pi(\intk\times(h_k,\infty))\neq0}&=1-\exp\Big\{-\int_{2^{-(k+1)}}^{2^{-k}}\int_{h_k}^\infty(\alpha-1)x^{-\alpha}\d x\d t\Big\}\\
	&=1-\exp\{\log((1-(k+2)^{-(1+\xi)})\}\leq k^{-(1+\xi)},\\
	\P{\Pi(\intk\times(\ell_k,\infty))=0}&=\exp\Big\{-\int_{2^{-(k+1)}}^{2^{-k}}\int_{\ell_k}^\infty(\alpha-1)x^{-\alpha}\d x\d t\Big\}\leq k^{-(1+\xi)},
	\ea \ee 
	which are both summable. Therefore, by the Borel-Cantelli lemma, it follows that almost surely there exist a random index $L$, such that for all $k\geq L$,
	\be \label{eq:fitsup}
	\sup_{(t,f)\in\Pi:t\in\intk}f \in (\ell_k,h_k).
	\ee 
	Now, on the event $\{t\leq 2^{-L}\}$,
	\be \ba\label{eq:Tbound}
	T^t(\Pi)&=\int_t^1 \Big(\int_E f\ind_{\{u\leq s\}}\d\Pi(u,f)\Big)^{-1}\d s \\
	&= \int_t^{2^{-L}} \Big(\int_E f\ind_{\{u\leq s\}}\d\Pi(u,f)\Big)^{-1}\d s+\int_{2^{-L}}^1 \Big(\int_E f\ind_{\{u\leq s\}}\d\Pi(u,f)\Big)^{-1}\d s\\
	&\leq \int_t^{2^{-L}} (\sup_{(u,f)\in\Pi:u\leq s} f)^{-1}\d s + \Big(\int_E f\ind_{\{u\leq 2^{-L}\}}\d\Pi(u,f)\Big)^{-1}
	\ea \ee 
	By applying \eqref{eq:fitsup} to the both integrals, we find an upper bound
	\be
	\sum_{j=L}^{\lceil\log_2(1/t)\rceil} 2^{-(j+2)}\ell_{j+1}^{-1} + \ell_L^{-1}.
	\ee
	Using the definition of $\ell_j$ in \eqref{eq:lkhk}, for $j$ large and some $\zeta\in(0,\alpha-1)$, we obtain
	\be \ba\label{eq:Ttbound}
	T^t(\Pi)&\leq C\sum_{j=L}^{\lceil\log_2(1/t)\rceil} 2^{(j+1)((1+\zeta)/(\alpha-1)-1)} + \ell_L^{-1}\\
	&\leq  \wt C t^{1-(1+\zeta)/(\alpha-1)}+\ell^{-1}_L,
	\ea\ee 
	for some constant $\wt C>0$. Again using \eqref{eq:fitsup} and on $\{k>L\}$ (similar to $t\leq 2^{-L}$), we find
	\be \ba
	\sup_{(t,f)\in\Pi: t\in[2^{-(k+1)},2^{-k})} fT^t(\Pi) &\leq h_k(\wt C 2^{(k+1)((1+\zeta)/(\alpha-1)-1)}+\ell_L^{-1})\leq \wt C 2^{(k+1)(\zeta/(\alpha-1)-1)}k^\gamma+h_k\ell_L^{-1},
    \ea	\ee 
    for some $\gamma>(1+\xi)/(\alpha-1)$. We finish the argument by noting that $L<\infty$ almost surely and hence
	\be \ba\label{eq:limfT}
	\lim_{K\rightarrow\infty}\sup_{k\geq K}\sup_{(t,f)\in\Pi:t\in\intk}f T^t(\Pi)&\leq \lim_{K\rightarrow\infty}\sup_{k\geq K} \wt C 2^{(k+1)(\zeta/(\alpha-1)-1)}k^\gamma+h_k\ell_L^{-1}\\
	&=\lim_{K\rightarrow\infty}  \wt C 2^{(K+1)(\zeta/(\alpha-1)-1)}K^\gamma+h_K\ell_L^{-1}=0,
	\ea\ee 
	by the choice of $\zeta$. Thus, $D_2\toas 0$ as $\eps\downarrow 0$. Together with the convergence of $D_1$ to zero in probability, we obtain \eqref{eq:loseconstraint}. Recall $F_{n,\eps,\delta}$ from \eqref{eq:eepsdelta} and $F_{\eps,\delta}=\lim_{n\rightarrow \infty}F_{n,\eps,\delta}$ under \eqref{eq:maxtcontmap}. Evidently, by a similar argument as in \eqref{eq:emptyppp}, $\lim_{\delta\downarrow 0}\P{F_{\eps,\delta}}=1$ for all $\eps\in(0,1)$, which also shows the third probability in \eqref{eq:mappingtoindis} tends to zero as $n\to\infty$ and then $\delta\downarrow 0$. Combining this with \eqref{eq:loseconstraint} and \eqref{eq:maxtcontmap} yields
	\be\label{eq:contmap}
	\lim_{\eps\downarrow0}\lim_{\delta\downarrow0}\lim_{n\to\infty} \mathbb{P}\bigg(\Big\{\max_{\eps n\leq i\leq n:\F_i\geq \delta u_n}\frac{\F_i}{u_n}T_\delta^{i/n}(\Pi_n)\in C_\eta\Big\}\cap F_{n,\eps,\delta}\bigg) =\mathbb{P}\bigg(\sup_{(t,f)\in\Pi}fT^t(\Pi)\in C_\eta \bigg).
	\ee 
    Recall $E_{n,\eps,\delta}(\eta)$ from \eqref{eq:eepsdelta}. What remains to prove, is that for all $\eta>0$ fixed,
	\be 
	\lim_{\eps\downarrow 0}\lim_{\delta\downarrow 0}\lim_{n\rightarrow\infty}\P{E_{n,\eps,\delta}(\eta)^c}=0,
	\ee 
	which is very similar to \eqref{eq:loseconstraint}, though we now deal with $\Pi_n$ rather than $\Pi$. Again, we use the triangle inequality to find
	\be \ba\label{eq:p1p2}
	\P{E_{n,\eps,\delta}(\eta)^c}\leq{}&\P{\Big|\max_{\eps n\leq i\leq n: \F_i\geq \delta u_n}\frac{\F_i}{u_n}T_\delta^{i/n}(\Pi_n)-\max_{\eps n\leq i\leq n}\frac{\F_i}{u_n}T^{i/n}(\Pi_n)\Big|\geq \eta/2}\\
	&+\P{\Big|\max_{\eps n\leq i\leq n}\frac{\F_i}{u_n}T^{i/n}(\Pi_n)-\max_{ \inn}\frac{\F_i}{u_n}T^{i/n}(\Pi_n)\Big|\geq \eta/2}\\
	=:{}&P_1+P_2.
	\ea\ee 
	We first deal with $P_1$. As in \eqref{eq:d1bound}, we split this into two terms, namely
	\be \ba\label{eq:p1bound}
	P_1\leq{}& \P{\Big|\max_{\eps n\leq i\leq n: \F_i\geq \delta u_n}\frac{\F_i}{u_n}T_\delta^{i/n}(\Pi_n)-\max_{\eps n\leq i\leq n}\frac{\F_i}{u_n}T_\delta^{i/n}(\Pi_n)\Big|\geq \eta/4}\\
	&+\P{\Big|\max_{\eps n\leq i\leq n}\frac{\F_i}{u_n}T_\delta^{i/n}(\Pi_n)-\max_{\eps n\leq i\leq n}\frac{\F_i}{u_n}T^{i/n}(\Pi_n)\Big|\geq \eta/4}.
	\ea\ee 
	To show the first probability tends to zero, we write
	\be 
	\Big|\max_{\eps n\leq i\leq n: \F_i\geq \delta u_n}\frac{\F_i}{u_n}T_\delta^{i/n}(\Pi_n)-\max_{\eps n\leq i\leq n}\frac{\F_i}{u_n}T_\delta^{i/n}(\Pi_n)\Big|\leq \delta \max_{\eps n\leq i\leq n} T^{i/n}_\delta(\Pi_n)\leq \delta T^\eps_\delta(\Pi_n).
	\ee 
	Then, on $F_{n,\eps,\delta}$, $T^\eps_\delta(\Pi_n)$ converges in distribution to $\delta T^\eps_\delta(\Pi)$ by the continuous mapping theorem and the fact that $T^\eps_\delta$ is continuous in $\Pi_n$, as follows from the proof of Lemma \ref{lemma:cont} and remark \ref{rem:cont}. So, as $\delta\downarrow 0$, $T^\eps_\delta(\Pi)\toinp T^\eps(\Pi)$, as follows from the proof of Lemma \ref{lemma:Tdeltaconv}, which implies that $\delta T^\eps_\delta(\Pi)\toinp 0$ as $\delta\downarrow 0$. As before, $\P{F_{n,\eps,\delta}}\to 1$ as $n\to\infty$ and then $\delta\downarrow 0$, so by intersecting the first probability on the right-hand-side of \eqref{eq:p1bound} with $F_{n,\eps,\delta},F_{n,\eps,\delta}^c$, as in \eqref{eq:mappingtoindis}, yields that it tends to zero as $n\to\infty$ and then $\delta\downarrow 0$.  What remains is to show that the second probability on the right-hand-side of \eqref{eq:p1bound} tends to zero as $n$ tends to infinity, then $\delta\downarrow 0$ and finally $\eps \downarrow0$. We again use a similar argument as in \eqref{eq:d1bound} to find
	\be \ba\label{eq:maxftbound}
	\Big|\max_{\eps n\leq i\leq n}\frac{\F_i}{u_n}T_\delta^{i/n}(\Pi_n)-\max_{\eps n\leq i\leq n}\frac{\F_i}{u_n}T^{i/n}(\Pi_n)\Big|&\leq \Big(\max_{\inn} \frac{\F_i}{u_n}\Big)\max_{\eps n\leq i \leq n}(T^{i/n}_\delta(\Pi_n)-T^{i/n}(\Pi_n))\\
	&\leq \Big(\max_{\inn} \frac{\F_i}{u_n}\Big)(T^{\eps}_\delta(\Pi_n)-T^{\eps}(\Pi_n)).
	\ea\ee 
	We show that the product of the maximum and $(T^{\eps}_\delta(\Pi_n)-T^{\eps}(\Pi_n))$ converges to zero in probability as first $n\rightarrow\infty$ and then $ \delta\downarrow0$. We can use the fact that $(T^{\eps}_\delta(\Pi_n)-T^{\eps}(\Pi_n))$ tends to zero in probability as $n\rightarrow\infty$ and then $\delta\downarrow0$, as is shown in the proof of Lemma \ref{lemma:Tdeltaconv}. In order to extend this result to the product of these two random processes, we introduce the events $A_{n,\delta}:=\{\max_{\inn}\F_i/u_n\leq \delta^{-\xi}\}$, for some $\xi\in(0,(2-\alpha)/2)$. Then, splitting the second probability on the right-hand-side of \eqref{eq:p1bound} into two parts by using \eqref{eq:maxftbound} and intersecting with the events $A_{n,\delta}$ and $A_{n,\delta}^c$, we obtain the upper bound
	\be \ba
	\P{\Big|\max_{\eps n\leq i\leq n}\frac{\F_i}{u_n}T_\delta^{i/n}(\Pi_n)-\max_{\eps n\leq i\leq n}\frac{\F_i}{u_n}T^{i/n}(\Pi_n)\Big|\geq \eta/4}\leq{}& \P{T^\eps_\delta(\Pi_n)-T^\eps(\Pi_n) \geq \eta \delta^\xi /4}\\
	&+\P{A_{n,\delta}^c}.
	\ea\ee 
	$\mathbb{P}(A_{n,\delta}^c)$ converges to $\mathbb{P}(A^c_\delta)$, where $A_\delta:=\{Y\leq \delta^{-\xi}\}$ and $Y$ is the distributional limit of \\$\max_{\inn}\F_i/u_n$. Then, as $\delta\downarrow0$, $\P{A_\delta^c}\rightarrow0$, as $Y$ is almost surely finite. Following the steps of the argument in \eqref{eq:Tdeltapinconv} through \eqref{eq:Tdeltapinlim} with $\eta \delta^\xi/4$ instead of $\eta$, we find
	\be\ba
	\limsup_{n\rightarrow\infty}\mathbb{P}(|T^\eps(\Pi_n)-T^\eps_\delta(\Pi_n)|\geq \eta \delta^\xi/4)&\leq \frac{4(1-\eps)}{\eta(\alpha-2)}\delta^{-\xi}(1+\delta^{-\xi})^{-2}+\limsup_{n\rightarrow\infty}\mathbb{P}(\Pi_n(E^\xi_\delta)=0)\\
	&=\frac{4(1-\eps)}{\eta(\alpha-2)}\delta^{-\xi}(1+\delta^{-\xi})^{-2}+\mathbb{P}(\Pi(E^\xi_\delta)=0),
	\ea\ee 
	which tends to zero as $\delta\downarrow0$. It thus follows that $P_1\to0$ as $n\rightarrow \infty$ and then $\delta\downarrow0$. 

	What remains, is to show that $P_2$ tends to zero as $n\rightarrow\infty,\eps\downarrow0$. This follows from a similar approach as in \eqref{eq:d2lim} through \eqref{eq:limfT}. Recall $\ell_k,h_k$ from \eqref{eq:lkhk}. We then divide the set of indices $\inn$ into subsets $A_{k,n}:=\{\inn: i\in(2^{-(k+1)}n,2^{-k}n]\},0\leq k \leq \lfloor \log n/\log 2\rfloor,$ and define the events $\mathcal{A}^\F_{k,n}:=\big\{\max_{i\in A_{k,n}}\F_i/u_n \in (\ell_k,h_k)\big\}$. Using \eqref{eq:fbounds}, it readily follows that
	\be 
	\liminf_{n\to\infty}\P{\mathcal{A}^\F_{k,n}}\geq 1-2k^{-(1+\xi)}.
	\ee 
	Hence, when letting $k_n:=\lfloor\log n/\log 2\rfloor$, 
	\be \label{eq:afknbound}
	\liminf_{n\to\infty}\mathbb{P}\bigg(\bigcap_{K\leq k \leq k_n} \mathcal{A}^\F_{k,n}\bigg) \geq 1-CK^{-\xi},
	\ee 
	for some constant $C>0$, independent of $K$. Similar to \eqref{eq:d2lim}, we write
	\be
	\limsup_{\eps\downarrow 0}\limsup_{n\rightarrow\infty}P_2= \limsup_{K\rightarrow\infty}\limsup_{n\rightarrow\infty}\mathbb{P}\bigg(\sup_{k\geq K}\sup_{i\in A_{k,n}} \frac{\F_i}{u_n}T^{i/n}(\Pi_n)\geq \eta/4\bigg).
	\ee
	Now, by intersecting with a similar event to the one in \eqref{eq:afknbound}, we find the upper bound
	\be \ba\label{eq:supsup}
	 \limsup_{K\rightarrow\infty}\limsup_{n\rightarrow\infty}\mathbb{P}&\bigg(\Big\{\sup_{k\geq K}\sup_{i\in A_{k,n}} \frac{\F_i}{u_n}T^{i/n}(\Pi_n)\geq \eta/4\Big\}\cap  \Big(\bigcap_{\sqrt K\leq k\leq k_n} A^\F_{k,n}\Big)\bigg)\\
	 &+\mathbb{P}\bigg(\bigcup_{\sqrt K\leq k \leq k_n}\hspace{-6pt} \Big(A^\F_{k,n}\Big)^c\bigg).
	\ea\ee  
	By \eqref{eq:afknbound}, it follows that the double limit of the second probability equals zero, so we focus on the first probability. Following the approach in \eqref{eq:Tbound} and \eqref{eq:Ttbound} and using a Markov bound, we bound the first probability on the right-hand-side of \eqref{eq:supsup} from above by
	\be \ba\label{eq:pinbound}
	\frac{4}{\eta}&\mathbb{E}\Big[\sup_{k\geq K}\sup_{i\in A_{k,n}}\frac{\F_i}{u_n}T^{i/n}(\Pi_n) \ind_{\cap_{\sqrt K\leq k \leq k_n}A_{k,n}^\F}\Big]\\
	&\leq \frac{4}{\eta}\mathbb{E}\bigg[\sup_{k\geq K}\sup_{i\in A_{k,n}}  \frac{h_k}{n}\bigg(\sum_{j=i}^{2^{-\sqrt K}n} \big(M_j/u_n\big)^{-1}+\sum_{j=2^{-\sqrt K}n}^n\big(M_j/u_n\big)^{-1} \bigg)\ind_{\cap_{\sqrt K\leq k \leq k_n}A_{k,n}^\F}\bigg],
	\ea\ee
	where we recall that $M_j:=\max_{m\leq j}\F_m$. We then bound the maximum in the second sum from below by considering only the indices $m\leq 2^{-\sqrt{K}}n$ and using the events in the indicator to further bound the maximum from below by $\ell_{\sqrt{K}}$. The terms of the second sum then are independent of $j$, which yields the upper bound $n(\ell_{\sqrt{K}})^{-1}$. We rewrite the first sum, where we note that for $i\in A_{k,n},i\geq 2^{-(k+1)}n$, and as before bound the maximum from below to find 
	\be 
	\sum_{j=i}^{2^{-\sqrt{K}}n}(M_j/u_n)^{-1}\leq \sum_{j\geq \sqrt{K}}^{k+1}\sum_{p\in A_{j,n}}(\ell_{j+1})^{-1}\leq n\sum_{j\geq \sqrt{K}}^{k+1} 2^{-(j+1)}(\ell_{j+1})^{-1}.
	\ee 
	Since, for large $j$, we can bound $(\ell_j)^{-1}$ from above by $2^{j(1/(\alpha-1)+\zeta)}$ for some small $\zeta$, we obtain the upper bound $Cn 2^{(k+1)((2-\alpha)/(\alpha-1)+\zeta)}$, for some constant $C>0$. Note that this upper bound, as well as the upper bound stated above for the second sum in \eqref{eq:pinbound} are deterministic. Hence, using both upper bounds and bounding the indicator in the expectation in \eqref{eq:pinbound} from above by $1$ yields the upper bound
	\be \ba
	C_\eta \sup_{k\geq K}\sup_{i\in A_{k,n}}( h_k 2^{(k+1)((2-\alpha)/(\alpha-1)+\zeta)}+(\ell_{\sqrt{K}})^{-1}h_k)&\leq C_\eta \sup_{k\geq K} 2^{-(k+1)(1-\zeta)}k^\gamma+\ell_{\sqrt K}^{-1}h_k\\
	&=C_\eta 2^{-(K+1)(1-\zeta)}K^\gamma+\ell_{\sqrt K}^{-1}h_K,
	\ea \ee 
	for some $\gamma>0$ and where $C_\eta=(4/\eta)\max\{C,1\}$. This bound no longer depends on $n$, and as we let $K$ tend to infinity the bound tends to zero. This proves $P_2$ tends to zero with $n\to\infty$ and then $\epsilon\downarrow 0$. Combining this result with the convergence of $P_1$ to zero with $n\to\infty$ and then $\delta\downarrow 0$, it follows that the upper bound in \eqref{eq:p1p2} tends to zero, and therefore the two probabilities on the second line of the right-hand-side of \eqref{eq:mappingtoindis} tend to zero with $n\to\infty$, then $\delta\downarrow 0$ and finally $\eps\downarrow 0$. Together with \eqref{eq:contmap}, this yields
	\be 
	\limsup_{n\rightarrow\infty}\P{\max_{\inn}\frac{\F_i}{u_n}T^{i/n}(\Pi_n)\in C}\leq \P{\sup_{(t,f)\in\Pi}fT^t(\Pi)\in C_\eta}.
	\ee 
	Including the limit $\eta\downarrow 0$ finally yields, by the continuity of the probability measure,
	\be 
	\limsup_{n\rightarrow\infty}\mathbb{P}\bigg(\max_{\inn}\frac{\F_i}{u_n}T^{i/n}(\Pi_n)\in C\bigg)\leq \mathbb{P}\bigg(\sup_{(t,f)\in\Pi}fT^t(\Pi)\in C\bigg),
	\ee 
	and applying the Portmanteau lemma \cite[Theorem 13.16]{KleAch13} finishes the proof.
\end{proof} 

	\section{Martingales and concentration}\label{sec:ppp}

	In this section we state and prove several important results, required for the proof of Theorem \ref{Thrm:maxdegree}. As discussed in the overview of the proof of Theorem \ref{Thrm:maxdegree} in Section \ref{sec:overview}, for $\alpha\in(1,2)\cup(2,1+\theta_m)$ the approach to proving Theorem \ref{Thrm:maxdegree} is by showing the maximum degree is concentrated around the maximum of the conditional moments of the degrees and by showing that the latter converges to the right-hand-side of \eqref{eq:ppplimit} and \eqref{eq:infmeanppplimit} when $\alpha\in(1,2),(2,1+\theta_m)$, respectively. To this end, we formulate the following propositions:
	
	\begin{proposition}\label{lemma:condmeanconv} Consider the three PAF models as in Definition \ref{def:paf}. Let $\Pi$ be a Poisson Point Process (PPP) on $(0,1)\times (0,\infty)$ with intensity measure $\nu(\d t,\d x):=\d t\times(\alpha-1) x^{-\alpha}\d x$, and let $\theta_m:=1+\E{\F}/m$. Then, for $\alpha\in(2,1+\theta_m)$,
		\be\label{eq:condmeanconv}
		\max_{i\in[n]}\Ef{}{\zni/u_n}\toindis \max_{(t,f)\in\Pi}f(t^{-1/\theta_m}-1),
		\ee
		while for $\alpha\in (1,2)$,
		\be\label{eq:infcondmeanconv}
		\max_{i\in[n]}\Ef{}{\zni/n}\toindis m\max_{(t,f)\in\Pi} f \int_t^1\bigg(\int_E g\ind_{\{u\leq s\}}\d \Pi(u,g)\bigg)^{-1}\d s.
		\ee 
	\end{proposition} 
	
	\begin{proposition}\label{lemma:concentration}
		Consider the three PAF models as in Definition \ref{def:paf}. When $\alpha\in(2,1+\theta_m)$, for any $\eta>0$,
		\be \label{eq:conc}
		\lim_{n\rightarrow \infty}\P{\Big|\max_{i\in[n]}\zni-\max_{i\in[n]}\Ef{}{\zni}\Big|>\eta u_n}=0.
		\ee 
		Similarly, when $\alpha\in(1,2)$, for any $\eta>0$,
		\be \label{eq:infmeanconc}
		\lim_{n\rightarrow \infty}\P{\Big|\max_{i\in[n]}\zni-\max_{i\in[n]}\Ef{}{\zni}\Big|>\eta n}=0.
		\ee 
	\end{proposition}

	Before we prove these two propositions, we introduce a family of useful martingales and derive some of their properties. We define, for $k\in\R,n,n_0,m,m_0\in\N$ and $a,b>-1$ such that $a-b>-1$,
	\be \ba\label{eq:c1n}
	c^k_n(m)&:=\prod_{j=n_0}^{n-1}\prod_{\ell=1}^m \Big(1-\frac{k}{m_0+m(j-n_0)+k+(\ell-1)+S_j}\Big), \\ 
	\wt c^k_n(m)&:=\prod_{j=n_0}^{n-1} \Big(1-\frac{k}{m_0+m(j-n_0)+k+S_j}\Big)^m, \qquad {a \choose b}:=\frac{\Gamma(a+1)}{\Gamma(b+1)\Gamma(a-b+1)}.
	\ea\ee 
	For ease of writing, we omit the $(m)$ in $c^k_n(m),\wt c^k_n(m)$ whenever there is no ambiguity. We can then formulate the following lemma:
	
	\begin{lemma}\label{lemma:martingale}
	Let $i\in\N,k\geq -\min(\F_i,1)$. For the PAFRO model ($m=1$) and the PAFUD model with out-degree $m\in\N$, the random variable
	\be 
	M^k_n(i):=c_n^k(m){\zni+\F_i+(k-1) \choose k}
	\ee
	is a martingale with respect to $\G_{n-1}$ for $n\geq i\vee n_0$, under the conditional probability measure $\Pf{\cdot}$. For the PAFFD model with out-degree $m\in\N$, the random variable
	\be 
	\wt M^k_n(i):=\wt c_n^k(m){\zni+\F_i+(k-1) \choose k}
	\ee 
	is a supermartingale (resp.\ submartingale) with respect to $\G_{n-1}$ for $n\geq i\vee n_0$, under the conditional probability measure $\Pf{\cdot}$ when $k\geq 0$ (resp. $k\in(-\min(\F_i,1),0)$. Finally, for the PAFFD model, $M^1_n(i)$ is a martingale with respect to $\G_{n-1}$ for $n\geq i\vee n_0$ under the conditional probability measure $\Pf{\cdot}$.
	\end{lemma}

	\begin{proof}
	For ease of writing, let us define $X_n(i):=\zni+\F_i$ and $\Delta X_n(i):=X_{n+1}(i)-X_n(i)=\Delta \zni$. For  the PAFRO model, using \eqref{eq:c1n},
	\be\ba \label{eq:meanmkn}
	\Ef{}{M^k_{n+1}(i)\ \big|\ \G_n}&=c_{n+1}^k\Ef{\bigg}{{X_{n+1}(i)+(k-1) \choose k}\ \bigg|\ \G_n}\\
	&=c_{n+1}^k\Ef{\bigg}{{X_n(i)+(k-1) \choose k}\frac{\Gamma(X_{n+1}(i)+k)}{\Gamma(X_n(i)+k)}\frac{\Gamma(X_n(i))}{\Gamma(X_{n+1}(i))}\ \bigg|\ \G_n}\\
	&=c_{n+1}^k{X_{n}(i)+(k-1) \choose k}\Ef{\bigg}{1+\Delta X_n(i)\frac{k}{X_n(i)} \ \Big|\ \G_n},
	\ea\ee
	as $\Delta X_n(i)$ is either $0$ or $1$. Then, taking the expected value of $\Delta X_n(i)$ yields
	\be\ba 
	\Ef{}{M^k_{n+1}(i)\ \big|\ \G_n}&=c_{n+1}^k{X_{n}(i)+(k-1) \choose k}\left(1+\frac{X_n(i)}{m_0+(n-n_0)+S_n}\frac{k}{X_n(i)}\right)=M^k_n(i),
	\ea\ee
	as $c^k_{n+1}(1+k/(m_0+(n-n_0)+S_n))=c^k_n$. Note that the conditional mean of $M^k_n(i)$ is finite almost surely as well. For the PAFFD model with out-degree $m\in\N$, we can follow the same steps to find
	\be \ba\label{eq:paffdmart}
	\Ef{}{\wt M^k_{n+1}(i)\ \big|\ \G_n}&=\wt c_{n+1}^k{X_n(i)+(k-1) \choose k}\Ef{\bigg}{\frac{\Gamma(X_{n+1}(i)+k)}{\Gamma(X_n(i)+k)}\frac{\Gamma(X_n(i))}{\Gamma(X_{n+1}(i))}\ \bigg|\ \G_n}\\
	&=\wt c_{n+1}^k{X_n(i)+(k-1) \choose k}\mathbb{E}_{\F}\bigg[\prod_{\ell=0}^{\Delta X_n(i)-1}\frac{X_n(i)+k+\ell}{X_n(i)+\ell}\ \bigg|\ \G_n\bigg]\\
	&\leq \wt c_{n+1}^k{X_n(i)+(k-1) \choose k}\mathbb{E}_{\F}\bigg[\bigg(\frac{X_n(i)+k}{X_n(i)}\bigg)^{\Delta X_n(i)}\ \bigg|\ \G_n\bigg],
	\ea\ee 
	where we use Gamma function's properties in the second line and note that $x\mapsto (x+k)/x$ is decreasing in $x$ for $k\geq 0$ in the last step. For $k\in(-\min(\F_i,1),0)$ the upper bound becomes a lower bound, as $x\mapsto (x+k)/x$ is decreasing in $x$ in that case. Conditional on $\G_n$, the number of edges vertex $n+1$ connects to $i$ is a binomial random variable with $m$ trials and success probability $X_n(i)/\sum_{j=1}^n X_n(j)$, so
	\be \ba
	\mathbb{E}_{\F}\bigg[\bigg(\frac{X_n(i)+k}{X_n(i)}\bigg)^{\Delta X_n(i)}\ \bigg|\ \G_n\bigg]&=\bigg(1+\frac{k}{\sum_{j=1}^nX_n(j)}\bigg)^m,
	\ea\ee 
	where we use that a random variable $X\sim \mathrm{Bin}(m,p)$ has probability generating function $\E{z^X}=(pz+(1-p))^m$, $z\in\R$. Then, recalling that for the PAFFD model $\sum_{i=1}^n X_n(i)=m_0+m(n-n_0)+S_n$ yields the result. For  the PAFUD model, we require a few more steps. As the connection of the $i^{\text{th}}$ edge of vertex $n+1$ is dependent on the connection of edges $1,\ldots,i-1$, we iteratively condition on $\G_{n,j},j=m-1,m-2,\ldots,0$, the graph with $n$ vertices where the $n+1^{\text{st}}$ vertex has connected $j$ of its half-edges to the vertices $1,\ldots,n$. More precisely, letting $X_{n,j}:=\Zm_{n,j}(i)+\F_i$, we write
	\be \ba
	\Ef{}{M^k_{n+1}(i)\ \big|\ \G_n}&= c_{n+1}^k\Ef{\bigg}{\E{{X_{n+1,0}(i)+(k-1) \choose k}\ \bigg|\ \G_{n,m-1}}\bigg|\G_n}\\
	&=  c_{n+1}^k\Ef{\bigg}{\E{{X_{n,m-1}(i)+\ind_{n+1,m,i}+(k-1) \choose k}\ \bigg|\ \G_{n,m-1}}\bigg|\G_n},
	\ea\ee
	where $\ind_{n+1,m,i}$ is the indicator of the event that the $m^{\text{th}}$ half-edge of vertex $n+1$ connects with vertex $i$. Now, as in \eqref{eq:meanmkn}, we write this as
	\be\ba
	\Ef{}{M^k_{n+1}(i)\ \big|\ \G_n}= c_{n+1}^k\Ef{\bigg}{{X_{n,m-1}(i)+(k-1) \choose k}\Big(1+k\frac{\E{\ind_{n+1,m,i}\;|\;\G_{n,m-1}}}{X_{n,m-1}(i)}\Big)\bigg|\G_n}.
	\ea\ee 
	By the definition of  the PAFUD model, the mean of the indicator equals \\ $X_{n,m-1}(i)/\sum_{j=1}^n X_{n,m-1}(j)=X_{n,m-1}(i)/(m_0+m(n-n_0)+(m-1)+S_n)$. Hence, we obtain
	\be 
	\Ef{}{ M^k_{n+1}(i)\ \big|\ \G_n}= c_{n+1}^k \Big(1+\frac{k}{m_0+m(n-n_0)+(m-1)+S_n}\Big)\Ef{\bigg}{{X_{n,m-1}(i)+(k-1) \choose k}\bigg|\G_n},
	\ee
	which, when iteratively following the same steps by conditioning on $\G_{n,j}$ for $j=m-2,\ldots,0$, yields the required result. Finally, we prove that $M^1_n(i)$ is a martingale in the PAFFD model. We repeat the steps in \eqref{eq:paffdmart}, but note that as $k=1$, we can omit the inequality and obtain
	\be 
	\Ef{}{M^1_{n+1}(i)\, \big|\, \G_n}=c^1_{n+1}(m)X_n(i)(1+\Ef{}{\Delta X_n(i)\,\big|\,\G_n}/X_n(i)).
	\ee 
	As before, we note that $\Delta X_n(i)$ is a binomial random variable with mean $mX_n(i)/\sum_{j=1}^n X_n(j)$. Thus,
	\be \ba
	\Ef{}{M^1_{n+1}(i)\, \big|\, \G_n}&=c^1_{n+1}(m)X_n(i)\Big(1+\frac{m}{m_0+m(n-n_0)+S_n}\Big)=c^n_{n}(m)X_n(i)=M^1_n(i),
	\ea\ee 
	which finishes the proof.	
	\end{proof}

	From Lemma \ref{lemma:martingale}, we immediately conclude that the (super)martingales $M^k_n(i),\wt M^k_n(i)$ converge almost surely, as they are non-negative, to some random variables $\xi^k_i,\wt \xi^k_i$, respectively. To use this in a meaningful way, we look more closely at this almost sure limit, showing that it does not have an atom at zero, and we study the growth rate of the sequences $c^k_n,\wt c^k_n$. We first dedicate a lemma to the latter:

	\begin{lemma}\label{lemma:ckn}
	Consider the sequences $c^k_n,\wt c^k_n$ in \eqref{eq:c1n} and recall $\theta_m:=1+\E{\F}/m$. If $\E{\F^{1+\eps}}<\infty$ for some $\eps>0$, then for any $k\in\R,m\in\N$,
	\be \label{eq:cknconv}
	c^k_n(m) n^{k/\theta_m}\toas c_k(m),\qquad \wt c^k_n(m)n^{k/\theta_m}\toas \wt c_k(m),
	\ee 
	for some almost surely finite random variables $c_k(m),\wt c_k(m)$. When the fitness distribution satisfies Assumption \ref{Ass:powerlaw} with $\alpha\in(1,2)$, for any $k\in\R,m\in\N$,
	\be \label{eq:infmeancknconv}
	c^k_n\toas \underline{c}_k(m), \qquad \wt c^k_n \toas \underline{\wt c}_k(m),
	\ee 
	for some almost surely finite random variables $\underline{c}_k(m),\underline{\wt c}_k(m)$ (again omitting the $(m)$ whenever there is no ambiguity). Furthermore, the following upper and lower bounds hold almost surely for $c^k_n(m)$ when $\alpha>2$ (they hold for $\wt c^k_n(m)$ as well). For $n_0+1\leq i\leq n$,
	\be\ba\label{eq:c1nupperbound}
	\frac{c^k_i(m)}{c^k_n(m)}\Big(\frac{i}{n}\Big)^{k/\theta_m}\leq{}&\exp\bigg\{\frac{k}{\theta_m}\log\Big(\frac{i}{n}\frac{n-(n_0+1)}{(i-(n_0+1))\vee 1}\Big)+\frac{mk}{\E{\F}}\sum_{j=i}^\infty \frac{|S_j/j-\E{\F}|}{m_0+m(j-n_0)+S_j}\bigg\},\\
	\frac{c^k_i(m)}{c^k_n(m)}\Big(\frac{i}{n}\Big)^{k/\theta_m}\geq{}& 1-\frac{mk}{\E{\F}}\sum_{j=i}^{n-1} \frac{|S_j/j-\E{\F}|}{m_0+m(j-n_0)+S_j}-\frac{m}{2}\sum_{j=i}^{n-1}\Big(\frac{k}{S_j}\Big)^2\\
	&-\frac{m_0+\E{\F}n_0+(m-1)}{\theta_m^2}\frac{\pi^2}{6i}-\frac{1}{\theta_m((i-(n_0+1))\vee 1)}.
	\ea\ee
	\end{lemma}

	\begin{proof}
	We only prove the results for $c^k_n(1)$, as the proofs for $m>1$ and $\wt c^k_n(m)$ follow similarly. For ease of writing, let $\theta:=\theta_1$. We start by proving \eqref{eq:cknconv}. We can write
	\be\ba\label{eq:c1nbalance}
	c^k_n n^{k/\theta}=\exp\bigg\{&-\sum_{j=n_0}^{n-1}\log\bigg(1+\frac{k}{m_0+j-n_0+S_j}\bigg)+\frac{k}{\theta}\log n\bigg\}\\
	=\exp\bigg\{&-\sum_{j=n_0}^{n_0+\lceil 2|k|\rceil }\log\bigg(1+\frac{k}{m_0+j-n_0+S_j}\bigg) -\sum_{j=n_0+\lceil 2|k|\rceil+1}^{n-1}\frac{k}{j\theta}\\
	&-\sum_{j=n_0+\lceil 2|k|\rceil+1}^{n-1}\frac{k}{j\theta}\frac{(\E{\F}-S_j/j)-(m_0-n_0)/j}{(m_0-n_0)/j+1+S_j/j}\\
	&+\sum_{j=n_0+\lceil 2|k|\rceil+1}^{n-1}\sum_{\ell=2}^\infty (-1)^\ell\frac{1}{\ell}\Big(\frac{k}{m_0+j-n_0+S_j}\Big)^\ell+\frac{k}{\theta}\log n\bigg\},
	\ea\ee 
	where we apply a Taylor expansion on the logarithmic terms in the sum for $j\geq n_0+\lceil 2|k|\rceil+1$. The second sum and the last term balance, their sum converges to some finite value depending on $k$ and $\gamma$, where $\gamma$ is the Euler-Mascheroni constant. We now show the almost sure absolute convergence of the third sum in the second line of \eqref{eq:c1nbalance}. This is implied by the almost sure convergence of 
	\be
	\sum_{j=1}^n \frac{1}{j^2}|S_j-j\E{\F}|.
	\ee 
	We prove this by showing that the mean of this sum converges. Let $\eps>0$ such that the $(1+\eps)^{\mathrm{th}}$ moment of the $\F_i$ exists. Using H\"older's inequality, we obtain
	\be
	\sum_{j=1}^n\E{|S_j-j\E{\F}|/j^2}\leq \sum_{j=1}^n\frac{1}{j^2}\E{|S_j-j\E{\F}|^{1+\eps}}^{1/(1+\eps)}.
	\ee 
	Now, we use a specific case of the Marcinkiewicz-Zygmund inequality \cite[Proposition 3.8.2]{Gut13}, which states that for $q\in[1,2]$ and i.i.d.\  $X_i$ with $\E{X_1}=0,\E{|X_1|^q}<\infty$, there exists a constant $c_q$ such that
	\be\label{eq:marcin}
	\mathbb{E}\bigg[\Big|\sum_{i=1}^j X_i\Big|^q\bigg]\leq c_q j \E{|X_1|^q}.
	\ee 
	Thus, if we set $X_i:=\F_{i}-\E{\F}$, it follows that
	\be
	\sum_{j=1}^n \frac{1}{j^2}\E{|S_j-j\E{\F}|^{1+\eps}}^{1/(1+\eps)}\leq c_{1+\eps}\E{|\F-\E{\F}|^{1+\eps}}^{1/(1+\eps)}\sum_{j=1}^n j^{-(2-1/(1+\eps))},
	\ee
	which converges, as $\eps>0$. Finally, taking the absolute value of the double sum in \eqref{eq:c1nbalance} yields the upper bound
	\be 
	\sum_{j=n_0+\lceil 2|k|\rceil+1}^{n-1}\sum_{\ell=2}^\infty\frac{1}{\ell}\Big(\frac{|k|}{m_0+j-n_0+S_j}\Big)^\ell\leq \sum_{\ell=2}^\infty \sum_{j=\lceil 2|k|\rceil+1}^\infty \frac{|k|^\ell}{j^\ell} \leq |k|\sum_{\ell=2}^\infty \sum_{i=2}^\infty i^{-\ell}=|k|\sum_{\ell=2}^\infty(\zeta(\ell)-1).
	\ee 
	In the first step, we first bound $m_0+j-n_0+S_j$ from below by $j-n_0$ and then take all terms where $ik< j\leq(i+1)k$, $i\geq 2$, and bound them from below by $ik$, which yields the same upper bound $k$ times in the third step. The right-hand-side equals $|k|$ and thus proves the almost sure convergence of the double sum. This proves \eqref{eq:cknconv}. For proving \eqref{eq:infmeancknconv} we use a different approach. Namely, we prove that $-\log c^k_n$ converges almost surely, which yields the desired result as well. To that end, let $M_j:=\max_{i\leq j}\F_i$. Then, we write
	\be \label{eq:logckn}
	-\log c^k_n = \sum_{j=n_0}^{n-1} \log\Big(1+\frac{k}{m_0+j-n_0+S_j}\Big)\leq \sum_{j=1}^{n-1}\frac{k}{S_j}\leq \sum_{j=1}^J\frac{k}{M_j}+k\sum_{j=J+1}^n j^{-1/(\alpha-1)+\eps},
	\ee  
	where we use \eqref{eq:Miprobbound} in the last step to conclude that, by the Borel-Cantelli lemma, there exists an almost surely finite random index $J$ such that for all $j\geq J$, $M_j\geq j^{1/(\alpha-1)-\eps}$, for some small $\eps\in(0,(2-\alpha)/(\alpha-1))$. It therefore follows that the upper bound on the right-hand-side of \eqref{eq:logckn} converges as $n$ tends to infinity almost surely, and therefore so does $c^k_n$, since $-\log c^k_n$ is non-negative and increasing. We now turn to the bounds in \eqref{eq:c1nupperbound}. Rather than using a Taylor expansion as in \eqref{eq:c1nbalance}, we simply use that $\log(1+x)\leq x$, to obtain
	\be \label{eq:cknfirstbound}
	\frac{c^k_i}{c^k_n}\Big(\frac{i}{n}\Big)^{k/\theta}\leq \exp\bigg\{k(E(n)-E(i))+k\sum_{j=i}^{n-1}\frac{S_j-j\E{\F}}{(m_0+j\theta-n_0)(m_0+j-n_0+S_j)}\bigg\},
	\ee 
	where 
	\be
	E(n):=\sum_{j=n_0}^{n-1}\frac{1}{m_0+j\theta-n_0 }-\frac{1}{\theta}\log n.
	\ee 
	We rewrite $E(n)$ to find
	\be\ba\label{eq:enexpand}
	E(n)={}&\bigg(\sum_{j=0}^{n-(n_0+1)}\frac{1}{m_0+\E{\F}n_0+j\theta}-\sum_{j=1}^{n-(n_0+1)}\frac{1}{j\theta}\bigg)\\
	&+\bigg(\sum_{j=1}^{n-(n_0+1)}\frac{1}{j\theta}-\frac{1}{\theta}\log(n-(n_0+1))\bigg)+\frac{1}{\theta}\log(1-(n_0+1)/n),
	\ea\ee 
	where we note that the first and second term are decreasing and the final term is increasing in $n$. Hence, we obtain the upper bound for all $n_0+1\leq i \leq n $,
	\be
	E(n)-E(i)\leq \frac{1}{\theta}\log\Big(\frac{i}{n}\frac{n-(n_0+1)}{(i-(n_0+1))\vee 1}\Big).
	\ee 
	Using this inequality and taking the absolute value of the terms in the sum in \eqref{eq:cknfirstbound}, yields the upper bound
	\be 
	\exp\bigg\{\frac{k}{\theta}\log\Big(\frac{i}{n}\frac{n-(n_0+1)}{(i-(n_0+1))\vee 1}\Big)+\frac{k}{\E{\F}}\sum_{j=i}^\infty \frac{|S_j/j-\E{\F}|}{m_0+j-n_0+S_j}\bigg\},
	\ee 
	as required. Similarly, we find a lower bound of the same form. As $\log(1+x)\geq x-x^2/2$ for $x\geq 0$, $\exp\{-x\}\geq 1-x$ for $x\in\R$, we find
	\be \ba \label{eq:c1nlower}
	\frac{c^k_i}{c^k_n}\Big(\frac{i}{n}\Big)^{k/\theta}&\geq \exp\bigg \{&&\hspace{-13pt}- k\sum_{j=i}^{n-1} \frac{|S_j/j-\E{\F}|}{\E{\F}(m_0+j-n_0+S_j)}-\frac{1}{2}\sum_{j=i}^{n-1}\Big(\frac{k}{m_0+j-n_0+S_j}\Big)^2\\
	& &&\hspace{-13pt}+k(E(n)-E(i))\bigg\}\\
	&\geq 1-k&&\hspace{-12pt}\sum_{j=i}^{n-1} \frac{|S_j/j-\E{\F}|}{\E{\F}(m_0+j-n_0+S_j)}-\frac{1}{2}\sum_{j=i}^{n-1}\Big(\frac{k}{m_0+j-n_0+S_j}\Big)^2\\
	&\hspace{10pt} +k(E&&\hspace{-10pt}(n)-E(i)).
	\ea \ee 
	Using \eqref{eq:enexpand} and the fact that $\sum_{j=1}^{n-1} \frac{1}{j} -\log n$ is non-decreasing, we obtain the lower bound
	\be\ba
	1&-k\sum_{j=i}^{n-1} \frac{|S_j/j-\E{\F}|}{\E{\F}(m_0+j-n_0+S_j)}-\frac{1}{2}\sum_{j=i}^{n-1}\Big(\frac{k}{S_j}\Big)^2-\frac{m_0+\E{\F}n_0}{\theta^2}\sum_{j=i-n_0}^{n-(n_0+1)} \frac{1}{j^2} \\
	&+\frac{1}{\theta(n-(n_0+1))}-\frac{1}{\theta((i-(n_0+1))\vee 1)}\\
	\geq1&-k\sum_{j=i}^{n-1} \frac{|S_j/j-\E{\F}|}{\E{\F}(m_0+j-n_0+S_j)}-\frac{1}{2}\sum_{j=i}^{n-1}\Big(\frac{k}{S_j}\Big)^2-\frac{m_0+\E{\F}n_0}{\theta^2}\frac{\pi^2}{6i}\\
	&-\frac{1}{\theta((i-(n_0+1))\vee 1)},
	\ea \ee 
	which finishes the proof.
	\end{proof}

    We now show that the almost sure limits of certain (super)martingales in Lemma \ref{lemma:martingale} do not have an atom at zero:
    
    \begin{lemma}\label{lemma:noatom}
    	For $k\geq 1$, consider the martingales $M^k_n(i)$ for the PAFRO and PAFUD models and $\wt M^k_n(i)$ for the PAFFD model as in Lemma \ref{lemma:martingale} and their almost sure limits $\xi^k_i,\wt\xi^k_i$, respectively. Then, the $\xi^k_i,\wt \xi^k_i$ do not have an atom at zero. 
    \end{lemma}

	\begin{proof}
		We first focus on the martingales $M^k_n$ for the PAFRO and PAFUD models. Let $\eps>0$. We can write,
		\be \ba\label{eq:noatom}
		\Pf{\xi_i^k<\eps}&=\lim_{n\to \infty}\Pf{c^k_n{\zni+\F_i+(k-1)\choose k}<\eps}\\
		&\leq \lim_{n\to\infty}\Pf{c^k_n(\zni+\F_i)^k< \eps\Gamma(k+1)},
		\ea \ee 
		since $x^k\leq \Gamma(x+k)/\Gamma(x)$ for $k\geq 1, x>0$, by \cite[Theorem 1]{Jam13}. Now, take  $p\in (-\min(\F_i,1)/k,0)$. The goal is to raise both sides to the power $p$ and use a Markov bound. We first, however, need some other inequalities to obtain useful expressions. Using the concavity of $\log x$ and noting that $x+pk$ is a weighted average of $x$ and $x+k$ when $p\in(0,1)$ and $x+k$ is a weighted average of $x$ and $x+pk$ when $p\geq 1$, we obtain, for all $x,k\geq 0$,
		\be\label{eq:cnkineq}
		\Big(1-\frac{k}{x+k}\Big)^p \geq 1-\frac{pk}{x+pk}\text{ when }p\in(0,1),\text{ and } \Big(1-\frac{k}{x+k}\Big)^p \leq 1-\frac{pk}{x+pk}\text{ when }p\geq 1.
		\ee  
		From the first inequality, we also immediately obtain, for $p\in(-1,0),k\geq0,x\geq k|p|$, 
		\be \label{eq:c1nnegk}
		\Big(1-\frac{k}{x+k}\Big)^p \leq 1-\frac{pk}{x+pk}.
		\ee 
		It thus follows that, when $p\in(-\min(\F_i,1)/k,0)$, $(c^k_n)^p\leq c^{kp}_n$, as $\F_i>k|p|$. Also, from \cite{Wen48} it follows that for all $x\geq 0,s\in(0,1)$, 
		\be 
		x^s \geq \frac{\Gamma(x+s)}{\Gamma(x)}.
		\ee 
		Hence, since $\Gamma(x)/\Gamma(x+s)$ is decreasing in $x$ for $s\geq 0$, when $p\in(-1,0),x\geq|p|$,
		\be\label{eq:gammanegk}
		x^p\leq \frac{\Gamma(x+p)}{\Gamma(x)},
		\ee
		so that, combining both \eqref{eq:c1nnegk} and \eqref{eq:gammanegk} in \eqref{eq:noatom} with $p\in(-\min(\F_i/k,1/k),0)$, yields
		\be \ba\label{eq:noatombound}
		\lim_{n\to\infty}\Pf{c^k_n(\zni+\F_i)^k< \eps\Gamma(k+1)}&\leq \lim_{n\to\infty}\Pf{M^{kp}_n(i)\geq \eps^p\Gamma(k+1)^p/\Gamma(kp+1)}\\
		&\leq \lim_{n\to\infty}\Ef{}{M^{kp}_n(i)}\eps^{|p|}\Gamma(k+1)^{|p|}\Gamma(kp+1)\\
		&=M^{kp}_{i\vee n_0}(i)\eps^{|p|}\Gamma(k+1)^{|p|}\Gamma(pk+1),
		\ea\ee 
		which is finite almost surely and tends to zero with $\eps$ a.s. Hence, almost surely,
		\be 
		\Pf{\xi^k_i=0}=\lim_{\eps\downarrow 0}\Pf{\xi^k_i<\eps}=0,
		\ee 
		and thus $\mathbb{P}(\xi^1_i=0)=0$, by the dominated convergence theorem.	For the PAFFD model, an altered argument is required, since $\wt M^k_n(i)$ is a submartingale for negative $k$, as follows from Lemma \ref{lemma:martingale} so that  the final steps in \eqref{eq:noatombound} no longer work. Rather, we only follow the same steps for $\wt \xi^k_i$ in \eqref{eq:noatom}. Then, let us define, for a large constant $C>0$, $\eta\in(0,\E{\F}/(\E{\F}+m))$ and a large integer $N\geq i\vee n_0$, the stopping time $T_N:=\inf\{n\geq N: \zni\geq Cn^{1-\eta}\}$. We aim to show that we can construct a sequence $\hat c^k_n$, to be defined later, such that 
		\be 
		\hat M^k_{T_N\wedge n}(i):= \hat c^k_{T_N\wedge n}{\Zm_{T_N\wedge n}(i)+\F_i+(k-1)\choose k}
		\ee 
		is a supermartingale for $k\in(-\min(\F_i,1),0)$ for the PAFFD model. First, recall the computations in \eqref{eq:paffdmart}. We notice that the product in the second line contains terms which are positive but less than $1$ when $k\in(-\min(\F_i,1),0)$. Therefore, the product decreases as the number of terms increases, so that we can bound the expected value from above by $1+k\P{\Delta \zni\geq 1\,|\,\G_n}/(\zni+\F_i)$. If we define
		\be 
		\hat c^k_n:=\prod_{j=n_0}^{n-1}\Big(1-\frac{km a_j}{m_0+m(j-n_0)+S_j+kma_j}\Big),\qquad a_n:=1-\frac{m-1}{2}\frac{Cn^{-\eta}+\F_i/n}{(m_0+m(n-n_0)+S_n)/n},
		\ee 
		we obtain
		\be \ba 
		\mathbb{E}_\F&\big[\hat M^k_{T_N\wedge (n+1)}(i)\ind_{\{T_N\geq n+1\}}\,\big|\, \G_n\big ]\\
		&\leq \hat M^k_n (i)\Big(1-\frac{km a_n}{m_0+m(n-n_0)+S_n+kma_n}\Big)\Big(1+k\frac{\P{\Delta\zni\geq 1\,|\,\G_n}}{X_n+\F_i}\Big)\ind_{\{T_N\geq n+1\}}.
		\ea \ee 
		We now bound $\P{\Delta\zni\geq 1\,|\,\G_n}$ from below, using that $1-(1-x)^m\geq mx-m(m-1)x^2/2$ for all $x\in(0,1),m\in\N$. Then, on $\{T_N\geq n+1\}$, we can bound $\zni$ from above by $Cn^{1-\eta}$, which yields the upper bound
		\be\ba
		\hat M^k_n(i) &\Big(1-\frac{km a_n}{m_0+m(n-n_0)+S_n+kma_n}\Big)\Big(1+\frac{km a_n}{m_0+m(n-n_0)+S_n}\Big)\ind_{\{T_N\geq n+1\}}\\
		&=\hat M^k_n(i) \ind_{\{T_N\geq n+1\}}=\hat M^k_{T\wedge n}(i)\ind_{\{T_N\geq n+1\}}.
		\ea\ee 
		Finally, as the event $\{T_N\leq n\}$ is $\G_n$ measurable,
		\be 
		\mathbb{E}_\F\big[M^k_{T_N\wedge (n+1)}(i)\ind_{\{T_N\leq n\}}\,\big|\, \G_n\big ]=\hat M^k_{T_N}(i)\ind_{\{T_N\leq n\}}=\hat M^k_{T_N\wedge n}(i)\ind_{\{T_N\leq n\}}.
		\ee 
		Together with the computations above, this yields
		\be 
		\Ef{\big}{\hat M^k_{T_N\wedge (n+1)}(i)\,\big|\,\G_n}\leq\hat M^k_{T_N\wedge n}(i),
		\ee 
		which shows  indeed that $\hat M^k_{T_N\wedge n}(i)$ is a supermartingale for $k\in(-\min(\F_i,1),0)$. It also follows relatively easily, following similar steps as in the proof of Lemma \ref{lemma:ckn}, that $\hat c^k_n n^{-k/\theta_m}\toas \hat c_k$ for some random variable $\hat c_k$ as $n$ tends to infinity. So, we can then write, for $k\geq 1,p\in(-\min(\F_i/k,1/k),0)$, continuing the steps in \eqref{eq:noatom} and using \eqref{eq:c1nnegk} and \eqref{eq:gammanegk} as in \eqref{eq:noatombound},
		\be
		\mathbb{P}_\F(\wt \xi^k_i<\eps)\leq \lim_{n\to\infty}\mathbb{P}_\F((c^{kp}_n/\hat c^{kp}_n)\hat M^{kp}_n(i)\geq \eps^p\Gamma(k+1)^p/\Gamma(kp+1)).
		\ee
		We now intersect with the event $\{T_N\geq n+1\}$ and its complement to obtain the upper bound
		\be \ba
		\lim_{n\to\infty}{}&\Pf{\{(c^{kp}_n/\hat c^{kp}_n)\hat M^{kp}_n(i)>\eps^p\Gamma(k+1)^p/\Gamma(kp+1)\}\cap \{T_N\geq n+1\}}+\Pf{T_N\leq n}\\
		&\leq \lim_{n\to\infty}\Pf{(c^{kp}_n/\hat c^{kp}_n)\hat M^{kp}_{T_N\wedge n}(i)>\eps^p\Gamma(k+1)^p/\Gamma(kp+1)}+\Pf{T_N\leq n}.
		\ea\ee
		Using the Markov inequality for the first probability and because $\hat M^{kp}_{T_N\wedge n}(i)$ is a supermartingale since $kp\in(-\min(\F_i,1),0)$, we find the upper bound 
		\be\ba\label{eq:epsbound}
		\lim_{n\to\infty}&(c^{kp}_n/\hat c^{kp}_n)\eps^{|p|}\Ef{}{\hat M^{kp}_{T_N\wedge n}(i)}\Gamma(k+1)^{|p|}\Gamma(kp+1)+\Pf{T_N\leq n}\\
		&\leq (c_{kp}/\hat c_{kp})\eps^{|p|}\Ef{}{\hat M^{kp}_{N}(i)}\Gamma(k+1)^{|p|}\Gamma(kp+1)+\lim_{n\to\infty}\Pf{T_N\leq n}.
		\ea \ee 
		We note that the first term tends to zero with $\eps$. For the second probability we write, for some $s^{\text{th}}$ moment bound, with $s>(\E{\F}/(\E{\F}+m)-\eta)^{-1}$,
		\be \ba 
		\Pf{T_N\leq n}&\leq  \sum_{j=N}^n \mathbb{P}_\F\big((\Zm_j(i)+\F_i)^s\geq C^sj^{s(1-\eta)}\big)\leq  \frac{\Gamma(k+1)}{C^s} \sum_{j=N}^n (\wt c^s_j)^{-1} j^{-s(1-\eta)} \Ef{}{\wt M^s_j(i)}.
		\ea\ee 
		Using the upper bound for $c^s_{n_0}/c^s_n=1/c^s_n$ in \eqref{eq:c1nupperbound}, we find the upper bound
		\be \ba
		C_{k,s}A\wt M^s_{i\vee n_0}(i)\sum_{j=N}^n j^{s(1/\theta_m-(1-\eps))}\leq \wt C_{k,s}A \wt M^s_{i\vee n_0}(i) N^{1-s(\E{\F}/(\E{\F}+m)-\eps)},
		\ea \ee 
		where $A$ equals the upper bound in \eqref{eq:c1nupperbound} with $i=n_0$. This upper bound is independent of $n$, so we find, combining this with \eqref{eq:epsbound},
		\be
		\lim_{\eps\downarrow 0}\mathbb{P}_\F(\wt\xi^k_i<\eps)\leq \wt C_{k,s}A \wt M^s_{i\vee n_0}(i) N^{1-s(\E{\F}/(\E{\F}+m)-\eta)},
		\ee
		where the right-hand-side tends to zero almost surely as $N$ tends to infinity, by the choice of $s$. Thus, it follows that $\lim_{\eps \downarrow 0}\mathbb{P}_\F(\wt \xi^k_i<\eps)=0$ for all $k\geq 1$. Again, using the dominated convergence theorem finally yields the required result.
	\end{proof}
    In order to show that the maximum degree converges almost surely when $\alpha>1+\theta_m$, we require a little bit more control over the (super)martingales $M^k_n(i),\wt M^k_n(i)$ than just their convergence, as we need to be able to bound their suprema, for which we introduce the following lemma:
    
    \begin{lemma}\label{lemma:supmkn}
    	Consider the martingales (resp.\ supermartingales) $M^k_n(i)$ (resp.\ $\wt M^k_n(i)$) as in Lemma \ref{lemma:martingale}. Let $M:=\sup\{s\geq 1:\E{\F^s}<\infty\}$. Then, for all $m\in\N,k\in(\theta_m,M)$, almost surely
    	\be\label{eq:supmkn0}
    	\lim_{i\to\infty}\sup_{n\geq n_0\vee i}M^k_n(i)=0,\qquad \lim_{i\to\infty}\sup_{n\geq n_0\vee i}\wt M^k_n(i)=0.
    	\ee 
    \end{lemma}

	\begin{proof}
	We note that the first result is implied if, for any $\eps>0$,
	\be 
	\P{\sup_{n\geq i\vee n_0}M^k_n(i)\geq \eps\text{ for infinitely many }i}=0,
	\ee 
	and similarly for $\wt M^k_n(i)$. We now use the `good' event $E_\ell(\delta):=\{|S_j/j-\E{\F}|\leq \delta\ \forall j\geq \ell \}$, where we take $\delta>0$ sufficiently small such that $k\in(\theta_m(1+\delta),I)$. That is, we intersect with $E_\ell(\delta)$ and $E_\ell(\delta)^c$. By writing i.o.\ for `infinitely often', we find
	\be \ba\label{eq:supmknio}
	\P{\sup_{n\geq i\vee n_0}M^k_n(i)\geq \eps\text{ i.o.}}&\leq \P{\Big\{\sup_{n\geq i\vee n_0}M^k_n(i)\geq \eps\text{ i.o.}\Big\}\cap E_\ell(\delta)}+\P{E_\ell(\delta)^c}\\
	&=\mathbb{P}\Big(\ind_{E_\ell(\delta)}\sum_{i=1}^\infty \ind_{A_i}=\infty\Big)+\P{E_\ell(\delta)^c},
	\ea\ee 
	where $A_i:=\{\sup_{n\geq i\vee n_0}M^k_n(i)\geq \eps\}$. We now show that the first probability on the right-hand-side equals $0$ for every $\ell\in\N$, by showing the sum of indicators has a finite mean. We write
	\be\label{eq:indmean}
	\mathbb{E}\bigg[\ind_{E_\ell(\delta)}\sum_{i=1}^\infty \ind_{A_i}\bigg]=\mathbb{E}\bigg[\ind_{E_\ell(\delta)}\Ef{\Big}{\sum_{i=1}^\infty \ind_{A_i}}\bigg],
	\ee 
	and first deal with the conditional expectation. We apply Doob's martingale inequality \cite[Theorem II 1.7]{RevYor13} to the events $A_i$ to find
	\be\label{eq:doobmart}
	\Pf{A_i}=\lim_{N\to\infty}\Pf{\sup_{i\vee n_0 \leq n\leq N}M^k_n(i)\geq \eps}\leq \lim_{N\to\infty}\eps^{-1}\Ef{}{M^k_N(i)}=\eps^{-1}\Ef{}{M^k_{i\vee n_0}(i)},
	\ee 
	where the first step holds by the monotonicity of the events $\{\sup_{i\vee n_0\leq n\leq N}M^k_n(i)\geq \eps\}$.
	
	Doob's martingale inequality holds for submartingales only, though. However, we can still prove the same upper bound for $\wt M^k_n(i)$, but a different technique is required. We define the stopping time $\tau_\eps:=\inf\{ n\geq i\vee n_0\, |\, \wt M^k_n(i)\geq \eps\}$. Then, for any $N\in\N$,
	\be\ba
	\mathbb{P}_\F\bigg(\sup_{i\vee n_0\leq n \leq N}\wt M^k_n(i)\geq \eps\bigg) &=\mathbb{P}_\F(\tau_\eps\leq N)=\mathbb{P}_\F\Big(\ind_{\{\tau_\eps\leq N\}}\wt M^k_{\tau_\eps}(i)\geq \eps\Big) \leq \frac{1}{\eps}\Ef{}{\ind_{\{\tau_\eps\leq N\}}\wt M^k_{\tau_\eps}(i)}\\
	&\leq \frac{1}{\eps}\Big(\Ef{}{\ind_{\{\tau_\eps\leq N\}}\wt M^k_{\tau_\eps}(i)} +\Ef{}{\ind_{\{\tau_\eps> N\}}\wt M^k_N(i)}\Big)=\frac{1}{\eps}\Ef{}{\wt M^k_{\tau_\eps \wedge N}(i)},
	\ea\ee
	see also \cite[Exercise $1.25$, Chapter II]{RevYor13}. We now use the optional sampling theorem \cite[Theorem 10.10]{Wil91}, which yields the required upper bound. Again, by monotonicity and taking $N$ to infinity we obtain the same result. Now, using \eqref{eq:doobmart} in \eqref{eq:indmean} and using Markov's inequality, on $E_\ell(\delta)$, 
	\be \ba
	\mathbb{E}\bigg[&\ind_{E_\ell(\delta)}\sum_{i=1}^\infty \eps^{-1} c^k_{i\vee n_0} { \Zm_{i\vee n_0}(i)+\F_i+(k-1)\choose k}\bigg]\\
	&\leq C\sum_{i=\ell}^\infty\E{\ind_{E_\ell(\delta)} i^{-k/(\theta_m(1+\delta))}{\F_i+(k-1)+m_0\choose k}}+\sum_{i=1}^{\ell-1} \E{{\F_i+(k-1)+m_0\choose k}} \\
	&\leq \wt C(1+\E{\F^k})\sum_{i=\ell}^\infty i^{-k/(\theta_m(1+\delta))}+\wt C(1+ \E{\F^k})\ell,
	\ea\ee 
	which is finite by the choice of $k$ and $\delta$. For the second sum we cannot use the event $E_\ell(\delta)$ and therefore bound $c^k_{i\vee n_0}$ from above by $1$. We note that we can indeed bound the mean of ${\F+(k-1)+m_0\choose k}$ by a constant times $1$ plus the $k^{\mathrm{th}}$ moment of $\F$. Namely, using the asymptotics of the Gamma function,
	\be \ba
	\E{{\F+(k-1)+m_0\choose k}}&= \int_0^\infty {x +(k-1)+m_0\choose k}\mu(\d x)\\
	&\leq \int_0^{x^*} {x +(k-1)+m_0\choose k}\mu(\d x)+C_1\int_{x^*}^\infty  x^k \mu(\d x)\\
	&\leq C_2(1+\E{\F^k}),
	\ea \ee
	with $C_2:=\max\{C_1,\int_0^{x^*} {x +(k-1)+m_0\choose k}\mu(\d x)\}$ and $x^*$ such that for $x\geq x^*, {x+(k-1)+m_0\choose k}\leq C_1 x^k$. It follows that the mean in \eqref{eq:indmean} is finite and thus that the first probability on the right-hand-side of \eqref{eq:supmknio} equals $0$. Hence,
	\be 
	\P{\sup_{n\geq i\vee n_0}M^k_n(i)\geq \eps\text{ i.o.}}\leq \P{E_\ell(\delta)^c},
	\ee 
	which tends to $0$ as $\ell\to\infty$ by the strong law of large numbers, and so we obtain \eqref{eq:supmkn0}.
	\end{proof}

    The final result we need comes from \cite{Athr08} and provides conditions such that the maximum of a double array converges to a certain limit:

	\begin{proposition}{\cite[Proposition 3.1]{Athr08}}\label{lemma:sequences}
	Let $\{a_{n,i}:i\in[n]\}_{n\geq 1}$ be a double array of non-negative numbers such that
	\begin{enumerate}
		\item For all $i\geq 1,\lim_{n\rightarrow\infty} a_{n,i}=a_i<\infty$,
		\item $\sup_{n\geq 1}a_{n,i}\leq b_i<\infty$,
		\item $\lim_{i\rightarrow \infty} b_i=0$,
		\item For $i\neq j$, $a_i\neq a_j$.
	\end{enumerate}
	Then,
	\begin{itemize}
		\item $\max_{i\in [n]}a_{n,i}\rightarrow \max_{i\geq 1}a_i$, as $n\rightarrow \infty$.
		\item In addition, there exist $I_0$ and $N_0$ such that $\max_{i\in[n]}a_{n,i}=a_{n,I_0}$ for all $n\geq N_0$.
	\end{itemize}
	\end{proposition}

	We now prove Proposition \ref{lemma:condmeanconv}:

	\begin{proof}[Proof of Proposition \ref{lemma:condmeanconv}]
	The focus of the proof is on  the PAFUD model. The proof for the PAFRO model follows by setting $m=1$, the proof for the PAFFD model follows in the same way, as we only look at the mean of $M^1_n(i)$, which by Lemma \ref{lemma:martingale} is a martingale for both the PAFUD \emph{and} PAFFD model. 
	
	We start by proving \eqref{eq:condmeanconv}. Take $\alpha\in(2,1+\theta_m)$. Using Lemma \ref{lemma:martingale}, it directly follows that 
	\be\label{eq:condmean} 
	\Ef{}{\zni}=(c_n^1(m))^{-1}\Ef{}{M^1_n(i)}-\F_i=\frac{c^1_{i\vee n_0}(m)}{c^1_n(m)}\Zm_{i\vee n_0}(i)+\F_i\Big(\frac{c_{i\vee n_0}^1(m)}{c_n^1(m)}-1\Big).
	\ee 
	Note that for $i\geq n_0$ the first term on the right-hand-side equals zero. We can then construct the inequalities
	\be \label{eq:maxinequalities}
	\max_{\inn}\frac{\F_i}{u_n} \Big(\frac{c^1_{i\vee n_0}}{c^1_n}-1\Big)\leq\max_{i\in[n]}\Ef{}{\zni/u_n}\leq\max_{\inn}\frac{\F_i}{u_n} \Big(\frac{c^1_{i\vee n_0}}{c^1_n}-1\Big)+\frac{m_0}{u_n c^1_n}.
	\ee 
	By Lemma \ref{lemma:ckn}, the last term on the right-hand-side tends to zero almost surely, as $\alpha-1<\theta_m$. By the reverse triangle inequality, it follows that for $x,y\in \R_+^n$, 
	\be \label{eq:maxabsdif}
	|\max_{\inn}x_i-\max_{\inn}y_i|=|\|x\|_\infty-\|y\|_\infty|\leq \|x-y\|_\infty=\max_{\inn}|x_i-y_i|.
	\ee 
	So, as 
	$c^1_{i\vee n_0}=c^1_i$ for all $i\geq n_0$,
	\be 
	\Big|\max_{\inn}\frac{\F_i}{u_n} \Big(\frac{c^1_{i\vee n_0}}{c^1_n}-1\Big)-\max_{\inn}\frac{\F_i}{u_n} \Big(\frac{c^1_i}{c^1_n}-1\Big)\Big|\leq \max_{\inn}\frac{\F_i}{u_n}\frac{c^1_i-c^1_{i\vee n_0}}{c^1_n}=\max_{i<n_0}\frac{\F_i}{u_n}\frac{c^1_i-c^1_{n_0}}{c^1_n},
	\ee 
	which again tends to zero almost surely by Lemma \ref{lemma:ckn}, as it is a maximum over a finite number of terms. Therefore, assuming the limits exist, it follows that
	\be \label{eq:maxlim}
	\lim_{n\rightarrow\infty}\max_{i\in[n]}\Ef{}{\zni/u_n}=\lim_{n\rightarrow\infty}\max_{\inn}\frac{\F_i}{u_n} \Big(\frac{c^1_i}{c^1_n}-1\Big)
	\ee 
	almost surely. We now show that
	\be \label{eq:maxc1norder}
	\Big|\max_{\inn}\frac{\F_i}{u_n} \Big(\frac{c^1_i}{c^1_n}-1\Big)-\max_{\inn}\frac{\F_i}{u_n}\Big(\Big(\frac{i}{n}\Big)^{-1/\theta_m}-1\Big)\Big| \toinp 0.
	\ee 
	Using \eqref{eq:maxabsdif} we find
	\be
	\bigg|\max_{\inn}\frac{\F_i}{u_n}\Big(\frac{c^1_i}{c^1_{n}}-1\Big)-\max_{\inn}\frac{\F_i}{u_n}\Big(\Big(\frac{i}{n}\Big)^{-1/\theta_m}-1\Big)\bigg|\leq \max_{\inn}\frac{\F_i}{u_n}\Big(\frac{n}{i}\Big)^{1/\theta_m}\bigg|\frac{c^1_i}{c^1_n}\Big(\frac{i}{n}\Big)^{-1/\theta_m}-1\bigg|.
	\ee 
	Then, let $\eta\in(1,(\alpha-2)\wedge 2)$ and let $(\eps_n)_{n\in\N}$ be a sequence such that $\eps_n:=n^{-\beta}$, with $\beta\in(0,\theta_m\eta/(1+(1+\theta_m)\eta))$. We split the maximum into two parts: indices $i$ which are at most $\eps_n n$ and at least $\eps_n n$ and deal with these separately. (Note that $\beta<1$ and thus $\eps_n n\rightarrow\infty$.) We first define, for $A\subseteq[n]$ and $\delta>0$, 
	\be 
	E_{A}:=\Big\{\max_{i\in A}\frac{\F_i}{u_n}\Big(\frac{n}{i}\Big)^{1/\theta_m}\bigg|\frac{c^1_i}{c^1_n}\Big(\frac{i}{n}\Big)^{1/\theta_m}-1\bigg|>\delta\Big\}.
	\ee
	This yields 
	\be \ba\label{eq:epssplit}
	\P{E_{[n]}}\leq \P{E_{[\eps_n n]}}+\P{E_{[n]\backslash[\eps_nn]}}.
	\ea\ee  
	We first investigate the latter probability.  We write,
	\be \ba\label{eq:probmaxmax}
	\P{E_{[n]\backslash [\eps_nn]}}\leq \P{\Big(\max_{i>\eps_n n}\frac{\F_i}{u_n}\Big)\eps_n^{-1/\theta_m}\max_{i> \eps_n n}\bigg|\frac{c^1_i}{c^1_n}\Big(\frac{i}{n}\Big)^{1/\theta_m}-1\bigg|>\delta},
	\ea \ee 
	where we bound the $(n/i)^{1/\theta_m}$ term from above by $\eps_n^{-1/\theta_m}$ and take the maximum over the fitness variables and the absolute value separately. It is clear that the first maximum on the right-hand-side converges in distribution, as the number of terms in the maximum is of order $n$, and so the scaling is of the correct order. When $i\geq \eps_n n$, the indices $i$ tend to infinity with $n$, which indicates that the terms in the absolute value should be small by Lemma \ref{lemma:ckn}. We show that the second maximum tends to zero almost surely even when multiplied with $\eps_n^{-1/\theta_m}$. In order to prove this, we use the bounds in \eqref{eq:c1nupperbound}. The upper bound, when $i\geq \eps_n n$, is largest for $i=\eps_n n$. Thus, we have a uniform upper bound for all $\eps_n n\leq i\leq n$, 
	\be 
	\frac{c^1_i}{c^1_n}\Big(\frac{i}{n}\Big)^{1/\theta_m}\leq \exp\bigg\{\frac{k}{\theta_m}\log\Big(\eps_n\frac{n-(n_0+1)}{\eps_nn-(n_0+1)}\Big)+\frac{mk}{\E{\F}}\sum_{j=\eps_nn}^\infty \frac{|S_j/j-\E{\F}|}{m_0+m(j-n_0)+S_j}\bigg\}.
	\ee 
	For $n$ large, the denominator in the sum can be bounded from below by $mj/2$ and the term in the logarithm can be bounded from above by $1+2(n_0+1)/(\eps_nn)$. Hence, we obtain the upper bound 
	\be 
	\frac{c^1_i}{c^1_n}\Big(\frac{i}{n}\Big)^{1/\theta_m}\leq \exp\bigg\{\frac{k}{\theta_m}\log\Big(1+\frac{2(n_0+1)}{\eps_nn}\Big)+\frac{2k}{\E{\F}}\sum_{j=\eps_nn}^\infty \frac{|S_j/j-\E{\F}|}{j}\bigg\}.
	\ee 
	Similarly, the lower bound in \eqref{eq:c1nlower} is largest when $i=n-1$ (note that the second maximum in \eqref{eq:probmaxmax} is never attained at $i=n$, so we can ignore this case), from which we obtain
	\be 
	\max_{i\geq \eps_n n}\Big(\frac{c^1_i}{c^1_n}\Big(\frac{i}{n}\Big)^{1/\theta_m}-1\Big)\geq-\frac{m_0+\E{\F}n_0+(m-1)}{\theta_m^2}\frac{\pi^2}{6(n-1)}-\frac{1}{\theta_m(n-(n_0+2))}\geq-\frac{C}{n},
	\ee  
	for some constant $C>0$. It then follows that, as $\eps_n^{-1/\theta_m}=n^{\beta/\theta_m}\geq 1$, $a(\mathrm{e}^x-1)\leq \e^{ax}-1$ for all $x\in\R$ when $a\geq 1$,
	\be \ba\label{eq:maxbound}
	\eps_n^{-1/\theta_m}\max_{i\geq \eps_n n}\bigg|\frac{c^1_i}{c^1_n}\Big(\frac{i}{n}\Big)^{1/\theta_m}-1\bigg| \leq \max\bigg\{&\frac{C}{n^{1-\beta/\theta_m}},\exp\bigg\{\frac{k}{\theta_m}\log\Big(\Big(1+\frac{2(n_0+1)}{n^{1-\beta}}\Big)^{n^{\beta/\theta_m}}\Big)\\
	&+\frac{2k}{\E{\F}}n^{\beta/\theta_m}\sum_{j=\eps_nn}^\infty \frac{|S_j/j-\E{\F}|}{j}\bigg\}-1\bigg\}.
	\ea \ee 
	Clearly, the first argument tends to zero, as $\beta < \theta_m$. What remains to prove is that the second argument of the maximum on the right-hand-side of \eqref{eq:maxbound} converges to zero in probability. The first term in the exponent tends to zero, as $1-\beta>\beta/\theta_m$ by the choice of $\beta$. For the second term, using Markov's inequality, for any $\delta>0$,
	\be\ba 
	\mathbb{P}\bigg(n^{\beta/\theta_m}\sum_{j=\eps_n n}^{\infty}\frac{|S_j/j-\E{\F}|}{j}>\delta\bigg)&\leq \delta^{-1}n^{\beta/\theta_m}\sum_{j=\eps_nn}^\infty j^{-2}\E{|S_j-j\E{\F}|}\\
	&\leq \delta^{-1}n^{\beta/\theta_m}\sum_{j=\eps_nn}^\infty j^{-2}\E{|S_j-j\E{\F}|^{1+\eta}}^{1/(1+\eta)},\\
	\ea\ee 
	where we note that $\eta\in(0,(\alpha-2)\wedge 2)$, such that we can apply the Marcinkiewicz-Zygmund inequality as in \eqref{eq:marcin}. This yields, for some constant $C>0$, the upper bound
	\be
	C n^{\beta/\theta_m}\sum_{j=\eps_n n}^\infty j^{-2+1/(1+\eta)}\leq \wt C n^{\beta(\eta/(1+\eta)+1/\theta_m)-\eta/(1+\eta)},
	\ee 
	which tends to zero by the choice of $\beta$. It now follows that the right-hand-side of \eqref{eq:maxbound} tends to zero in probability. This implies, using Slutsky's theorem \cite{Vaart00}, that for any $\delta>0$, 
	\be\label{eq:maxgeqepsn}
	\lim_{n\rightarrow \infty}\P{E_{[n]\backslash[\eps_nn]}}=0
	\ee 
	For the first probability on the right-hand-side of \eqref{eq:epssplit}, we show that $\max_{i\leq \eps_n n}(\F_i/u_n)(n/i)^{1/\theta_m}$ tends to zero in probability when $n$ tends to infinity and that $\max_{i\leq \eps_n n}|(c^1_i/c^1_n)(i/n)^{1/\theta_m}-1|$ converges almost surely. We focus on the former first. The claim is proved by using the Poisson Point Process (PPP) weak limit. Recall $\Pi_n$ in \eqref{eq:pin} and its weak limit $\Pi$. We write 
	\be\label{eq:pppweaklimit}
	\Pi_n=\sum_{i=1}^n \delta_{(i/n,\F_i/u_n)}\Rightarrow \sum_{i\geq 1} \delta_{(t_i,f_i)}=:\Pi\quad \text{in }M_p(E),
	\ee
	where $\delta$ is a Dirac measure, and $\Pi$ is a PPP on $(0,1)\times(0,\infty)$ with intensity measure $\nu(\d t,\d x):=\d t \times (\alpha-1)x^{-\alpha}\d x$ \cite[Corollary 4.19]{Res13}. We now define $\Pi'$ to be the PPP on $\R_+$ obtained from mapping points $(t,f)\in\Pi$ to $ft^{-1/\theta_m}$ and let $\Pi_\eps'$ be the restriction of $\Pi'$ to points $(t,f)$ such that $t\leq \eps$. More formally, 
	\be
	\Pi':=\sum_{(t,f)\in \Pi}\delta_{(ft^{-1/\theta_m})},\qquad \Pi_\eps':=\sum_{(t,f)\in\Pi}\ind_{\{t\leq \eps\}}\delta_{(ft^{-1/\theta_m})}.
	\ee  
	Now, we fix an arbitrary $\delta,\eta>0$. Then, we can find an $\eps>0$ sufficiently small, such that 
	\be \ba\label{eq:lawpiprime}
	\P{\max_{(t,f)\in\Pi:t\leq \eps}f t^{-1/\theta_m}>\delta}&=1-\P{\Pi'_\eps((\delta,\infty))=0}\\
	&=1-\exp\Big\{\int_0^\eps\int_{\delta t^{1/\theta_m}}^\infty (\alpha-1)f^{-\alpha} \d f \d t\Big\}\\
	&=1-\exp\Big\{-\frac{\theta_m}{\theta_m-(\alpha-1)}\delta^{-(\alpha-1)}\eps^{(\theta_m-(\alpha-1))/\theta_m}\Big\}<\eta/2
	\ea\ee 
	is satisfied. Due to \eqref{eq:pppweaklimit} and the continuous mapping theorem, any continuous functional $T$ of $\Pi_n$ converges in distribution to $T(\Pi)$. We use this to compare the law of $\max_{i\leq \eps n}(\F_i/u_n)(i/n)^{-1/\theta_m}$ and $\max_{(t,f)\in\Pi:t\leq \eps}f t^{-1/\theta_m}$ by defining, for $\eps\in(0,1]$, the functional $T_\eps$, such that $T_\eps(\Pi):=\max_{(t,f)\in\Pi:t\leq \eps}ft^{-1/\theta_m}$. Let $M_k:=\{\Pi\in M_p(E)\;|\;T_\eps(\Pi)<k\},k\in\N $. Then, on $M_k$, $T_\eps$ is continuous, and thus $T_\eps$ is continuous on $\cup_{k\in\N}M_k$. Since the point processes $\Pi$ with intensity measure $\nu$ as described above are such that $T_\eps(\Pi)$ is finite almost surely, as follows from \eqref{eq:lawpiprime}, $\Pi \in M_k$ for some $k\in \N$ and thus $T_\eps$ is continuous with respect to $\Pi$ almost surely for any $\eps\in(0,1]$. It follows that, for $\delta,\eta$ fixed, $\eps$ chosen such that \eqref{eq:lawpiprime} holds and $n$ sufficiently large,
	\be
	\P{\max_{i\in [ \eps n]}\frac{\F_i}{u_n}(i/n)^{-1/\theta_m}>\delta}\leq\P{\max_{(t,f)\in\Pi:t\leq \eps}f t^{-1/\theta_m}>\delta}+\eta/2<\eta.
	\ee 
	As $\eps_n$ decreases monotonically, $ \eps_n<\eps$ for $n$ sufficiently large. Hence, it follows that for $n$ large,
	\be \label{eq:convprobmaxf}
	\P{\max_{i\in [ \eps_n n]}\frac{\F_i}{u_n}(i/n)^{-1/\theta_m}>\delta}\leq \P{\max_{i\leq \eps n}\frac{\F_i}{u_n}(i/n)^{-1/\theta_m}>\delta}<\eta.
	\ee 
	We therefore can conclude that $\max_{i\in [\eps_n n]}(\F_i/u_n)(i/n)^{1/\theta_m}\toinp 0$ as $n\rightarrow \infty$, as $\eta$ is arbitrary. We now show that $\max_{i\leq \eps_n n}|(c^1_i/c^1_n)(i/n)^{1/\theta_m}-1|$ converges almost surely. Because of Lemma \ref{lemma:ckn}, for each fixed $i\in\N$, $|(c^1_i/c^1_n)(i/n)^{1/\theta_m}-1|$ converges almost surely to some limit random variable $A_i$, and $A_i\neq A_j$ almost surely for all $i\neq j$. Using the lower and upper bound in \eqref{eq:c1nupperbound}, we obtain for every $i\geq n_0+1$ fixed and $n\geq i$,
	\be \ba
	\sup_{n\geq i}\bigg|\frac{c^1_i}{c^1_n}\Big(\frac{i}{n}\Big)^{1/\theta_m}-1\bigg|\leq{}& \max\bigg\{&&\hspace{-10pt}\frac{mk}{\E{\F}}\sum_{j=i}^{\infty}  \frac{|S_j/j-\E{\F}|}{m_0+m(j-n_0)+S_j}+\frac{m}{2}\sum_{j=i}^{\infty}\Big(\frac{k}{S_j}\Big)^2\\
	& &&\hspace{-10pt}+\frac{m_0+\E{\F}n_0+(m-1)}{\theta_m^2}\frac{\pi^2}{6i}+\frac{1}{\theta_m((i-(n_0+1))\vee 1)},\\
	&\exp\bigg\{&&\hspace{-13pt}\frac{k}{\theta_m}\log\Big(\frac{i}{(i-(n_0+1))\vee 1}\Big)+\frac{mk}{\E{\F}}\sum_{j=i}^\infty \frac{|S_j/j-\E{\F}|}{m_0+j-n_0+S_j}\bigg\}-1 \bigg\}\\
	=:{}&B_i.
	\ea \ee 
	As the sums in the maximum are almost surely finite for all $i\in\N$, as follows from the proof of Lemma \ref{lemma:ckn} and the strong law of large numbers, $\lim_{i\rightarrow\infty}B_i=0$ almost surely. Thus, combining the above steps with Lemma \ref{lemma:sequences}, we conclude that as $n\rightarrow\infty$,
	\be 
	\max_{\inn}\bigg|\frac{c^1_i}{c^1_n}\Big(\frac{i}{n}\Big)^{1/\theta_m}-1\bigg|\toas\sup_{i\geq 1}A_i,
	\ee 
	and there exist almost surely finite random variables $I,N$, such that the maximum is almost surely
	attained at index $i=I$ for all $n\geq N$. It thus follows that the maximum converges almost surely to an almost surely finite limit $A_I$. We can now conclude that, as $\eps_n n\rightarrow \infty$,
	\be 
	\max_{i\leq \eps_n n}\bigg|\frac{c^1_i}{c^1_n}\Big(\frac{i}{n}\Big)^{1/\theta_m}-1\bigg|\toas\sup_{i\geq 1}A_i=A_I,
	\ee 
	which, together with \eqref{eq:convprobmaxf}, yields 
	\be 
	\max_{i\leq \eps_n n}\frac{\F_i}{u_n}\Big(\frac{i}{n}\Big)^{-1/\theta_m}\max_{i\leq \eps_n n}\bigg|\frac{c^1_i}{c^1_n}\Big(\frac{i}{n}\Big)^{1/\theta_m}-1\bigg|\toinp 0.
	\ee 
	Combining this with \eqref{eq:epssplit} and \eqref{eq:maxgeqepsn}, we obtain \eqref{eq:maxc1norder}. By a similar argument as before, we find, 
	\be \label{eq:pppmaplimit}
	\max_{\inn}\frac{\F_i}{u_n}\Big(\Big(\frac{i}{n}\Big)^{-1/\theta_m}-1\Big)\toindis \max_{(t,f)\in\Pi}f(t^{-1/\theta_m}-1).
	\ee 
	Thus, combining \eqref{eq:maxlim}, \eqref{eq:maxc1norder} and \eqref{eq:pppmaplimit} and applying Slutsky's theorem \cite{Vaart00}, we arrive at the desired result. 
	
	We now prove \eqref{eq:infcondmeanconv} and so we let $\alpha\in(1,2)$. An important result is stated in Proposition \ref{prop:infmeanmaxconv}. By the construction of $\Pi_n$ in \eqref{eq:pin} and the definition of $T^{\eps}$ in \eqref{eq:Top}, it follows that
	\be 
	\frac{\F_i}{u_n}T^{i/n}(\Pi_n)=	\frac{\F_i}{u_n}\int_{i/n}^1 \Big(\int_E f\ind_{\{t\leq s\}}\d \Pi_n(t,f)\Big)^{-1}\d s=	\frac{\F_i}{u_n}\frac{1}{n}\sum_{j=i}^n\frac{u_n}{S_j}=	\frac{\F_i}{n}\sum_{j=i}^n\frac{1}{S_j},
	\ee 
	as for $s\in[j/n,(j+1)/n)$ the integrand is constant. Hence, by Proposition \ref{prop:infmeanmaxconv}, what remains is to prove that
	\be \label{eq:infmaxmean}
	\bigg|\max_{\inn}\Ef{}{\Zm_n(i)/n}-\max_{\inn}\frac{\F_i}{n}\sum_{j=i}^n m/S_j\bigg|\toinp 0.
	\ee 
	Recall the result in \eqref{eq:maxlim} regarding the limit of the maximum conditional mean. The above is therefore implied by the following two statements:
	\be \ba\label{eq:c1nmaxconv}
	\bigg|\max_{\inn}\frac{\F_i}{n}\Big(\frac{c^1_i}{c^1_n}-1\Big)-\max_{\inn}\frac{\F_i}{n}\sum_{j=i}^n m/(m_0+m(j-n_0)+S_j)\bigg|\toinp 0,\\
	\bigg|\max_{\inn}\frac{\F_i}{n}\sum_{j=i}^n m/(m_0+m(j-n_0)+S_j)-\max_{\inn}\frac{\F_i}{n}\sum_{j=i}^n m/S_j\bigg|\toinp 0.
	\ea\ee 
	We start by proving the first line of \eqref{eq:c1nmaxconv}. Let us write $Z_j:=m_0+m(j-n_0)+S_j$. By \eqref{eq:maxabsdif}, it follows that 
	\be
	\bigg|\max_{\inn}\frac{\F_i}{n}(c^1_i/c^1_n-1)-\max_{\inn}\frac{\F_i}{n}\sum_{j=i}^n m/Z_j\bigg|\leq \max_{\inn}\frac{\F_i}{n}\bigg((c^1_i/c^1_n-1)-\sum_{j=i}^n m/Z_j\bigg),
	\ee 
	as the terms within the brackets on the right-hand-side are a.s.\ positive. Then, we further bound the expression on the right-hand-side from above by splitting the maximum into two parts, as
	\be \ba\label{eq:splitmax}
	\max_{\inn}\frac{\F_i}{n}\bigg((c^1_i/c^1_n-1)-\sum_{j=i}^n m/Z_j\bigg)\leq{}& \max_{i\in [ i_n]}\frac{\F_i}{n}\bigg((c^1_i/c^1_n-1)-\sum_{j=i}^n m/Z_j\bigg)\\
	&+\max_{i_n\leq i\leq n}\frac{\F_i}{n}\bigg((c^1_i/c^1_n-1)-\sum_{j=i}^n m/Z_j\bigg),
	\ea \ee  
	where $i_n$ is strictly increasing and tends to infinity with $n$. We first investigate the second maximum, by bounding the terms within the brackets. Namely,
	\be \ba
	(c^1_i/c^1_n-1)-\sum_{j=i}^n m/Z_j\leq \exp\bigg\{\sum_{j=i}^n m/Z_j \bigg\}-1-\sum_{j=i}^n m/Z_j=\sum_{k=2}^\infty \bigg(\sum_{j=i}^n m/Z_j\bigg)^k.
	\ea \ee 
	Now, fix $\eps>0$. By \eqref{eq:Miprobbound} there exists an almost surely finite random variable $J$ such that for all $j\geq J$, $M_j\geq j^{1/(\alpha-1)-\eps}$, with $M_j=\max_{k\leq j}\F_k$. So, on $\{i\geq J\}$, $Z_j\geq j^{1/(\alpha-1)-\eps}$ for all $j\geq i$. This yields the upper bound
	\be\label{eq:upbounddiffc1n}
	\sum_{k=2}^\infty m i^{-k((2-\alpha)/(\alpha-1)-\eps)} \leq C i^{-2((2-\alpha)/(\alpha-1)-\eps)},
	\ee 
	for some constant $C>0$, as we can bound an exponentially decaying sum by a constant times its first term. It follows, on $i_n\geq J$, which holds with high probability, and by \eqref{eq:upbounddiffc1n}, that
	\be \label{eq:boundatin}
	\max_{i_n\leq i\leq n}\frac{\F_i}{n}\bigg((c^1_i/c^1_n-1)-\sum_{j=i}^n m/Z_j\bigg)\leq C i_n^{-2((2-\alpha)/(\alpha-1)-\eps)} \frac{u_n}{n} \max_{i_n\leq i\leq n}\frac{\F_i}{u_n},
	\ee 
	which tends to zero in probability when $i_n^{-2((2-\alpha)/(\alpha-1)-\eps)} u_n/n=o(1)$, that is, when $i_n=n^\rho$, with $\rho\in(1/2,1)$. On the other hand, when considering the first maximum in \eqref{eq:splitmax}, we find
	\be \label{eq:boundat1}
	\max_{i\in[i_n]}\frac{\F_i}{n}\bigg((c^1_i/c^1_n-1)-\sum_{j=i}^n m/Z_j\bigg)\leq (1/c^1_n)\frac{u_{i_n}}{n}\max_{i\leq i_n}\frac{\F_i}{u_{i_n}},
	\ee 
	where we bound the terms inside the brackets on the left-hand-side by omitting all negative terms and by noting that $c^1_i\leq 1$ for all $i$. The right-hand-side of \eqref{eq:boundat1} converges to zero in probability when $u_{i_n}/n=o(1)$, that is, when $i_n=n^\rho$ with $\rho<\alpha-1$, since $c^1_n$ converges almost surely for $\alpha\in(1,2)$ by Lemma \ref{lemma:ckn}. We conclude that for $\alpha\in(3/2,2)$ we can find a $\rho\in(1/2,\alpha-1)$ such that both maxima tend to zero in probability. When $\alpha\in(1,3/2]$, however, such a $\rho$ cannot be found and more work is required to prove the desired result. In this case, we split the maximum into $K=K(\alpha)<\infty$ maxima, as follows: Let $A_{i,n}:=\F_i/n, B_{i,n}:=(c^1_i/c^1_n-1)-\sum_{j=i}^n m/Z_j$. Then, we define $i^k_n:=n^{\rho_k}$, $k=0,1,\ldots,K$, with $\rho_0=0,\rho_K=1$, and 
	\be
	\rho_k:=\frac{\alpha-1}{2}\frac{c^k-1}{c-1},\qquad k\in\{1,2,\ldots,K-1\},
	\ee
	where $c:=2(2-\alpha)-2\eps(\alpha-1)\neq 1$. Note that $\rho_k$ is strictly increasing in $k$, independent of $c<1$ or $c>1$. We now write
	\be\label{eq:summaxK}
	\max_{\inn}A_{i,n}B_{i,n}\leq \sum_{k=0}^{K-1} \max_{i^k_n\leq i\leq i^{k+1}_n} A_{i,n}\max_{i^k_n\leq i\leq i^{k+1}_n}B_{i,n}.
	\ee
	We first deal with the $k=0$ term. As in \eqref{eq:boundat1}, since $\rho_1<\alpha-1$, $\max_{i^0_n\leq i\leq i^{1}_n} A_{i,n}\max_{i^0_n\leq i\leq i^1_n}B_{i,n}$ tends to zero in probability. For $k=1,\ldots,K-2$, following the same steps that lead to the bound in \eqref{eq:boundatin}, we obtain
	\be 
	\max_{i^k_n\leq i \leq i^{k+1}_n}A_{i,n}\max_{i^k_n\leq i \leq i^{k+1}_n}B_{i,n}\leq C_k (i^k_n)^{-2((2-\alpha)/(\alpha-1)-\eps)}\frac{u_{i^{k+1}_n}}{n}\max_{i^k_n\leq i\leq i^{k+1}_n}\frac{\F_i}{u_{i^{k+1}_n}},
	\ee 
	for some constant $C_k>0$. This upper bound tends to zero in probability when
	\be\label{eq:rhobounds}
	\rho_{k+1}<\alpha-1+(2(2-\alpha)-2\eps(\alpha-1))\rho_{k}=(\alpha-1)+c\rho_{k}
	\ee 	
	is satisfied.  By the definition of $\rho_k$, this holds when
	\be 
	\frac{c^{k+1}-1}{c-1}-2<c\frac{c^k-1}{c-1} \ \Leftrightarrow -1+ \sum_{j=1}^k c^j < \sum_{j=1}^k c^j,
	\ee 
	which is indeed the case. Finally, for $k=K-1$, again using the similar bound as in \eqref{eq:boundatin}, we find that the final term of the sum in \eqref{eq:summaxK} converges to zero in probability when $\rho_{K-1}\in(1/2,1)$. What remains to show, is that for all $\alpha\in(1,3/2]$ there does exist a finite $K$ such that $\rho_{K-1}\in(1/2,1)$. We distinguish two cases: $\alpha=3/2$ and $\alpha\in(1,3/2)$. For the first case, $c<1$ for any choice of $\eps$. This implies that $\rho_k\to 1/(2\eps)$ as $k$ tends to infinity, so taking $\eps<1$ suffices. For $\alpha\in(1,3/2)$, we can choose $\eps$ sufficiently small, such that $c>1$, so that $\rho_k$ diverges. In both cases there therefore exists a $K$ such that $\rho_k>1/2$ for all $k\geq K-1$. Thus, in both cases, we can define $K:=\inf\{k\in\N\,|\, \rho_{k}>1/2\}+1$. The only issue left to address regarding $K$, is that it is possible that $\rho_{K-1}>1$. However, in that case we can simply choose $\rho_{K-1}=a$, for any $a\in(1/2,1)$, since $\rho_{K-2}\leq 1/2<a$ by the definition of $K$, and decreasing $\rho_{K-1}$ does not violate the constraint in \eqref{eq:rhobounds} for $k=K-2$. We hence obtain the first line in \eqref{eq:c1nmaxconv}.
	
	The proof for the second line in \eqref{eq:c1nmaxconv} follows similarly. First, by letting $i=i(n)$ tend to infinity with $n$, we bound, conditional on $\{i\geq J\}$,
	\be \ba\label{eq:invsumdiff}
	\bigg|\sum_{j=i}^n m/S_j - \sum_{j=i}^n m/Z_j \bigg|&\leq C\sum_{j=i}^n j/M_j^2 \leq C\sum_{j=i}^n j^{1-2/(\alpha-1)+\eps}\leq \wt C i^{-2((2-\alpha)/(\alpha-1)-\eps/2)},
	\ea \ee 
	for some constant $C\geq m+m_0$. We note that this bound is similar to the upper bound for $(c^1_i/c^1_n-1)-\sum_{j=i}^n 1/(j+S_j/m)$ in \eqref{eq:upbounddiffc1n}. Also, both sums on the left-hand-side of \eqref{eq:invsumdiff} converge almost surely, as $\alpha\in(1,2)$. Thus, a similar approach, with the same indices $i_n^0,\ldots,i_n^K$ can be used to obtain the desired result. Combining both statements in \eqref{eq:c1nmaxconv} and using the triangle inequality and the continuous mapping theorem proves \eqref{eq:infmaxmean}, which together with Proposition \ref{prop:infmeanmaxconv} finishes the proof.
	\end{proof}	

	We now prove Proposition \ref{lemma:concentration}:

	\begin{proof}[Proof of Proposition \ref{lemma:concentration}]
	
	The focus of the proof is on  the PAFUD model, for which we use the martingales $M^k_n(i)$. The proof for  the PAFRO model follows by setting $m=1$, and for the PAFFD model it follows in a similar fashion, where all upper bounds still hold when the supermartingale $\wt M^k_n(i)$ is to be used. We prove \eqref{eq:conc} first. Applying \eqref{eq:maxabsdif}, a $p^{\text{th}}$ moment bound for some $p>1$ to be determined later, using Markov's inequality and H\"older's inequality yields
	\be \ba \label{eq:maxconc}
	\mathbb{P}_\F(|\max_{\inn}\Zm_n(i)-\max_{\inn}\Ef{}{\Zm_n(i)}|>\eta u_n)&\leq \mathbb{P}_\F\Big(\max_{\inn}|\Zm_n(i)-\Ef{}{\Zm_n(i)}|>\eta u_n\Big)\\
	&\leq (\eta u_n)^{-p}\sum_{i=1}^n \Ef{\big}{|\zni-\Ef{}{\zni}|^p}\\
	&\leq (\eta u_n)^{-p}\sum_{i=1}^n \mathbb{E}_\F\big[|\zni-\Ef{}{\zni}|^{2k}\big]^{p/(2k)},
	\ea \ee 
	where $k>p/2$ is an integer. As $\zni-\Ef{}{\zni}=(\zni+\F_i)-\Ef{}{\zni+\F_i}$ and $2k$ is even, we find, using H\"older's and Jensen's inequality and setting $X_n(i):=\zni+\F_i$,
	\be \ba
	\Ef{\big}{|\zni-\Ef{}{\zni}|^{2k}}={}&\sum_{j=0}^{2k}{2k \choose j}\Ef{}{X_n(i)^{j}}(-1)^j \Ef{}{X_n(i)}^{2k-j}\\
	={}&\sum_{j=0}^k{2k \choose 2j}\Ef{}{X_n(i)^{2j}}\Ef{}{X_n(i)}^{2k-2j}\\&-\sum_{j=1}^k{2k\choose 2j-1}\Ef{}{X_n(i)^{2j-1}}\Ef{}{X_n(i)}^{2k-(2j-1)}\\
	\leq{}& \sum_{j=0}^k {2k\choose 2j}\Ef{}{X_n(i)^{2k}}-\sum_{j=1}^k {2k \choose 2j-1}\Ef{}{X_n(i)}^{2k}.
	\ea\ee 
	Using that 
	\be 
	\sum_{j=0}^{2k}{2k\choose j}=2^{2k},\qquad \sum_{j=0}^{2k}{2k \choose j}(-1)^j=0,
	\ee
	it follows that both sums in the last line of \eqref{eq:maxconc} equal $2^{2k-1}$. We can thus bound \eqref{eq:maxconc} from above by
	\be \label{eq:concub}
	\frac{2^{2k-1}}{(\eta u_n)^p}\sum_{i=1}^n (\Ef{\big}{(\zni+\F_i)^{2k}}-\Ef{}{\zni+\F_i}^{2k})^{p/(2k)}.
	\ee 
	We now aim to bound the $2k^{\text{th}}$ moment of $\zni+\F_i$. Since, for $x\geq 0,k\in\N$, $x^{2k}\leq \prod_{j=1}^{2k}(x-(j-1))={x+(2k-1)\choose 2k}(2k)!$, it follows from Lemma \ref{lemma:martingale} that
	\be 
	\mathbb{E}_\F\big[(\zni+\F_i)^{2k}\big]\leq(c^{2k}_n)^{-1}(2k)!\Ef{}{M_n^{2k}(i)}=\frac{c^{2k}_{i\vee n_0}}{c^{2k}_n}(2k)!{\Zm_{i\vee n_0}(i)+\F_i+2k-1\choose 2k}.
	\ee 
	We note that this inequality would still hold for the PAFFD model, when using the supermartingales $\wt M^k_n(i)$ and the sequences $\wt c^k_n(i)$. We thus obtain the upper bound
	\be \ba
	\Ef{\big}{(\zni+\F_i)^{2k}}\leq\frac{c^{2k}_{i\vee n_0}}{c^{2k}_n}(\Zm_{i\vee n_0}+\F_i)^{2k}+\frac{c^{2k}_{i\vee n_0}}{c^{2k}_n}P_{2k-1}(\Zm_{i\vee n_0}(i)+\F_i),
	\ea\ee  
	where $P_{2k-1}(x)=(2k)!{x+2k-1\choose 2k}-x^{2k}$ is a polynomial of degree $2k-1$. Using \eqref{eq:condmean}, we find
	\be \ba\label{eq:2kmomentdiff}
	\Ef{\big}{(\zni+\F_i)^{2k}}-\Ef{}{\zni+\F_i}^{2k}\leq{}&\Big(\frac{c^{2k}_{i\vee n_0}}{c^{2k}_n}-\Big(\frac{c^1_{i\vee n_0}}{c^1_n}\Big)^{2k}\Big)(\Zm_{i\vee n_0}(i)+\F_i)^{2k}\\
	&+\frac{c^{2k}_{i\vee n_0}}{c^{2k}_n}P_{2k-1}(\Zm_{i\vee n_0}(i)+\F_i).
	\ea\ee 
	Using the definition of $c^k_n$ in \eqref{eq:c1n} yields, for all $1\leq r\leq n$,
	\be\ba \label{eq:c1ndiffbound}
	\frac{c^{2k}_{r}}{c^{2k}_n}&=\prod_{j=r\vee n_0}^{n-1}\prod_{\ell=1}^m \Big(1+\frac{2k}{m_0+m(j-n_0)+(\ell-1)+S_j}\Big)\\
	&\leq\prod_{j=r\vee n_0}^{n-1}\prod_{\ell=1}^m\Big(1+\frac{1}{m_0+m(j-n_0)+(\ell-1)+S_j}\Big)^{2k}=\Big(\frac{c^1_{r}}{c^1_n}\Big)^{2k}. 
	\ea\ee 
	Therefore, using this in \eqref{eq:2kmomentdiff} we obtain an upper bound that contains powers of $\F_i$ of order at most $2k-1$. This is the essential step to proving concentration holds. Namely, in \eqref{eq:concub}, this upper bound yields an expression with powers of $\F_i$ of order at most  $p(1-1/2k)$, which is just slightly less than $p$. The aim is, for every value of $\alpha>2$, to find values $p,k$ such that the $p(1-1/2k)^{\text{th}}$ moment of $\F$ exists and such that the entire expression in \eqref{eq:concub} still tends to zero. 
	
	Let us write
	\be 
	P_{2k-1}(x)=\sum_{\ell=0}^{2k-1}C_\ell x^\ell,
	\ee
	for non-negative constants $C_\ell$. Combining \eqref{eq:2kmomentdiff} and \eqref{eq:c1ndiffbound} in \eqref{eq:concub}, bounding $\Zm_{i\vee n_0}(i)$ from above by $m_0$ and recalling that $p/(2k)<1$, results in the upper bound
	\be \label{eq:aink}
	\frac{2^{2k-1}}{(\eta u_n)^p}\sum_{i=1}^n \Big(\frac{c^{2k}_i}{c^{2k}_n}\Big)^{p/(2k)}\sum_{\ell=0}^{2k-1}\wt C_{\ell}^{p/(2k)}\F_i^{\ell p/2k},
	\ee
	where the $\wt C_{\ell}>0$ are constants. We focus on the term where $\ell=2k-1$, as this is the boundary case. All other cases follow analogously. For the first $n_0$ terms, we can bound $c_i^{2k}$ from above by $(i/n_0)^{-p/\theta_m}$. For $n_0+1\leq i\leq n$, we use \eqref{eq:c1nupperbound} to bound $c^{2k}_i/c^{2k}_n$ from above. This yields for all terms, for some constant $C>0$,
	\be \ba\label{eq:concfinalub}
	\frac{\wt C_{2k-1}^{p/(2k)}2^{2k-1}}{(\eta u_n)^p}&\bigg(\exp\bigg\{-C+\frac{mp}{\E{\F}}\sum_{j=n_0}^\infty \frac{|S_j/j-\E{\F}|}{j-n_0+S_j}\bigg\}\vee 1\bigg)\sum_{i=1}^n \Big(\frac{i}{nn_0}\Big)^{-p/\theta_m}\F_i^{p(1-1/2k)} \\
	&\leq C_{k,p,\theta_m} \exp\bigg\{\frac{mp}{\E{\F}}\sum_{j=n_0}^\infty \frac{|S_j/j-\E{\F}|}{j-n_0+S_j}\bigg\}\frac{n^{p/\theta_m}}{u_n^p}\sum_{i=1}^n i^{-p/\theta_m}\F_i^{p(1-1/2k)},
	\ea\ee 
	for some constant $C_{k,p,\theta_m}$. In the last line, the exponential term is almost surely finite, as follows from the proof of Lemma \ref{lemma:ckn}. We now show that the fraction multiplied by the sum converges to zero in mean when $p$ and $k$ are chosen in a specific way. That is, for $\alpha>2$, set $p:=(1+\eps)(\alpha-1)$, where $\eps\in(0,1/(\alpha+1))$ and set $k:=\lceil p/2\rceil$. First note that $2k>p$, which was required for the H\"older inequality used in \eqref{eq:maxconc}. We now show that the $p(1-1/(2k))^{\text{th}}$ moment of the fitness distribution exists. For this to hold, $\alpha-1>p(1-1/(2k))$ needs to be satisfied, or, equivalently, 
	\be
	k<\frac{p}{2(p-(\alpha-1))}=\frac{1+\eps}{2\eps},
	\ee 
	and, as $\eps\in(0,1/(\alpha+1))$,
	\be 
	\frac{1+\eps}{2\eps}-\frac{p}{2}=\frac{1+\eps}{2}(1/\eps-(\alpha-1))>1+\eps.
	\ee 
	It follows that, indeed,
	\be
	(1+\eps)/(2\eps)>p/2+1+\eps>\lceil p/2\rceil =k.
	\ee 
	Hence, taking the mean, we obtain
	\be 
	\frac{n^{p/\theta_m}}{u_n^p}\sum_{i=1}^n i^{-p/\theta_m}\E{\F_i^{p(1-1/(2k))}}\leq C\frac{n^{p/\theta_m}}{u_n^p}n^{(1-p/\theta_m)\vee 0},
	\ee 
	with $C>0$ a constant. This tends to zero with $n$, as $u_n=n^{1/(\alpha-1)}\wt \ell(n)$ for some slowly-varying function $\wt\ell(n)$, and both $p>\alpha-1$ and $\theta_m>\alpha-1$ hold. So, the last expression in \eqref{eq:concfinalub} consists of an almost surely finite random variable (the exponential term) and a term that converges to zero mean, which implies that the entire expression converges to zero in probability. The same argument holds also for all other values of $\ell$ in \eqref{eq:aink}. Thus, as $n$ tends to infinity,
	\be \label{eq:maxconccond}
	\mathbb{P}_\F\Big(|\max_{\inn}\Zm_n(i)-\max_{\inn}\Ef{}{\Zm_n(i)}|>\eta u_n\Big)\toinp 0.
	\ee 
	As this conditional probability measure is bounded from above by one, it follows from the dominated convergence theorem and \eqref{eq:maxconccond} that \eqref{eq:conc} holds. 
	
	We now prove \eqref{eq:infmeanconc}, so let $\alpha\in(1,2)$. A different approached is required, so we write, using \eqref{eq:maxabsdif}, a union bound and Chebyshev's inequality,
	\be \ba \label{eq:maxmartbound}
	\mathbb{P}_\F\Big(|\max_{\inn}\Zm_n(i)-\max_{\inn}\Ef{}{\Zm_n(i)}|>\eta u_n\Big)&\leq \mathbb{P}_\F\Big(\max_{\inn}|\Zm_n(i)-\Ef{}{\Zm_n(i)}|>\eta u_n\Big)\\
	&\leq \sum_{i=1}^n\Pf{|M_n^1(i)-\Ef{}{M_n^1(i)}|\geq \eta u_n c^1_n}\\
	&\leq(\eta u_n c^1_n)^{-2}\sum_{i=1}^n \Var_\F(M^1_n(i)).
	\ea\ee
	We now use the martingale property to split the variance in the variance of martingale increments. To this end, we need to introduce some notation. Recall that $\Zm_{n,j}(i)$ is the degree of $i$ in $\G_{n,j}$, the graph with $n$ vertices where the $n+1^{\mathrm{st}}$ vertex has connected $j$ half-edges with the first $n$ vertices. Now, let us write
	\be\ba 
	c^1_{n,j}(m)&:=\prod_{r=n_0}^{n-1}\prod_{\ell=1}^j \Big(1-\frac{1}{m_0+m(r-n_0)+(\ell-1)+1+S_r}\Big), \\ M^1_{n,j}(i)&:=c^1_{n-1,j}(m)(\Zm_{n-1,j}(i)+\F_i).
	\ea \ee 
	If we let $M_\ell:=M^1_{n,j}(i)$, where $n\geq n_0,j\in[m]$ are such that $mn+(j-1)=\ell$, it follows from the proof of Lemma \ref{lemma:martingale} that $M_\ell$ is a martingale for the PAFRO and PAFUD model. Hence, we can then write the conditional variance of $M^1_n(i)$ as in \eqref{eq:maxmartbound} as
	\be \label{eq:martincr}
	\Var_\F(M^1_n(i))=\sum_{k=i+1\vee n_0+1}^n \sum_{j=1}^m \Var_\F(\Delta M^1_{k,j}(i)),
	\ee 
	where $\Delta  M^1_{k,j}(i):= M^1_{k,j}(i)- M^1_{k,j-1}(i)$, and where we note that $M^1_{k,0}(i)=M^1_{k-1,m}(i)=M^1_k(i)$ for all $k=i\vee n_0,\ldots,n$. We then obtain
	\be \ba\label{eq:martvar}
	\Var_\F(\Delta M^1_{k,j}(i)) =(c^1_{k,j-1})^2\Ef{\Big}{\Big(\ind_{k,j,i}-\frac{\Zm_{k-1,j-1}(i)+\F_i+\ind_{k,j,i}}{m_0+m((k-1)-n_0)+(j-1)+1+S_{k-1}}\Big)^2},
	\ea\ee 
	where $\ind_{k,j,i}$ is the indicator of the event that vertex $k$ connects its $j^{\mathrm{th}}$ half-edge to vertex $i$. We rewrite this to find the upper bound
	\be \ba \label{eq:condvar}
	\Var_\F(\Delta M^1_{k,j}(i))&\leq \Ef{\Big}{\Big(\ind_{k,j,i}-\frac{\Zm_{k-1,j-1}(i)+\F_i}{m_0+m((k-1)-n_0)+(j-1)+S_{k-1}}\Big)^2}\\
	&=\Ef{\big}{\Var(\ind_{k,j,i}\,|\,\G_{k-1,j-1})}\\
	&\leq \Ef{\Big}{\frac{\Zm_{k-1,j-1}(i)+\F_i}{m_0+m((k-1)-n_0)+(j-1)+S_{k-1}}}.
	\ea\ee  
	Combining this with \eqref{eq:maxmartbound} and \eqref{eq:martincr} and switching summations yields
	\be
	\mathbb{P}_\F\Big(|\max_{\inn}\Zm_n(i)-\max_{\inn}\Ef{}{\Zm_n(i)}|>\eta u_n\Big)\leq (\eta u_n c^1_n)^{-2}mn,
	\ee 
	This final expression tends to zero almost surely, as $c^1_n$ converges almost surely when $\alpha\in(1,2)$, as follows from Lemma \ref{lemma:ckn}. For the PAFFD model, we can use similar steps. We construct $\wt M_\ell:=\wt M^1_{n,j}(i)$ as above, with $\wt M^1_{n,j}:=\wt c^1_{n-1,j}(m)(\Zm_{n-1,j}(i)+\F_i)$, and 
	\be 
	\wt c_{n,j}(m):=\prod_{r=n_0}^{n-1}\Big(1-\frac{1}{m_0+m(r-n_0)+S_r}\Big)^j.
	\ee
	It again follows from the proof of Lemma \ref{lemma:martingale} that $\wt M_\ell$ is a supermartingale, thus yielding \eqref{eq:martincr} for $\wt M_{n}(i)$. Then, all further steps can be applied for the PAFFD model as well, where the equality in \eqref{eq:martincr} becomes an upper bound and the denominator of the fractions in \eqref{eq:martvar} and \eqref{eq:condvar} changes to $m_0+m((k-1)-n_0)+S_{k-1}$.
	  
	For the PAFRO model, an adapted final step is required, as the conditional moments in \eqref{eq:condvar} do not sum to one (when summing over $i$ from $1$ to $k-1$). Rather, we set $m$ to $1$ and follow the same steps up to \eqref{eq:condvar}. Then, we obtain by switching the summations,
	\be 
	(\eta u_n c^1_n)^{-2}\sum_{i=1}^n \sum_{k=i+1\vee n_0+1}^n \Ef{}{\Var(\ind_{k,1,i}\;|\;\G_{k-1})}\leq (\eta u_n c^1_n)^{-2} \sum_{k=n_0+1}^n (k-1)\leq \frac{n^2}{2(\eta u_n c^1_n)^2}.
	\ee  
	Again, $c^1_n$ converges almost surely. Hence, the right-hand-side tends to zero almost surely, since $1-1/(\alpha-1)<0$. Thus, for all models the conditional probability of concentration tends to zero with $n$. 
	
	Finally, like the argument made above \eqref{eq:maxmartbound}, applying the dominated convergence theorem proves \eqref{eq:infmeanconc}, which concludes the proof.
	\end{proof}	

	\section{Proof of the maximum degree growth theorem}\label{sec:mainproof}
	In this section, we use the results from Section \ref{sec:infmean} and \ref{sec:ppp}, in particular Propositions \ref{lemma:condmeanconv} and \ref{lemma:concentration}, to prove Theorem \ref{Thrm:maxdegree}. 
	
	\begin{proof}[Proof of Theorem \ref{Thrm:maxdegree}]
		We start by proving $(i)$ and $(ii)$. This directly follows from Lemmas \ref{lemma:martingale} and \ref{lemma:ckn}. As discussed after Lemma \ref{lemma:martingale}, the martingales (resp.\ supermartingales) $M^k_n(i)$ (resp.\ $\wt M^k_n(i)$) converge almost surely to $\xi_i^k$ (resp.\ $\wt \xi_i^k$). Also, for the PAFFD model, $M^1_n(i)$ converges almost surely to $\xi^1_i$ as well. By these two lemmas, $c^1_n\zni=M^1_n(i)-c^1_n\F_i$ converges almost surely to $\xi_i^1$ for the PAFRO and PAFUD models, $\wt c^1_n \zni=\wt M^1_n(i)-\wt c^1_n\F_i$ converges almost surely to $\wt \xi^1_i$ for the PAFFD model and $c^1_n(m) n^{1/\theta_m}$ and $\wt c^1_n(m) n^{1/\theta_m}$ converge almost surely to $c_1,\wt c_1$, respectively, when $\E{\F^{1+\eps}}<\infty$ for some $\eps>0$. Hence, we can set $\xi_i:=(c_1)^{-1}\xi_i^1$ for the PAFRO (note $m=1$) and the PAFUD model, and $\xi_i:=(\wt c_1)^{-1}\wt \xi^1_i$ for the PAFFD model. Since $c_1$ and $\wt c_1$ are finite almost surely, it follows directly from Lemma \ref{lemma:noatom} that $\xi_i$ has no atom at zero for all $i\in\N$ for any of the three models.

		When $\alpha\in(1,2)$, we note that $c^1_n\toas \underline c_1$ without the need of rescaling and thus \eqref{eq:maxconv2} follows with $\Zm_{\infty}(i):=\xi_i^1/\underline c_1-\F_i$, as $\zni=M^1_n(i)/c^1_n-\F_i$, for the PAFRO and PAFUD models and $\Zm_{\infty}(i):=\wt\xi^1_i/\underline {\wt c}_1 -\F_i$ for the PAFFD model.
		
		We now prove $(iii)$. From the second inequality in \eqref{eq:cnkineq} we obtain $(c^1_n)^k \leq c^k_n$ when $k\geq 1$. Furthermore, from \cite[Theorem 1]{Jam13} it follows that $x^k\leq \Gamma(x+k)/\Gamma(x)$ for all $x>0,k\geq 1$. Hence, $(c^1_n\zni)^k \leq c^k_n(\zni+\F_i)^k\leq M^k_n(i)\Gamma(k+1)$ for $k\geq 1$. Recall $M$ from Lemma \ref{lemma:supmkn}. Clearly, $M>\theta_m$ when $\E{\F^{\theta_m+\eps}}<\infty$ for some $\eps>0$. So, if we let $k\in(\theta_m,M)$, Lemma \ref{lemma:supmkn} yields
		\be
		\lim_{i\to\infty}\sup_{n\geq i}c^1_n \zni =0 \text{ almost surely},
		\ee  
		as $M>\theta_m$ when $\E{\F^{\theta_m+\eps}}<\infty$ for some $\eps>0$.
		It then follows from Lemma \ref{lemma:sequences}, as $c^1_n\zni\toas \xi^1_i$ and $\xi^1_i\neq \xi_j^1$ almost surely for $i\neq j$,
		\be 
		\max_{\inn} n^{-1/\theta_m}\zni=(n^{1/\theta_m}c^1_n)^{-1}\max_{\inn}c^1_n\zni \toas (c_1)^{-1}\sup_{i\geq 1}\xi^1_i=\sup_{i\geq 1}\xi_i,\quad\text{and}\quad  I_n\toas I,
		\ee 
		for some almost surely finite random variable $I$. The same approach with $\wt M^k_n(i)$ holds for the PAFFD model. We now turn to the convergence of $\max_{\inn}\zni/u_n$ and $\max_{\inn}\zni/n$ as in $(iv)$ and $(v)$. This follows immediately by applying Slutsky's theorem to the results in Propositions \ref{lemma:condmeanconv} and \ref{lemma:concentration}. For the convergence of $I_n/n$ as in \eqref{eq:ppplimit}, \eqref{eq:lawI} and \eqref{eq:infmeanppplimit}, we let $0\leq a<b\leq 1$, and define, using $z(t,f):=f(t^{-1/\theta_m}-1)$, the random variables
		\be\ba
		Q_\ell(a)&:=\max_{(t,f)\in\Pi:0<t<a}z(t,f),\quad Q(a,b):=\max_{(t,f)\in\Pi:a<t<b}z(t,f),\quad Q_r(b):=\max_{(t,f)\in\Pi:b<t<1}z(t,f),
		\ea\ee
		and events
		\be \ba\label{eq:mevents}
		M_n(a,b)&:=\Big\{\max_{an<i<bn}\zni/u_n>(\max_{1\leq i\leq an}\zni/u_n \vee \max_{bn\leq i\leq n}\zni/u_n)\Big\},\\
		M(a,b)&:=\Big\{Q(a,b)>Q_\ell(a)\vee Q_r(b)\Big\}.
		\ea\ee
		We can then conclude, for $\alpha\in(2,1+\theta_m)$,
		\be \ba\label{eq:inppplimit}
		\lim_{n\to\infty}\P{I_n/n\in(a,b)}=\lim_{n\to\infty}\mathbb{P}(M_n(a,b))=\P{M(a,b)},
		\ea\ee 
		since it follows from the proof of Propositions \ref{lemma:condmeanconv} and \ref{lemma:concentration} that the vector $(\zni/u_n)_{i\in[n]}$ converges in distribution when $\alpha\in(2,1+\theta_m)$. Now, by the fact that $\Pi$ is a PPP with intensity measure $\nu(\d t\times \d x)=\d t\times (\alpha-1)x^{-\alpha}\d x$, we find
		\be \ba\label{eq:lawqab}
		\P{Q(a,b)\leq x}&=\exp\bigg\{-\int_a^b \int_{x(t^{-1/\theta_m}-1)^{-1}}^\infty (\alpha-1)s^{-\alpha}\d s\d t\bigg\}=\exp\{-g(a,b) x^{-(\alpha-1)}\},
		\ea \ee 
		where $g(a,b):=\int_a^b (t^{-1/\theta_m}-1)^{\alpha-1}\d t<\infty$ for all $0\leq a\leq b\leq 1$. Similarly, using the independence property of PPPs, 
		\be \label{eq:lawmaxqlqr}
		\P{Q_\ell(a)\vee Q_r(b)\leq x}=\exp\{-(g(0,a)+g(b,1))x^{-(\alpha-1)}\}.
		\ee  
		Combining \eqref{eq:lawqab} and \eqref{eq:lawmaxqlqr} in \eqref{eq:inppplimit} by conditioning on $Q_\ell(a)\vee Q_r(b)$, we obtain
		\be 
		\lim_{n\to\infty}\P{I_n/n\in (a,b)}=1-\int_0^\infty (\alpha-1)x^{-\alpha}(g(0,a)+g(b,1))\exp\{-g(0,1)x^{-(\alpha-1)}\}\d x =\frac{g(a,b)}{g(0,1)},
		\ee 
		which yields the required result. Via a similar approach, redefining $M_n(a,b)$ and $M(a,b)$ accordingly for $\alpha\in(1,2)$, we can show $I_n/n$ converges in distribution when $\alpha\in(1,2)$, though it is not possible to find a closed expression for the law of $I$. Finally, we address the joint convergence of $(I_n/n,\max_{\inn}\zni/u_n)$. We let $0<c<d<\infty$ and define the events
		\be\label{eq:eevents}
		E_n(a,b,c,d)=:\Big\{\max_{an<i<bn}\zni/u_n\in(c,d)\Big\},\quad E(a,b,c,d):=\Big\{Q(a,b)\in(c,d)\Big\}.
		\ee 
		We can then write, using these events and the events in \eqref{eq:mevents} and letting $A:=(a,b)\times(c,d)$,
		\be 
		\mathbb{P}\Big((I_n/n,\max_{\inn}\zni/u_n)\in A\Big)=\P{M_n(a,b)\cap E_n(a,b,c,d)},
		\ee
		which converges to $\P{M(a,b)\cap E(a,b,c,d)}$ as $n$ tends to infinity by the same argument as provided for the limit in \eqref{eq:inppplimit}. Again, by conditioning on $Q_\ell(a)\vee Q_r(b)$ and using \eqref{eq:lawmaxqlqr}, we find 
		\be \ba
		\mathbb{P}(M(a,&b)\cap E(a,b,c,d))\\
		={}&\mathbb{P}(E(a,b,c,d))\P{Q_\ell(a)\vee Q_r(b)\leq c}\\
		&+\int_c^d\P{E(a,b,x,d)}(\alpha-1)x^{-\alpha}(g(0,a)+g(b,1))\exp\{-(g(0,a)+g(b,1))x^{-(\alpha-1)}\}\d x.
		\ea\ee 
		Using \eqref{eq:lawqab}and \eqref{eq:lawmaxqlqr}, the first term on the right-hand-side equals 
		\be \ba 
		(\exp\{&-g(a,b)d^{-(\alpha-1)}\}-\exp\{-g(a,b)c^{-(\alpha-1)}\})\exp\{-(g(0,a)+g(b,1))c^{-(\alpha-1)}\}\\
		&=\exp\{-g(a,b)d^{-(\alpha-1)}-(g(0,a)+g(b,1))c^{-(\alpha-1)}\}-\exp\{-g(0,1)c^{-(\alpha-1)}\},
		\ea \ee 
		and the second term equals 
		\be\ba
		\exp\{&-g(a,b)d^{-(\alpha-1)}\}(\exp\{-(g(0,a)+g(b,1))d^{-(\alpha-1)}\}-\exp\{-(g(0,a)+g(b,1))c^{-(\alpha-1)}\})\\
		&-\int_c^d(\alpha-1)x^{-\alpha}(g(0,a)+g(b,1))\exp\{-g(0,1)x^{-(\alpha-1)}\}\d x\\
		=&{}\exp\{-g(0,1)d^{-(\alpha-1)}\}-\exp\{-g(a,b)d^{-(\alpha-1)}-(g(0,a)+g(b,1))c^{-(\alpha-1)}\}\\
		&-\Big(1-\frac{g(a,b)}{g(0,1)}\Big)(\exp\{-g(0,1)d^{-(\alpha-1)}\}-\exp\{-g(0,1)c^{-(\alpha-1)}\}),
		\ea\ee 
		which, when combined, yields as $n$ tends to infinity,
		\be \ba
		\mathbb{P}\Big((I_n/n,\max_{\inn}\zni/u_n)\in A\Big)&\to\frac{g(a,b)}{g(0,1)}(\exp\{-g(0,1)d^{-(\alpha-1)}\}-\exp\{-g(0,1)c^{-(\alpha-1)}\})\\
		&=\P{I\in(a,b)}\mathbb{P}\Big(\max_{(t,f)\in\Pi}f(t^{-1/\theta_m}-1)\in(c,d)\Big),
		\ea\ee 
		where the final step regarding the law of the maximum of the PPP, a Fr\'echet distribution with shape parameter $g(0,1)$, follows from a similar argument as in \eqref{eq:lawpiprime}. As before, redefining the events in \eqref{eq:mevents} and \eqref{eq:eevents} accordingly and using the same steps yields the joint convergence of
		$(I_n/n,\max_{\inn}\zni/n)$ when $\alpha\in(1,2)$, which concludes the proof of Theorem \ref{Thrm:maxdegree}.
	\end{proof}
 
	\bibliographystyle{abbrv}
\bibliography{transferbib}	
\end{document}